\documentclass[11pt, reqno]{amsart}
\usepackage{etex}
\usepackage{textcmds} 
\usepackage{amssymb, amsfonts, amstext, verbatim, amsthm, mathrsfs,bm}
\usepackage[centertags]{amsmath}
\usepackage{microtype}
\usepackage[all]{xy}
\usepackage[modulo]{lineno}
\usepackage[usenames]{color}
\usepackage{enumitem}
\usepackage{xspace}
\usepackage{rotating}
\usepackage[margin=3.5cm]{geometry}
\usepackage{dsfont}
\usepackage{bm}
\usepackage{color}
\usepackage{subfigure}
\usepackage{array}
\usepackage[all]{xy}
\usepackage{euscript}
\usepackage[T1]{fontenc}
\usepackage{mathbbol}
\usepackage{nccmath}
\usepackage{marginnote}
\usepackage{stmaryrd}
\usepackage{youngtab}
\usepackage{tikz,pgfplots}
\usepackage{blkarray}
\usepackage{ulem}
\usepackage{mathtools}
\usepackage{tikzsymbols}
\usepackage{enumitem}
\usetikzlibrary{decorations.markings,arrows}
\usepackage{float}
\usepackage{wasysym}
\usepackage{listings}
\usepackage[abs]{overpic} %% For help labeling figures

\usepackage{graphics,graphicx}  %% comment for "draft" mode w/o pictures
\let\oldtocsection=\tocsection
\let\oldtocsubsection=\tocsubsection

\renewcommand{\tocsection}[2]{\hspace{0em}\oldtocsection{#1}{#2}}
\renewcommand{\tocsubsection}[2]{\hspace{1em}\oldtocsubsection{#1}{#2}}

\usepackage[colorlinks=true,linkcolor=black,citecolor=blue,urlcolor=blue,citebordercolor={0 0 1},urlbordercolor={0 0 1},linkbordercolor={0 0 1}]{hyperref} %needs to be 
\tikzset{node distance=3cm, auto}

\makeatletter
\def\@secnumfont{\bfseries}
\def\section{\@startsection{section}{1}%
  \z@{.7\linespacing\@plus\linespacing}{.5\linespacing}%
  {\normalfont\Large\bfseries}}
\def\subsection{\@startsection{subsection}{2}%
  \z@{.5\linespacing\@plus.7\linespacing}{-.5em}%
  {\normalfont\large\bfseries}}
  \def\subsubsection{\@startsection{subsubsection}{3}%
  \z@{.5\linespacing\@plus.7\linespacing}{-.5em}%
  {\normalfont\bfseries}}
\makeatother

\numberwithin{figure}{subsection}  % WAS PREVIOUSLY section

\numberwithin{equation}{subsection}  % WAS PREVIOUSLY section

\numberwithin{table}{subsection}  % WAS PREVIOUSLY section

\makeatletter
\let\c@equation\c@figure
\makeatother

\makeatletter
\let\c@table\c@figure
\makeatother

\makeatletter
\let\c@algorithm\c@figure
\makeatother

\newtheorem{thm}{Theorem}[subsection]
\newtheorem{lemma}[thm]{Lemma}
\newtheorem{prop}[thm]{Proposition}
\newtheorem{cor}[thm]{Corollary}
\newtheorem{conjecture}[thm]{Conjecture}
\newtheorem{definition}[thm]{Definition}

\theoremstyle{definition}

\newtheorem{example}[thm]{Example}

\newtheorem{remark}[thm]{Remark}
\newtheorem{rmk}[thm]{Remark}

\newcommand{\N}{{\mathbb{N}}}

\newcommand{\Tt}{\mathcal{T}}

\newenvironment{itemlist}
   { \begin{list} {$\bullet$}
         { \setlength{\topsep}{.5ex}  \setlength{\itemsep}{.5ex} \setlength{\leftmargin}{2.5ex} } }
   { \end{list} }
   
   \newcommand{\bx}{{\bf{x}}}

\newcommand{\bw}{{\bf{w}}}

\newcommand{\bE}{{\bf{E}}}

\newcommand{\ov}{\overline}

\newcommand{\be}{{\beta}}

\newcommand{\Om}{{\Omega}}

\newcommand{\la}{{\lambda}}

\newcommand{\p}{{\partial}}

\newcommand{\Z}{\mathbb{Z}}

\newcommand{\R}{\mathbb{R}}

\newcommand{\C}{\mathbb{C}}

\newcommand{\eps}{\varepsilon}

\newcommand{\calL}{\mathcal{L}}

\newcommand{\e}{\eps}

\newcommand{\Lam}{\Lambda}

\newcommand{\acc}{\mathrm{acc}}
\newcommand{\sembeds}{\stackrel{s}{\hookrightarrow}}
\newcommand{\vol}{\operatorname{vol}}

\makeatletter
\newcommand{\dashover}[2][\mathop]{#1{\mathpalette\df@over{{\dashfill}{#2}}}}
\newcommand{\fillover}[2][\mathop]{#1{\mathpalette\df@over{{\solidfill}{#2}}}}
\newcommand{\df@over}[2]{\df@@over#1#2}
\newcommand\df@@over[3]{%
  \vbox{
    \offinterlineskip
    \ialign{##\cr
      #2{#1}\cr
      \noalign{\kern1pt}
      $\m@th#1#3$\cr
    }
  }%
}
\newcommand{\dashfill}[1]{%
  \kern-.5pt
  \xleaders\hbox{\kern.5pt\vrule height.4pt width \dash@width{#1}\kern.5pt}\hfill
  \kern-.5pt
}
\newcommand{\dash@width}[1]{%
  \ifx#1\displaystyle
    2pt
  \else
    \ifx#1\textstyle
      1.5pt
    \else
      \ifx#1\scriptstyle
        1.25pt
      \else
        \ifx#1\scriptscriptstyle
          1pt
        \fi
      \fi
    \fi
  \fi
}
\newcommand{\solidfill}[1]{\leaders\hrule\hfill}
\makeatother

\title{Four-periodic infinite staircases for four-dimensional polydisks}
\date{\today}

\author{Caden Farley}
\address{CF: Rice University}
\email{ctf4@rice.edu}

\author{Tara S. Holm}
\address{TH: Cornell University}
\email{tara.holm@cornell.edu}

\author{Nicki Magill}
\address{NM: Cornell University}
\email{nm627@cornell.edu}

\author{Jemma Schroder}
\address{JS: Massachusetts Institute of Technology}
\email{jemma@mit.edu}

\author{Zichen Wang}
\address{ZW: Cornell University}
\email{zw336@cornell.edu}

\author{Morgan Weiler}
\address{MW: Cornell University}
\email{morgan.weiler@cornell.edu}

\author{Elizaveta Zabelina}
\address{EZ: Cornell University}
\email{ez283@cornell.edu}

\begin{document}

\begin{abstract} 
The ellipsoid embedding function of a symplectic four-manifold measures the amount by which its symplectic form must be scaled in order for it to admit an embedding of an ellipsoid of varying eccentricity. This function generalizes the Gromov width and ball packing numbers. In the one continuous family of symplectic four-manifolds that has been analyzed, one-point blowups of the complex projective plane, there is an open dense set of symplectic forms whose ellipsoid embedding functions are completely described by finitely many obstructions, while there is simultaneously a Cantor set of symplectic forms for which an infinite number of obstructions are needed. In
the latter case, we say that the embedding function has an infinite staircase. In this paper we identify a new infinite staircase when the target is a four-dimensional polydisk, extending a countable family identified by Usher in 2019. %We use almost toric fibrations to compute key upper bounds and outline an analogy with the case of the one-point blowup of the complex projective plane.
Our work computes the function on infinitely many intervals and thereby indicates a method of proof for a conjecture of Usher. %We further describe the continued fractions of ellipsoid eccentricities which provide key lower bounds on the ellipsoid embedding function, and we explain Python code for efficiently exploring the space of symplectic embeddings.
\end{abstract}

\maketitle

\tableofcontents

\section{Introduction}\label{sec:intro}

A \textbf{symplectic form} on a $2n$-dimensional smooth manifold $X$ is a differential 2-form satisfying:
\begin{itemize}
    \item $d\omega=0$, i.e., $\omega$ is closed, and
    \item $\omega^n\neq0$, i.e., $\omega$ is nondegenerate.
\end{itemize}
A symplectic form can be thought of as a skew-symmetric version of a Riemannian metric, providing area rather than length measurement. Symplectic geometry forms the mathematical framework for classical mechanics and is a go-between from Riemannian to complex geometry.

The \textbf{volume} $\vol(X)$ of a symplectic manifold is the quantity $\int_X\omega^n$. We say a smooth embedding $\varphi:(X,\omega)\to(X',\omega')$ is \textbf{symplectic} if $\varphi^*(\omega')=\omega$, and we denote symplectic embedding by
\[
\varphi:(X,\omega)\sembeds(X',\omega'),
\]
or $ X\sembeds X'$ when the symplectic form is clear from context and we are not emphasizing the specific embedding $\varphi$.

Let $(X,\omega)$ be a four-dimensional symplectic manifold. Its {\bf ellipsoid embedding function}\footnote{It is sometimes also called the \textbf{embedding capacity function} or \textbf{capacity function}.} is
\begin{equation}\label{eqn:cXdefn}
    c_X(z):=\inf\left\{\lambda\ \Big|\ (E(1,z),\omega_0)\sembeds (X,\lambda\omega)\right\},
\end{equation}
where $z\in\R_{>0}$,  $\lambda X: = (X,\lambda\omega)$ is $X$ with the symplectic form scaled, the \textbf{ellipsoid} $E(c,d)\subset \mathbb{C}^2$ is the set
$$
E(c,d)=\left\{(\zeta_1,\zeta_2)\in\mathbb{C}^2 \ \big|\  \pi \left(  \frac{|\zeta_1|^2}{c}+\frac{|\zeta_2|^2}{d} \right)<1 \right\},
$$
and $\omega_0$ is the standard symplectic form $dx_1\wedge dy_1+dx_2\wedge dy_2$ on $\C^2$. Note that the associated volume form is twice the standard volume form on $\R^4$, thus $\vol(E(c,d))=cd$. There is a symmetry that allows us to reduce to $z\geq1$. Namely, for $0<z<1$ we have $c_X(z)=zc_X(1/z)$, because $\omega_0$ restricted to $E(1,z)$ equals $z\omega_0$ restricted to $E(1/z,1)$ under the diffeomorphism $(\zeta_1,\zeta_2)\mapsto(\zeta_1/\sqrt{z},\zeta_2/\sqrt{z})$. Therefore, from now on we restrict the domain of $c_X(z)$ to $\R_{\geq1}$.

The ellipsoid embedding function generalizes the Gromov width\footnote{The \textbf{Gromov width} of a symplectic manifold is $\sup\,\{r\,|\,E(r,r)\sembeds(X,\omega)\}$, or the largest ball that embeds into $(X,\omega)$.} via
\[
c_{Gr}(X,\omega)=\frac{1}{c_X(1)}
\]
and the fraction of the volume of $X$ that can be filled by $n\in\Z_{\geq1}$ equal balls can, by \cite[Thm.~1.1]{Mc0}, be computed from $c_X$ via
\[
\frac{n}{c_X(n)^2\vol(X)}.
\]

For a class of targets $(X,\omega)$ called ``finite type convex toric domains'' (see \S\ref{ssec:toric}) which includes the polydisks that we study, the ellipsoid embedding function satisfies several key properties.

\begin{prop}[{\cite[p. 4, Prop.~2.1]{cghmp}}]\label{prop:cXprops}
Let $(X,\omega)$ be a finite type convex toric domain.  The ellipsoid embedding function $c_X(z)$ satisfies the following properties.

    \begin{enumerate}[label=(\roman*)]
    %$\nbsp$
        \item $c_X(z)\geq\sqrt{\frac{z}{\vol(X)}}$;
        \item $c_X$ is nondecreasing;
        \item $c_X$ is {\bf sublinear}: for all $t\geq1$, we have $c_X(tz)\leq tc_X(z)$;
        \item $c_X(z)$ is continuous (in $z$); 
        \item $c_X(z)$ is equal to the volume curve for sufficiently large values of $z$; \label{it:stab} and
        \item $c_X(z)$ is piecewise linear, when not equal to the volume curve and not at the limit of singular points. 
    \end{enumerate}
\end{prop}

We say $c_X$ or $X$ has an \textbf{infinite staircase} if it is nonsmooth at infinitely many points. An \textbf{outer corner} is a nonsmooth point near which the function is concave while an \textbf{inner corner} is is one near which the function is convex.
By Proposition \ref{prop:cXprops} \ref{it:stab}, the set of nonsmooth points is bounded. By \cite[Thm.~1.13]{cghmp} (see Theorem \ref{thm: acc} for a statement in our case), the nonsmooth points of $c_X$ have a unique finite limit point called the \textbf{accumulation point}, whose $z$-coordinate we denote by $\acc(X)$. (By abuse of notation, we also refer to this $z$-coordinate as the ``accumulation point.'') 
%The limit of the nonsmooth points of $c_X$ is the \textbf{accumulation point} $\acc(X)$ of the infinite staircase, which is unique and finite by \cite[Thm.~1.13]{cghmp}. 
We say an infinite staircase is \textbf{ascending} if the nonsmooth points accumulate from the left and \textbf{descending} if the nonsmooth points accumulate from the right. 
These concepts are illustrated in Figure~\ref{fig:corners}.
In this paper, we will establish the existence of an ascending staircase.

\begin{center}
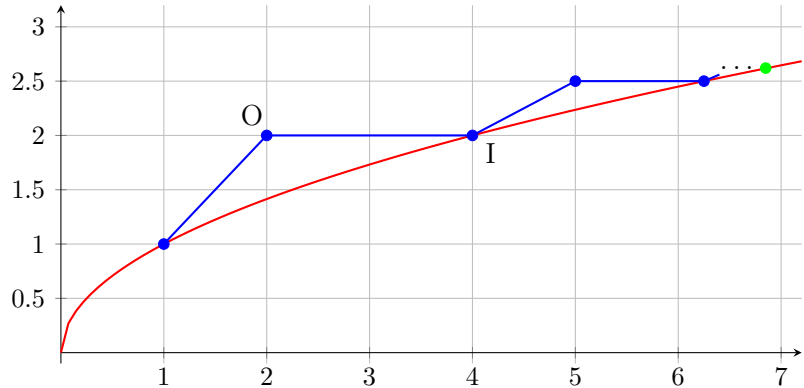
\begin{figure}[H]
\begin{tikzpicture}
\begin{axis}[
	axis lines = middle,
	xtick = {1,2,...,7},
	ytick = {0.5,1,...,3},
	tick label style = {font=\small},
	xmin=0,
	xmax=7.2,
	ymin=-0.1,
	ymax=3.2,
	grid=major,
	width=4.5in,
	height=2.5in]
\addplot [red, thick,
	domain = 0:7.2,
	samples = 100
](x,{sqrt(x)});
\addplot [blue, only marks,  ultra thin, mark=*] coordinates{(1,1) (2,2) (4,2) (5,5/2) (25/4,5/2)};
\addplot [green, only marks,  ultra thin, mark=*] coordinates{(6.85,2.62)};
\addplot [blue, thick,
	domain = 1:2,
	samples = 100
](x,x);
\addplot [blue, thick,
	domain = 2:4,
	samples = 100
](x,2);
\addplot [blue, thick,
	domain = 4:5,
	samples = 100
](x,1/2*x);
\addplot [blue, thick,
	domain = 5:25/4,
	samples = 100
](x,5/2 );
\addplot [blue, thick,
	domain = 25/4:6.4,
	samples = 100
](x,2/5*x);
\node[yshift=1.3in,xshift=1in] {O};
\node[yshift=1.1in,xshift=2.25in] {I};
\node[yshift=1.55in,xshift=3.55in] {\dots};
\end{axis}
\end{tikzpicture}
\caption{ 
In blue, the graph of the embedding capacity function for a ball $X=B^4(1)$ is shown  on the domain indicated.  
The graph in red is the volume lower bound established
in Proposition~\ref{prop:cXprops}(i).  The point marked O is an outer
corner and the point marked I is an inner corner.  This target has
an ascending infinite staircase, first identified by McDuff
and Schlenk \cite{ball} and called the Fibonacci staircase in the literature.
The green point is the accumulation point.
}
\label{fig:corners} 
\end{figure}
\end{center}

\subsection{Summary of results} \label{ssec:summary}
Our target of choice will be the \textbf{polydisk}, defined for $\beta\in\R_{\geq1}$ by  
$$
P(1,\beta):=\left\{(\zeta_1,\zeta_2)\in\C^2 \;\middle|\; \pi|\zeta_1|^2\leq1,\;\pi|\zeta_2|^2\leq \beta\right\}.
$$
We denote by $c_\beta$ its ellipsoid embedding function $c_{P(1,\beta)}$. The polydisk is a finite type convex toric domain, so $c_\beta$ satisfies Proposition~\ref{prop:cXprops}.
In this case there are two functions
\begin{align*}
\acc(\beta)&:=\acc(P(1,\beta)):[1,\infty)\to\left[3+2\sqrt{2},\infty\right)
\\\vol(\beta)&:=\sqrt{\frac{\acc(\beta)}{\vol(P(1,\beta))}}:[1,\infty)\to\left[1+\frac{\sqrt{2}}{2},1\right)
\end{align*}
where if $c_\beta$ has an infinite staircase, its accumulation point has coordinates $(\acc(\beta),\vol(\beta))$ by \cite[Thm.~1.13]{cghmp}; see Lemma \ref{lem:accvol}.

The first ellipsoid embedding function was computed for $X=B^4:=E(1,1)$ by McDuff and Schlenk in \cite{ball}. They found that its graph contained an infinite staircase whose inner and outer corners were derived from the Fibonacci numbers. Further work by Frenkel and M\"uller in \cite{frenkelmuller} exhibited a similar infinite staircase in $c_1$ governed by the Pell numbers, while on the other hand work of Cristofaro-Gardiner, Frenkel, and Schlenk showed that the property of having an infinite staircase is not universal: the functions $c_n$ for $n\in\Z_{>1}$ do not contain infinite staircases \cite{integralpolydisks}. More generally, a conjecture of Cristofaro-Gardiner, Holm, Mandini, and Pires in \cite{cghmp} suggests that $c_\beta$ should not contain an infinite staircase for any rational $\beta$.

However, work by Usher \cite{usher} suggested that the set of irrational $\beta$ for which $c_\beta$ has an infinite staircase might be quite rich: he identified a bi-infinite family $L_{n,k}\in\R_{\geq1}$ for which $c_{L_{n,k}}$ have infinite staircases.\footnote{In this paper as well as in the closely related papers \cite{ICERM}, \cite{symm}, and \cite{MMW} we use $k$ to denote the staircase step and $i$ to denote the image of a step, staircase, or $b$ value under a symmetry analogous to Usher's Brahmagupta moves (\cite[Def.~2.10]{usher}). Our notation differs from Usher's in that what the $i$ and $k$ indices denote are switched. We generally stick to our convention throughout but use Usher's convention here.} Of particular interest to us are his
\[
L_{n,0}:=\sqrt{n^2-1},\quad n\geq2
\]
which generate the $k>0$ values of $L$ with infinite staircases (see \S\ref{ssec:B}). See Figure \ref{fig:polydisk-accum} for a visualization of these results via a plot of the relevant accumulation points.

\begin{center}
\begin{figure}[H]
\begin{tikzpicture}
\begin{axis}[
	axis lines = middle,
	xtick = {5,6,...,13},
	ytick = {1,1.25,1.5,1.75},
	tick label style = {font=\small},
	xlabel = $z$,
	ylabel = $y$,
	xlabel style = {below right},
	ylabel style = {above left},
	xmin=4.9,
	xmax=13.2,
	ymin=0.9,
	ymax=1.8,
	grid=major,
	width=4.75in,
	height=2.5in]
\addplot [red, thick,
	domain = 1:6,
	samples = 100
]({ (((x+1)*(x+1)/((x))-1)+sqrt( ((x+1)*(x+1)/((x))-1)* ((x+1)*(x+1)/((x))-1) -1 ))},{sqrt(( (((x+1)*(x+1)/((x))-1)+sqrt( ((x+1)*(x+1)/((x))-1)* ((x+1)*(x+1)/((x))-1) -1 )))/(2*x))});
\addplot [blue, only marks,  ultra thin, mark=*] coordinates{(6.46,1.366) (8.243,1.2071) (10.164,1.1455) (12.124,1.1124)};
\addplot [red, only marks,  ultra thin, mark=*] coordinates{(5.828,1.7071)};
\addplot [black, only marks, very thick, mark=x] coordinates{(6.854,1.309)  (8.55,1.1937) (10.404,1.1404) (12.319,1.1099)};
\addplot [green, only marks, ultra thin, mark=*] coordinates{(8.161,1.211) (10.123,1.1464) (12.099,1.1127)};
\end{axis}
\end{tikzpicture}
\caption{ 
This figure shows the parameterized curve $(\acc(\beta),\vol(\beta))$ 
in red. The point on the curve at $\beta$ represents a point at 
which an infinite staircase for $c_\beta$ must accumulate, if it
exists. The red dot is the accumulation point of the Pell 
stairs of Frenkel-M\"uller; the blue dots are the $L_{n,0}$
staircases of Usher; and the black $\bm{\times}$s indicate 
values of $\beta$ without infinite staircases, proved by
Cristofaro-Gardiner--Frenkel--Schlenk. The accumulation 
points of the new infinite staircases of 
Theorem~\ref{thm:main} and Conjecture \ref{thm:generalb} are indicated by green dots. 
}
\label{fig:polydisk-accum} 
\end{figure}
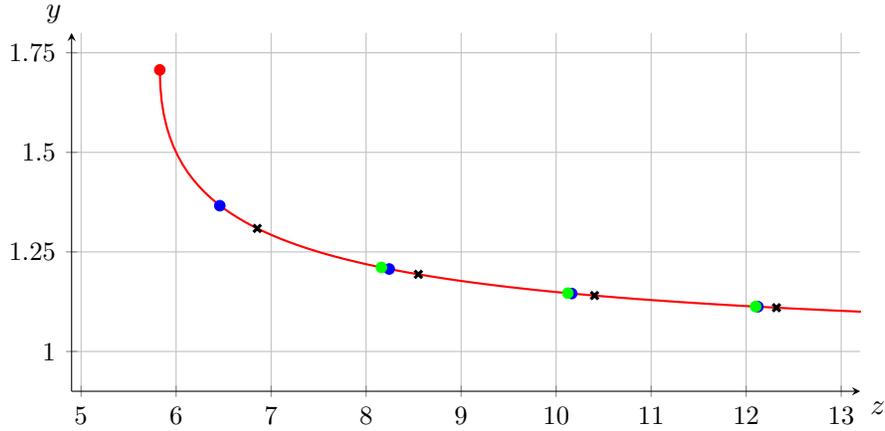
\end{center}

Work by Bertozzi, Holm, Maw, McDuff, Mwakyoma, Pires, and Weiler \cite{ICERM} and by Magill and McDuff \cite{symm} proved an analogous result for the target
\[
H_{b}:=\left\{(\zeta_1,\zeta_2)\in\C^2 \;\middle|\; \pi|\zeta_1|^2+\pi|\zeta_2|^2\leq1,\; \pi|\zeta_2|^2\leq1-b\right\}.
\]
(The region $H_b$ is equivalent in terms of ellipsoid embeddings, see \S\ref{sssec:toricsymp}, to $\C P^2\#\overline{\C P}^2$, thus in the literature on infinite staircases it is also called the \textbf{Hirzebruch surface}.) 
They showed that there are two bi-infinite families $b_{n,i,\delta}$, with $n,i\in\Z_{\geq0}$ and $\delta\in\{0,1\}$, for which $c_{H_b}$ has an ascending infinite staircase. Moreover, each ascending infinite staircase comes paired with a descending infinite staircase.
%Moreover, for each $b_{n,i,\delta,\ell}$ there is a corresponding $b_{n,i,\delta,u}$ with a descending infinite staircase, with
% \[
% \cdots<b_{n,0,0,\ell}<b_{n,0,0,u}<b_{n+1,0,0,\ell}<\cdots.
% \]

One feature that all infinite staircases described so far appear to have in common is that their outer corners are at $z$-values whose continued fractions grow by a predictable pattern of adding pairs of integers. Recall that real numbers can be described by their \textbf{continued fractions}, e.g.
    \[
    [m,n,\ell]=m+\frac{1}{n+\frac{1}{\ell}},
    \]
    with repeated parts denoted by
    \[
    \left[m,\{n,\ell\}^k\right]=[m,\hspace{-3pt}\underbrace{n,\ell}_{k\text{ times}}\hspace{-3pt}], \quad \left[m,\{n,\ell\}^\infty\right]=[m,n,\ell,n,\ell,n,\ell,\dots].
    \]
    Every positive real number has a continued fraction with all entries positive integers; rational numbers have finite continued fractions, quadratic irrational numbers (irrational roots of quadratic equations with rational coefficients) have infinite periodic continued fractions, and non-quadratic irrational numbers have infinite non-periodic continued fractions. We will abuse notation and occasionally allow 
   the last entry in a finite continued fraction to bea real number, e.g. in the proof of Theorem \ref{thm:main} (ii) in \S\ref{ssec:outer}. Doing so is just a matter of notation, because if $a_i\in\Z_{>0}$ and $z\in\R$ has continued fraction $[b_0,b_1,\dots]$ then $[a_0,\dots,a_n,z]=[a_0,\dots,a_n,b_0,b_1,\dots]$. Allowing the last number to be real can be helpful when trying to understand the algebraic relationships among the continued fraction's rational approximations, as in Lemma \ref{lem: contfrac}.

Recall that an \textbf{outer corner} of $c_X$ is a nonsmooth point near which $c_X$ is convex; see Figure \ref{fig:main}. The outer corners of the Fibonacci stairs of McDuff-Schlenk have continued fractions
\[
[2], [5], [6,1,5,2], [6,1,5,1,4], [6,1,5,1,5,2], [6,1,5,1,5,1,4],\dots
\]
The accumulation points of all infinite staircases discussed so far are quadratic irrationals with two-periodic continued fractions. We say an infinite staircase is \textbf{$2m$-periodic} if the continued fraction of the $k^\text{th}$ outer corner equals that of the $(k-2)^\text{th}$ outer corner with a fixed length $2m$ sequence of integers added after a fixed sequence of integers at the beginning. For example, in the sequence above, the Fibonacci stairs are 2-periodic with a pair $1,5$ inserted recursively after the $6.$

In \cite{MMW}, Magill, McDuff, and Weiler showed that between each of the pairs of adjacent ascending and descending infinite staircases studied in \cite{ICERM, symm} there is a further Cantor set of values of $b$ for which $c_{H_{b}}$ has an infinite staircase. These include infinite staircases whose outer corners and accumulation points appear to have higher-periodic continued fractions, as well as infinite staircases whose accumulation points may not be quadratic irrational. They were obtained by generalizing the procedure to construct an infinite staircase whose outer corners have four-periodic continued fractions accumulating to $[\{7,5,3,1\}^\infty]$ 
% for
% \[
% H\left(\frac{2(22 - \sqrt{165})}{29}\right)
% \]
from the descending staircase accumulating to $[7,\{5,1\}^\infty]$ and the ascending staircase accumulating to $[\{7,3\}^\infty]$.
% \[
% b_{0,0,0,u}=\frac{3(7 + \sqrt{5})}{44}, \text{ and } b_{1,0,0,\ell}=\frac{11 - \sqrt{21}}{10}.
% \]

\begin{figure}[h]
    \subfigure[]{
         \centering
         \includegraphics[width=.48\textwidth]{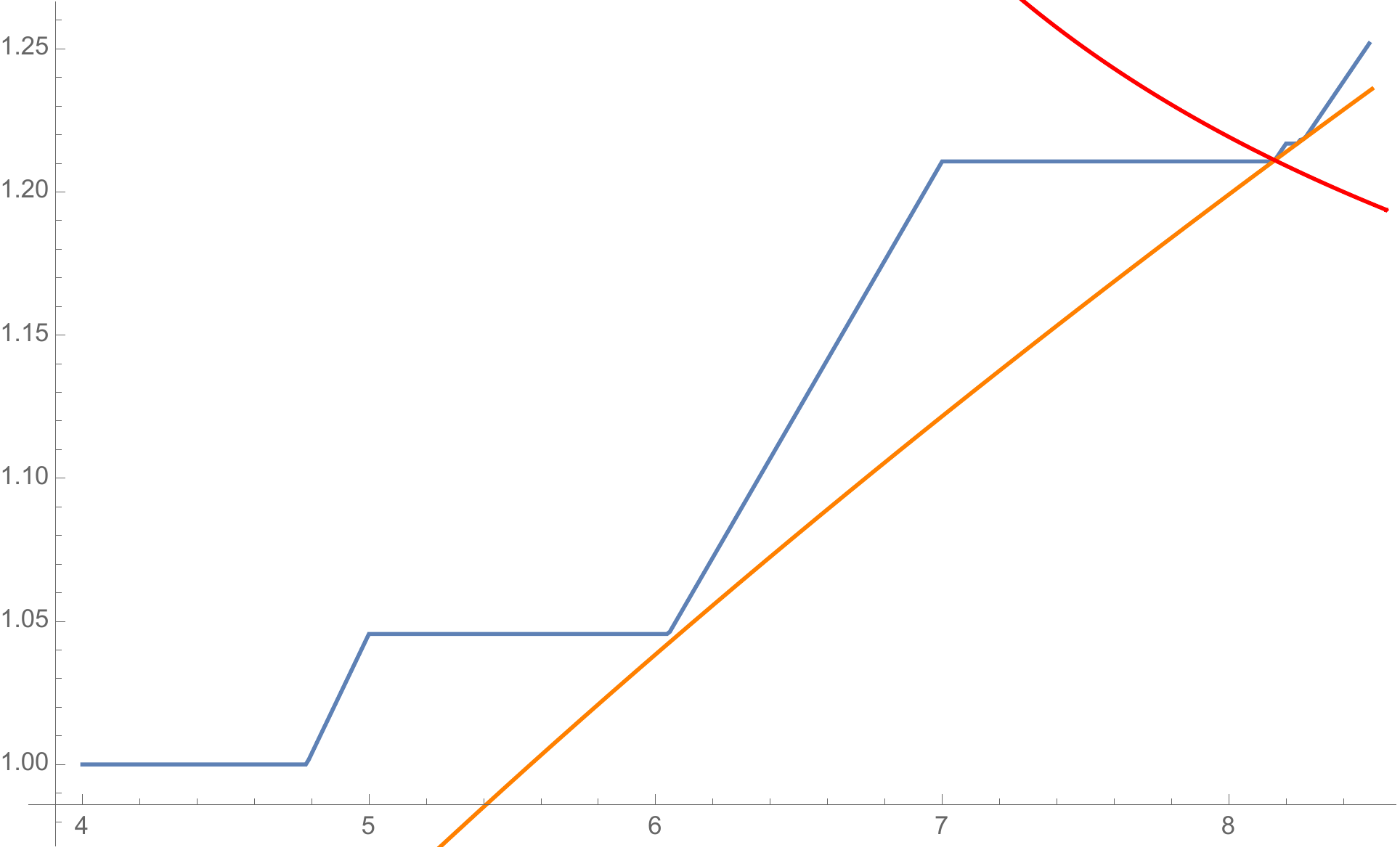}
     }\hfil
    \subfigure[]{
         \centering
         \includegraphics[width=.48\textwidth]{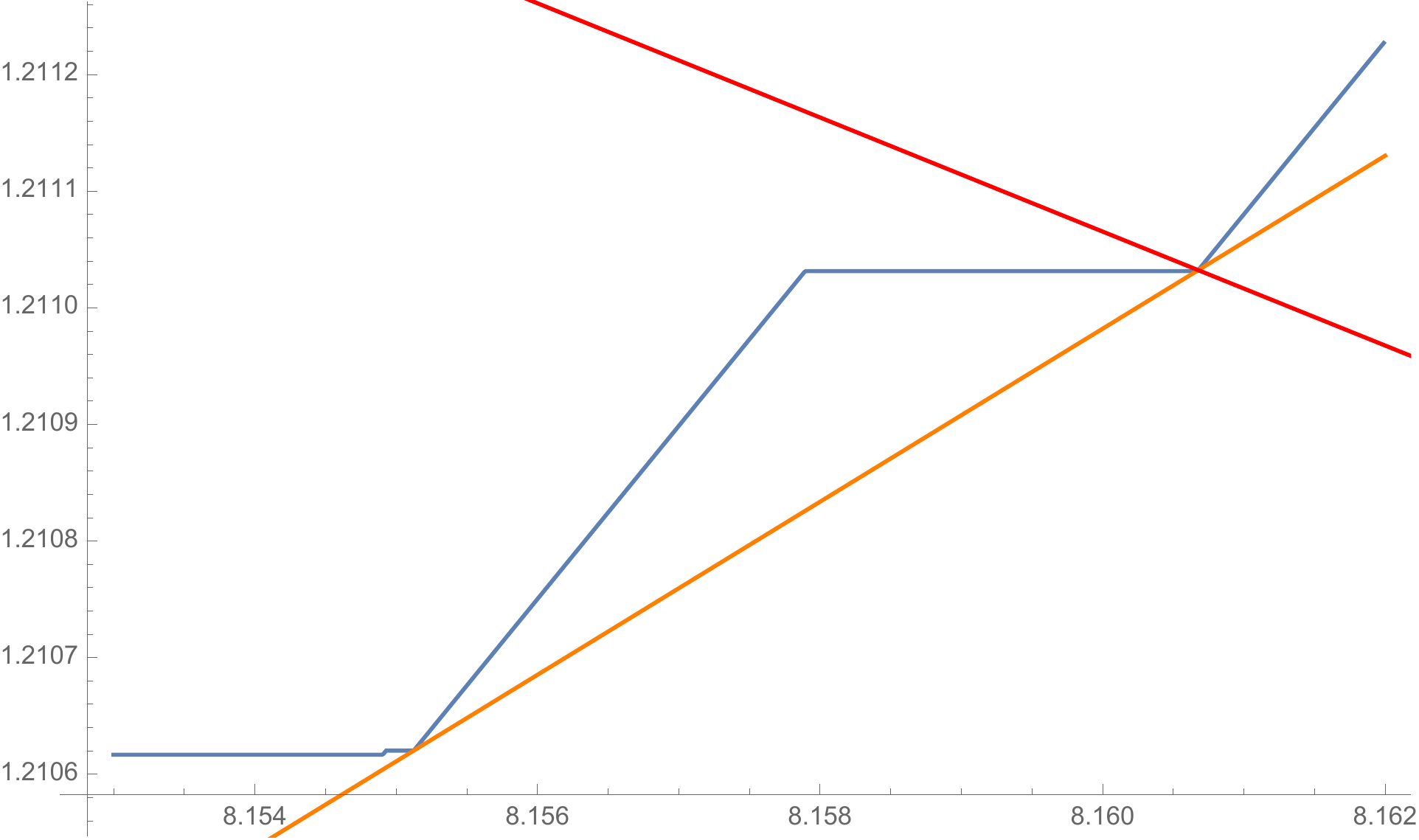}
         }
     \hfil
     \caption{This figure depicts the infinite staircase $c_\beta$ of Theorem \ref{thm:main}. In both figures, with $\beta$ as in Theorem \ref{thm:main}, the orange curve is $\vol_\beta(z)$ and $c_\beta$ is in blue. The accumulation point curve $(\acc(\beta),\vol(\beta))$ is in red -- for this curve, $\beta$ varies. Thus the accumulation point of $c_\beta$ occurs at the intersection of these three curves. In (b), we have zoomed in; the obstructions from $\bE_0, \hat\bE_1$, and $\bE_2$ are visible.  See sections \S\ref{ssec:ECdefs} and \S\ref{sec:proofs} for these definitions.}
    \label{fig:mainplot}
\end{figure}

We predict a very close correspondence between the cases of the polydisk and $H_{b}$ in Conjecture \ref{conj:Pfractal}. Our main theorem provides evidence for this conjecture.
\begin{thm}\label{thm:main}
    Set
    \[
    \beta=\frac{6+5\sqrt{30}}{12}.
    \]
    \begin{enumerate}[label=(\roman*)]
    \item The function $c_\beta$ has an infinite staircase.
    \item It is four-periodic, with $\acc(\beta)=[\{8,6,4,2\}^\infty]$.
    \end{enumerate}
\end{thm}
See Figure \ref{fig:mainplot} for a visualization. Detail on the location of the accumulation point of the infinite staircase of Theorem \ref{thm:main} is given in Figure \ref{fig:polydisk-accum-zoom}.

Of note is the fact that we prove Theorem \ref{thm:main} in \S\ref{sec:proofs} by computing it on infinitely many intervals. In \S\ref{sec:proofs} we also outline a procedure for computing $c_\beta$ on the entire interval $[1,\acc(\beta)]$ containing the infinite staircase. This would prove an analogue of \cite[Conj.~4.23]{usher}, with the role of his $A$ classes being played by our $\bE$ classes and his $\hat A$ classes replaced by our $\hat\bE$ classes: see the preamble to \S\ref{sec:proofs} for the definitions $\bE$ and $\hat\bE$, and see \S\ref{ssec:Uconj} for further discussion of Usher's conjecture.

\begin{center}
\begin{figure}[H]
\begin{tikzpicture}
\begin{axis}[
	axis lines = middle,
	xtick = {6,7,...,9},
	ytick = {1,1.25,1.5,1.75},
	tick label style = {font=\small},
	xlabel = $z$,
	ylabel = $y$,
	xlabel style = {below right},
	ylabel style = {above left},
	xmin=5.7,
	xmax=9.2,
	ymin=0.9,
	ymax=1.8,
	grid=major,
	width=4.75in,
	height=2.5in]
\addplot [red, thick,
	domain = 1:6,
	samples = 100
]({ (((x+1)*(x+1)/((x))-1)+sqrt( ((x+1)*(x+1)/((x))-1)* ((x+1)*(x+1)/((x))-1) -1 ))},{sqrt(( (((x+1)*(x+1)/((x))-1)+sqrt( ((x+1)*(x+1)/((x))-1)* ((x+1)*(x+1)/((x))-1) -1 )))/(2*x))});
\addplot [blue, only marks,  thick, mark=*] coordinates{(6.46,1.366) (8.243,1.2071) };
\addplot [red, only marks,  thick, mark=*] coordinates{(5.828,1.7071)};
\addplot [black, only marks, very thick, mark=x] coordinates{(6.854,1.309)  (8.55,1.1937) };
\addplot [green, only marks, thick, mark=*] coordinates{(8.161,1.211)};
\end{axis}
\end{tikzpicture}
\caption{This figure uses the same color scheme as Figure \ref{fig:polydisk-accum}. More detail near the infinite staircase of Theorem \ref{thm:main} is shown. The new staircase's accumulation point is the green dot, while the two blue dots are Usher's staircases with $\beta=L_{2,0}$ and $L_{3,0}$.
}\label{fig:polydisk-accum-zoom} 
\end{figure}
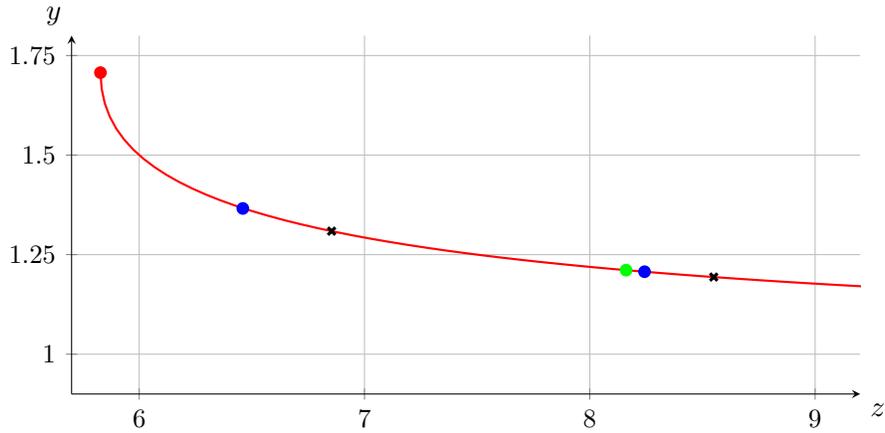
\end{center}

We furthermore expect (from experimental evidence and by combining Conjecture \ref{thm:generalb} with \cite[Thm.~1.1.1]{MMW}) that our result generalizes to all $n\in\Z_{\geq2}$:
\begin{conjecture}\label{thm:generalb} Let $\beta_n$ be of the form
\[
\beta_n=\frac{1}{2}+\frac{(2n+1)\sqrt{n(n^3+2n^2-1)}}{2n(n+1)}
%b_n=\frac{1}{2}+\frac{(2n+5)\sqrt{n^4+10n^3+36n^2+55n+30}}{2(n+2)(n+3)}
\]

\noindent with $n\in\mathbb{Z}_{\geq2}$.

\begin{enumerate}[label=(\roman*)]
    \item The function $c_{\beta_n}$ has an infinite staircase.
    \item It is four-periodic, with $\acc(\beta_n)=[\{2n+4,2n+2,2n,2n-2\}^\infty]$.
\end{enumerate}
\end{conjecture}

\noindent Note that setting $n=2$ in Theorem \ref{thm:generalb} reduces to Theorem \ref{thm:main}.

\subsection{Connections to other targets}\label{ssec:context}

%An ellipsoid embeds into $P(1,\beta)$ precisely when one embeds into a two-fold blowup $\C P^2(b+1)\#\overline{\C P}^2(b)\#\overline{\C P^2}(1)$ by \cite[Thm.~1.4]{cghmp}; see further explanation in \S\ref{ssec:toric}. 
There is a function
\[
\acc_H:[0,1)\to\left[3+2\sqrt{2},\infty\right),
\]
which is analogous to $\acc$ in the following way: if $c_{H_b}$ has an infinite staircase then $\acc_H(b)$ is the $z$-coordinate of its accumulation point. It is 2-1 in general but when restricted to $[1/3,1)$ it is 1-1 with range $\left[3+2\sqrt{2},\infty\right)$. These facts have a similar proof to Lemma \ref{lem:accvol}.

\begin{conjecture}\label{conj:Pfractal} Define a function $f:\R\to\R$ by sending $z$ to the number whose continued fraction is obtained from the continued fraction of $z$ by subtracting one from each entry.

Let $\beta\geq\sqrt{3}.$ The function $c_\beta$ has an infinite staircase if and only if the function $c_{H_b}$ has an infinite staircase, where
\[
b=\acc_H^{-1}\circ f\circ\acc(\beta).
\]
\end{conjecture}

The infinite staircases $c_{L_{n,0}}$ of \cite{usher}, $c_{\beta_n}$ of Conjecture \ref{thm:generalb}, and the fact that $c_n, n\in\Z$ do not contain infinite staircases from \cite{integralpolydisks} all support Conjecture \ref{conj:Pfractal}. Further evidence is explored in 
\cite[\S 3 and Remark~3.1.7]{MPW}.
It is possible to extend Conjecture \ref{conj:Pfractal} to all $\beta$ using the Brahmagupta moves of Usher and their counterparts for the Hirzebruch surface from \cite{symm};  explaining the extension is beyond the scope of this paper.
Because of the similarities between \S\ref{ssec:inner} and \cite{M1}, we expect there is a more direct relationship between $b$ and $\beta$ than via the accumulation point function $\acc$, but this has not yet been discovered.

In Section~\ref{ssec:inner}, to construct the desired embeddings, we followed the sequences of mutations on almost toric fibrations found by Magill in \cite{M1} and \cite{M2}. In \cite{M1}, ATF mutations for elliposid embeddings into $H_b$ were considered. The formulas found for the mutations in Section~\ref{ssec:inner} mirror the formulas found in \cite{M1}. In fact, the formulas in Lemma~\ref{lem:als}~and~\ref{lem:k3} could be easily generalized to mirror the formulas of \cite[Def.~3.8]{M1}. Therefore, we expect a generalization of Prop~\ref{prop:fullacc} similar to \cite[Thm.~1.1]{M1}. This is more evidence for the correspondence between the embedding functions for polydisks and $H_{b}.$

In \cite{M2}, ellipsoid embeddings into a two-fold blow up of $\C P^2$ were considered. In Section \ref{ssec:inner}, we follow the same mutation sequences Magill used to compute some of the inner corners of the function. One new addition in Section~\ref{ssec:inner} is Conjecture~\ref{conj:ic2} that adding one extra mutation to Magill's sequences will compute all the inner corners of the function for $c_\beta$. The work of Casals and Vianna in \cite{CV} and Cristofaro-Gardiner, Holm, Mandini, and Pires in \cite{cghmp} show that almost toric mutations give all embeddings for particular rational convex toric domains. Conjecture~\ref{conj:ic2} would imply a similar statement holds for $c_\beta$ where $\beta=\frac{6+5\sqrt{30}}{12}.$ It would be interesting to see if similar statements hold for $H_b$ and the two fold blow up of $\C P^2$.

\subsection{Outline of the paper}\label{ssec:outline}

We introduce the necessary tools to analyze $c_\beta$ in \S\ref{sec:tools}, prove Theorem \ref{thm:main} in \S\ref{sec:proofs}, and in \S\ref{sec:ideas} we outline future work supported by other experimental evidence discovered in summer 2022. The section \S\ref{sssec:toricsymp} %and \S\ref{sssec:ecs}
requires a graduate-level background in geometry and can be skipped on a first reading.

\subsection*{Acknowledgements}
We would like to thank the Department of Mathematics at Cornell University for hosting us under the 2022 SPUR program. We also thank Michael Usher for pointing out the relationship between the obstructions $\hat{A}_{i,n}$ and mutation of perfect classes, as well as helpful comments regarding \S\ref{sec:ideas}. We thank the anonymous referee for helpful comments. In addition,

\vskip 0.1in

\noindent Caden Farley was supported by Jo Nelson's NSF grant DMS-2104411.

\noindent Tara Holm was supported by NSF grant DMS-2204360.

\noindent Nicki Magill was supported by NSF Graduate Research Grant DGE-1650441.

\noindent Jemma Schroder was supported by the MIT Department of Mathematics.

\noindent Zichen Wang was supported by the Cornell University Department of Mathematics.

\noindent Morgan Weiler was supported by NSF Research Grant DMS-2103245.

\noindent Elizaveta Zabelina was supported by the Nexus Scholars Program in the College of Arts \& Sciences, Cornell University.

\noindent Any opinions, findings and conclusions or recommendations expressed in this material
are those of the authors and do not necessarily reflect the views of the National Science
Foundation.

\section{Tools for obstructing and constructing embeddings}\label{sec:tools}

In this section we define the tools we use to prove Theorem \ref{thm:main}.

\subsection{Embedding functions of toric domains}\label{ssec:toric}

A \textbf{toric domain} $X_{\Omega}$ in $\mathbb{C}^2$ is the preimage of a domain $\Omega\subset\mathbb{R}^2_{\geq0}$ under the map $\mu:\mathbb{C}^2\rightarrow\mathbb{R}^2$ given by

\[
(\zeta_1,\zeta_2)\mapsto(\pi|\zeta_1|^2,\pi|\zeta_2|^2).
\]

\noindent We call the map $\mu$ the \textbf{moment map} and the domain $\Omega$ the \textbf{moment polygon} of $X_{\Omega}$, as they are analogous to the moment maps and moment polygons associated to closed toric symplectic manifolds. We say that a toric domain $X_{\Omega}$ is \textbf{convex} if the domain $\Omega$ is a closed, connected region of $\mathbb{R}^2$ and is convex as a polygon in $\mathbb{R}^2$. As a consequence of the presence of factors of $\pi$ in the expression for the moment map $\mu$, the volume of a toric domain $X_{\Omega}$ coincides with twice the area of its moment polygon $\Omega$. 

When $(X,\omega)=(X_\Omega,\omega_0)$, instead of  \eqref{eqn:cXdefn} we write
\[
c_X(z):=\inf\left\{\lambda\ \Big|\ E(1,z)\sembeds X_{\lambda\Omega}\right\},
\]
dropping the symplectic forms from the notation.

We say that a convex toric domain $X_{\Omega}$ is of \textbf{finite type} if $\Omega$ has only finitely many sides and all of these sides have rational slopes. For these finite type toric domains, the accumulation points of potential infinite staircases can be computed as solutions to an explicit quadratic equation. For details of this result and the following definition, see the paper \cite{cghmp}.

\begin{definition}\label{def: afflength}
Let $L$ be a line segment in $\mathbb{R}^2$. The \textbf{affine length} of $L$ is the length of the image $AT(L)$ of $L$ under a composition of a translation $T$ with a linear transformation $A\in \mathrm{SL}(2,\mathbb{Z})$, where $A$ and $T$ are chosen so that $AT(L)$ lies along the $x$-axis.

If $\Omega$ is a polygon in $\mathbb{R}^2_{\geq0}$ with only finitely many sides each of which has a rational slope, define the \textbf{affine perimeter} of $\Omega$ to be the sum of the affine lengths of its sides, and denote this quantity by $\mathrm{per(\Omega)}$.
\end{definition}

With these definitions, we can now state the following result about the accumulation points of infinite staircases of finite type convex toric domains:

\begin{thm}[{\cite[Thm.~1.13]{cghmp}}]\label{thm: acc}
Let $X_{\Omega}$ be a finite type convex toric domain. If the ellipsoid embedding function $c_{X_{\Omega}}(z)$ has an infinite staircase, then it accumulates at $\mathrm{acc}(\Omega) \geq 1$, a real solution\footnote{The solutions to this equation have product one and are either positive or complex. For the polydisk, there is always a unique real solution larger than one.} to the quadratic equation

\[
z^2-\bigg(\frac{\mathrm{per}(\Omega)^2}{2\cdot \mathrm{area}(\Omega)}-2\bigg)z+1=0.
\]

\noindent In this case, at $\mathrm{acc}(\Omega)$, the ellipsoid embedding function touches the volume curve:

\[
c_{X_{\Omega}}(\mathrm{acc}(\Omega))=\sqrt{\frac{\mathrm{acc}(\Omega)}{2\cdot \mathrm{area}(\Omega)}}.
\]
\end{thm}

In the setting of this paper, $X_{\Omega}$ will be the polydisk $P(1,\beta)$, which has moment polygon $\Omega_\beta$ a rectangle situated at the origin with sides of length $1$ and $\beta$ parallel to the $x$- and $y$-axes. Here, the affine perimeter of $\Omega_\beta$ is the same as its regular perimeter, $\mathrm{per}(\Omega_\beta)=2(\beta+1)$, and the area of $\Omega_\beta$ is $\mathrm{area}(\Omega_\beta)=b$. In this case, the quadratic equation in Theorem \ref{thm: acc} becomes

\begin{equation}\label{eqn:accpteqn}
z^2-\bigg(\frac{2(\beta+1)^2}{\beta}-2\bigg)z+1=0.
\end{equation}

In addition to providing an explicit way to calculate accumulation points, Theorem \ref{thm: acc} also describes a necessary condition for the existence of an infinite staircase for different values of $\beta$. We call the difference $c_{X_{\Omega}}(\mathrm{acc}(\Omega))-\sqrt{\frac{\mathrm{acc}(\Omega)}{2\cdot\mathrm{area}(\Omega)}}\geq0$ the \textbf{staircase obstruction} of $X_{\Omega}$. Theorem \ref{thm: acc} indicates that if the ellipsoid embedding function $c_{X_{\Omega}}(z)$ has an infinite staircase, then the staircase obstruction of $X_{\Omega}$ vanishes. For the case where $X_{\Omega}=P(1,\beta)$, if the staircase obstruction does not vanish for a particular value of $\beta$, we say that this $\beta$-value is \textbf{blocked}, and we conclude that the ellipsoid embedding function $c_\beta(z)$ does not have an infinite staircase.

Finally, because the accumulation point of an infinite staircase is on the volume obstruction, the formula on the right hand side of Proposition \ref{prop:cXprops} (i) specialized to the case of the polydisk will be key throughout; we set the notation
\[
\vol_\beta(z):=\sqrt{\frac{z}{2\beta}}=\sqrt{\frac{z}{2\cdot\mathrm{area}(\Omega_\beta)}}.
\]

We compute the ranges of $\acc$ and $\vol$ to motivate Figures \ref{fig:polydisk-accum} and \ref{fig:polydisk-accum-zoom}.

\begin{lemma}\label{lem:accvol} Setting $\acc(\beta)=\acc(\Omega_\beta)$, we have
    \[
\acc:[1,\infty)\to\left[3+2\sqrt{2},\infty\right)
\]
and $\acc$ is increasing. If we set $\vol(\beta)=\vol_\beta(\acc(\beta))$ then
\[
\vol:[1,\infty)\to\left[1+\frac{\sqrt{2}}{2},1\right)
\]
and $\vol$ is decreasing.
\end{lemma}
\begin{proof}
Solving \eqref{eqn:accpteqn} we obtain
\[
\acc(\beta)=z=\beta+1+\frac{1}{\beta}+\sqrt{\beta^2+2\beta+2+\frac{2}{\beta}+\frac{1}{\beta^2}},
\]
thus
\[
\acc(1)=3+\sqrt{8}=3+2\sqrt{2},
\]
and $\lim_{\beta\to\infty}\acc(\beta)=\infty$ because $\acc(\beta)>\beta$. The function $\acc(\beta)$ is increasing because
\[
\frac{\p}{\p \beta}\left(\beta+1+\frac{1}{\beta}\right)=1-\frac{1}{\beta^2}
\]
and
\[
\frac{\p}{\p \beta}\left(\beta^2+2\beta+2+\frac{2}{\beta}+\frac{1}{\beta^2}\right)=2\beta+2-\frac{2}{\beta^2}-\frac{2}{\beta^3},
\]
which are both positive if $\beta>1$.

Because $\vol_\beta(z)$ has $\beta$ in the denominator, it is decreasing if $z$ is increasing, so $\vol(\beta)$ is decreasing in $\beta$. We compute
\[
\left(1+\frac{\sqrt{2}}{2}\right)^2=1+\sqrt{2}+\frac{1}{2}=\frac{3+2\sqrt{2}}{2}=\sqrt{\frac{\acc(1)}{2}}=\vol(1).
\]
Finally, by the fact that $\vol$ and $\acc$ are continuous and defined on $[1,\infty)$,
\begin{align*}
\left(\lim_{\beta\to\infty}\vol(\beta)\right)^2&=\lim_{\beta\to\infty}\frac{\beta+1+\frac{1}{\beta}+\sqrt{\beta^2+2\beta+2+\frac{2}{\beta}+\frac{1}{\beta^2}}}{2\beta}
\\&=\lim_{\beta\to\infty}\frac{1}{2}+\frac{1}{2\beta}+\frac{1}{2\beta^2}+\sqrt{\frac{1}{4}+\frac{1}{2\beta}+\frac{1}{2\beta^2}+\frac{1}{2\beta^3}+\frac{1}{4\beta^4}}
\\&=1.
\end{align*}
\end{proof}

\subsubsection{Closed toric symplectic manifolds}\label{sssec:toricsymp}

Our methods rely on the fact that ellipsoid embeddings into certain finite type convex toric domain targets are equivalent
to ellipsoid embeddings into certain closed symplectic manifolds, specifically toric blowups of $\C P^2$.
Topologically, {\bf symplectic blowup} is a procedure where an open ball is removed from a manifold, and the resulting 
boundary sphere is collapsed along the Hopf fibration.  This can be achieved in a sympelctic manner if the ball was
symplectically embedded; see \cite[Thm.~7.1.21]{McSal}.
In the special case when the initial manifold $M$ is four-dimensional, the symplectic blowup procedure is equivalent to
the symplectic connected sum $M\#\overline{\C P}^2$; see  \cite[Ex.~7.1.4]{McSal}.
Moreover, when $M$ is a four-dimensional toric symplectic manifold and the blowup respects the action, then
at the level of moment polygons, the toric blowup has the impact of truncating a vertex \cite[Ex.~7.1.15]{McSal}.

Toric symplectic manifolds are classified by their moment polytope, up to equivariant symplectomorphism of the manifold
and up to affine equivalence of the polytopes. Those polytopes which are the moment polytope of some toric symplectic
manifold are called {\bf Delzant polytopes}.
For four-dimensional toric symplectic manifolds, Delzant polygons are those that have edges with rational slope and for
each vertex, the two primitive vectors pointing in the directions of the edges form a $\Z$-basis of the integer lattice
in $\R^2$.  Because we work up to affine equivalence of Delzant polytopes, we may assume that a Delzant polygon
has a vertex at the origin, that the edges emanating from the origin point along the positive $x$- and $y$-axes, and the
polygon is contained in the positive quadrant.
\textbf{Almost toric fibrations}, defined in \S\ref{ssec:ATFdefs}, and natural operations on them allow us to modify the Delzant polygon of $M_\Omega$ to indicate new fibrations. We use the modified Delzant polygon to identify new embeddings $E(c,d)\sembeds M_\Omega$ and \cite[Thm.~1.4]{cghmp}, stated below, to prove there is thus an embedding into $X_\Omega$.

\begin{prop}[{\cite[Theorem~1.4]{cghmp}}]\label{prop:compact-domain}
    If $M_\Omega$ is the toric symplectic manifold with Delzant polygon $\Omega$, then
    \[
    E(c,d)\sembeds M_\Omega \iff E(c,d)\sembeds X_\Omega.
    \]
\end{prop}

\subsection{Quasi-perfect Diophantine classes}\label{ssec:ECdefs}
Embeddings of rational ellipsoids into finite type convex toric domains are completely characterized by the homology classes of symplectically immersed spheres in 
blow ups of $\C P^2$, a method due to McDuff and Polterovich (see the proof of \cite[Prop.~3.2]{m2} and the original reference of \cite{mp}). 
We will not review this entire story, but refer the reader to the original proof, the in-depth survey \cite{Hsurvey} for the case of ellipsoid targets, or the shorter, more general summary in \cite[\S2.3]{cghmp}.
Here we make the definitions and simplifications %of McDuff's method in \cite{m2}
used in this paper.

Define the \textbf{integral weight expansion} $W(p,q)$ of a pair of coprime integers $p>q$ recursively by
\[
W(q,p)=W(p,q)=(q)\cup W(p-q,q),
\]
and the \textbf{weight expansion} $\mathbf{w}(z)$ of a rational number $z=p/q$ to be
\[
\mathbf{w}(z):=W(p,q)/q.
\]
The \textbf{weights} of $z$ are the entries in its weight expansion. Irrational numbers also have (infinite) weight expansions $\mathbf{w}(z):=W(z,1)$.

\begin{example}
    We compute
    \begin{align*}
    W(41,5)&=(5)\cup W(36,5)
    \\&=(5,5)\cup W(31,5)
    \\&=\cdots=(5^{\times8})\cup W(5,1)
    \\&=(5^{\times8},1^{\times5}),
    \end{align*}
    thus $\mathbf{w}(41/5)=(1^{\times8},1/5^{\times5})$.
\end{example}

\begin{rmk}\label{rmk:weights}
\begin{enumerate}[label=(\roman*)]
    \item The continued fraction of $z$ equals the list of multiplicities of its weights, e.g.,
    \[
    [8,5]=8+\frac{1}{5}=\frac{41}{5}.
    \]
    \item By \cite[Lem.~1.2.6]{ball}, if $\mathbf{w}(p/q)=(w_1,\dots,w_M)$, then
    \begin{align*}
        \sum_{i=1}^Mw_i^2&=\frac{p}{q}
        \\\sum_{i=1}^Mw_i&=\frac{p}{q}+1-\frac{1}{q}
    \end{align*}
    \end{enumerate}
\end{rmk}

\begin{definition} We call a 5-tuple of integers
\[
\bE=(d,e,p,q,t)
\]
with $p$ and $q$ coprime a \textbf{quasi-perfect Diophantine class} if 
\begin{equation}\label{eqn:Ds}
    2(d+e)=p+q,\quad 2de=pq-1,\quad t=\sqrt{p^2+q^2-6pq+8}.
\end{equation}
We say $p/q$ is the \textbf{center} of $\bE$, and call the first two equations in (\ref{eqn:Ds}) the \textbf{Diophantine equations}.
\end{definition}

% Combining \cite[Prop.~3.2]{m2} and \cite[\S6.1]{biran}, we may interpret $c_\beta$ in terms of Diophantine classes:
% \begin{equation}\label{eqn:cbED}
% c_\beta(z)=\sup_{\bE\text{ Diophantine}}\left\{\mu_{\bE,\beta}(z),\vol_\beta(z)\right\},
% \end{equation}
% where if $\bE=(d,e,p,q,t)$, then
Let $\mu_{\bE,\beta}$ be the obstruction function defined by
\[
\mu_{\bE,\beta}(z):=\frac{W(p,q)\cdot\mathbf{w}(z)}{d+e\beta}.
\]
Our computations in \S\ref{sec:proofs} will rely on the fact that
\begin{equation}\label{eqn:cboundmu}
c_\beta(z)\geq\mu_{\bE,\beta}(z)
\end{equation}
for all quasi-perfect Diophantine classes $\bE$. This follows from the fact that $\bE$ represents the homology class of a symplectically immersed sphere in a blowup of $\C P^2$; the fact that the immersion is symplectic means the sphere has positive area, providing us with an inequality. For the purposes of this paper, \eqref{eqn:cboundmu} may be taken as a black box following from \cite[Prop.~3.2]{m2}.

\begin{rmk}
\begin{enumerate}[label=(\roman*)]
    \item Computing $\mu_{\bE,\beta}$ at the center of $\bE$ is particularly simple by Remark \ref{rmk:weights} (ii):
\begin{equation}\label{eqn:mupq}
    \mu_{\bE,\beta}\left(\frac{p}{q}\right)=\frac{q\bw(p/q)\cdot\bw(p/q)}{d+e\beta}=\frac{p}{d+e\beta}.
\end{equation}
%It is \eqref{eqn:mupq} which, in combination with \eqref{eqn:cbED}, allows us to compute key lower bounds to $c_\beta$ in Propositions \ref{prop:oc}~(i) and \ref{prop:oc}~(ii).
Many outer corners of $c_\beta$, including those in the infinite staircase of Theorem \ref{thm:main}, have $z$-values equal to centers of quasi-perfect Diophantine classes, and near those centers $c_\beta(z)=\mu_{\bE,\beta}(z)$.

\item The fact that $t$ is an integer is redundant:
\[
t^2=4(d+e)^2-16de=4(d-e)^2.
\]
\end{enumerate}
\end{rmk}

Finally, we have the following identities relating $d,e,p,q$, and $t$: 
\begin{lemma}\label{lem:depqt} A integral tuple $(d,e;p,q,t)$ is a quasi-perfect Diophantine class if and only if $t$ is defined from $p,q$ as in \eqref{eqn:Ds} and there are integers $(d,e)$ such that
\[
4d=p+q+t\quad \mbox{and} \quad 4e=p+q-t.
\]
\end{lemma}
\begin{proof}
    Using the linear Diophantine equation \eqref{eqn:Ds}, we solve for $e$:
    \[
    e=\frac{p+q}{2}-d.
    \]
    We then plug this into the quadratic Diophantine equation, giving us
    \begin{align*}
    d(p+q-2d)=pq-1 &\iff 2d^2-d(p+q)+(pq-1)=0
    \\&\iff d=\frac{p+q+\sqrt{(p+q)^2-8(pq-1)}}{4}
    \\&\iff 4d=p+q+t,
    \end{align*}
    using the fact that $d>e$. The formula for $e$ follows in exactly the same way, using the fact that $e<d$ to obtain the other solution in the quadratic formula.
\end{proof}

\subsection{ECH capacities}\label{ssec:ECHdefs} 

Another way to obtain a lower bound on the ellipsoid embedding function of a symplectic manifold is through embedded contact homology (ECH). Computing these lower bounds is algorithmic, and so allows us to explore the space of ellipsoid embedding functions $c_\beta$ efficiently. In Lemma \ref{lem:Eck} we relate ECH obstructions to quasi-perfect Diophantine classes.

Defined in \cite{Hut}, the \textbf{ECH capacities} of a convex toric domain $X_\Om$ form a sequence
\[
0=c_0(X_\Om)<c_1(X_\Om)\leq c_2(X_\Om)\leq\cdots\leq\infty,
\]
which obstruct symplectic embeddings:
\[
X_\Om\sembeds X_{\Om'} \Rightarrow c_k(X_\Om)\leq c_k(X_{\Om'})\;\forall k.
\]
Our computation of ECH capacities for $P(1,\beta)$ is based on \cite[App.~A]{cg}. 

\begin{definition}
    A \textbf{convex lattice path} $\Lam:[0,1]\rightarrow\R_{\ge 0}$ is a continuous map satisfying \begin{enumerate}
        \item piecewise linearity,
        \item all \textbf{vertices} (nonsmooth points) lie in $\Z^2$,
        \item $\Lam(0)$ is on the $y$-axis and $\Lam(a)$ is on the $x$-axis,
        \item the region enclosed by $\Lam$ and the axes is convex.
        \end{enumerate}
    Its \textbf{edges} are the vector differences between adjacent vertices.
\end{definition}
The function $\calL(\Lam)$ counts the number of lattice points enclosed by $\Lam$, which includes points on $\partial\Lam$ and those lying on the axes. We further define the \textbf{$\Om$-length} $\ell_\Om(\Lam)$ of a given path $\Lam$ as $$\sum_{\nu\in\text{Edges}(\Lam)}\det{[\nu~p_{\Om,\nu}]}$$
where $p_{\Om,\nu}\in\partial\Om$ is the unique point where $\nu$, shifted to be based at $p_{\Om,\nu}$, is tangent to $\partial\Om$ and where $\Om$ lies entirely to the right-hand side of $\nu$. See Figure \ref{fig:pOmegav}. 

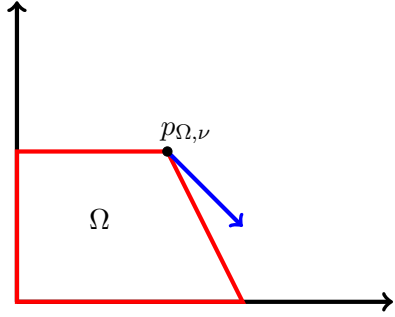
\begin{figure}[H]
\centering

\begin{tikzpicture}
  % lattice
  \draw[->,ultra thick] (0,0) -- (5,0);
  \draw[->,ultra thick] (0,0) -- (0,4);
  \begin{scope}[]
    \draw[ultra thick, red, %postaction={decorate},decoration={
    %       markings,
    %       mark=at position .145 with {\arrow{latex}},
    %       mark=at position .375 with {\arrow{latex}},
    %       mark=at position .635 with {\arrow{latex}},
    %       mark=at position .875 with {\arrow{latex}},
    %     }
    ]
    (0,0) -- +(0,2) -- +(2,2) -- +(3,0) -- cycle;
  \end{scope}
  \draw[->,ultra thick,blue] (2,2) -- (3,1);
  \filldraw (2,2) circle[radius=1.75pt];
  \node at (2.25,2.25) {$p_{\Om,\nu}$};
  \node at (1.1,1.1) {$\Om$};
\end{tikzpicture}
\caption{With $\Omega$ the region outlined in red and $\nu$ in blue, the black point is $p_{\Om,\nu}$.}
\label{fig:pOmegav}
\end{figure}

\begin{thm}[{\cite[Cor.~A.5]{cg}}]\label{thm:echcvx} If $X_\Omega$ is a convex toric domain, then
$$c_k(X_\Om)=\min\{\ell_{\Om}(\Lam): \text{convex lattice paths $\Lam$ where $\calL(\Lam)=k+1$}\}.$$
\end{thm}

The ECH capacities of an ellipsoid $E(a,b)$ can also be computed via:
\begin{prop}[{\cite[Prop.~1.2]{Hut}}]\label{prop:echE} Let $N(a,b)$ be the sequence of elements of the array $(am+bn)_{m,n\in\N}$ listed in ascending order with repetitions. The $k$th element indexed from zero of this sequence, $N_k(a,b)$, is exactly equal to $c_k(E(a,b))$. 
\end{prop}

The use of ECH capacities to obstruct symplectic embeddings of ellipsoids into some 
target relies on the following result of Frenkel-M\"uller and Hutchings, which is also a special case of a theorem of Cristofaro-Gardiner.

\begin{thm}[{\cite[Cor.~1.5]{frenkelmuller}}, {\cite[Cor.~11]{Hsurvey}}, {\cite[Thm.~1.2]{cg}}]\label{thm:mcdcg}
%Given $X_\Om$ a convex toric domain,
There exists a symplectic embedding \[E(1,z)\xhookrightarrow{s} P(1,\beta)\] if and only if 
\[c_k(E(1,z))\le c_k(P(1,\beta))\] for all $k\in\mathbb{Z}_{\ge 0}$.
\end{thm}
Since our target is $P(1,\beta)$, which is convex, we used the methods of \cite[\S5]{ICERM} as well as Theorems \ref{thm:echcvx}, \ref{thm:mcdcg} and Proposition \ref{prop:echE} to compute a lower bound for $c_\beta$. We 
identify $(P(1,\beta),\lambda\omega_0)=(P(\lambda,\lambda\beta),\omega_0)$ by the diffeomorphism $(\zeta_1,\zeta_2)\mapsto(\sqrt{\lambda}\zeta_1,\sqrt{\lambda}\zeta_2)$.
\begin{align*}
    \lambda\geq c_\beta(z)&\iff E(1,z)\sembeds P(\lambda,\lambda\beta)
    \\&\iff c_k(E(1,z))\leq c_k(P(\lambda,\lambda\beta))\;\forall k
    \\&\iff c_k(E(1,z))\leq \lambda\cdot c_k(P(1,\beta))\;\forall k
    \\&\iff \frac{c_k(E(1,z))}{c_k(P(1,\beta))}\leq \lambda\;\forall k,
\end{align*}
where the third line follows by the conformality of ECH capacities \cite[(2.5)]{Hut}. Because $c_\beta(z)$ is the infimum over all such $\lambda$, we obtain
\begin{equation}\label{eqn:echsup}
c_\beta(z)=\sup_k\frac{c_k(E(1,z))}{c_k(P(1,\be))}.
\end{equation}
It is \eqref{eqn:echsup} which allowed us to explore the space of functions $c_\beta$ for potential infinite staircases and identify our values in Theorem \ref{thm:main} and Conjecture \ref{thm:generalb}, by computing
\begin{equation}\label{eqn:maxK}
\max_{k\leq K}\frac{c_k(E(1,z))}{c_k(P(1,\be))}\leq c_\beta(z),
\end{equation}
for $K$ large (e.g. $K=25,000$ or $100,000$). The maximum in \eqref{eqn:maxK} is a good approximation for $c_\beta$ when $K$ is large by \cite[Thm.~1.1]{cghr}.

We can use ECH capacities to identify outer corners of $c_\beta$. Complementary to \eqref{eqn:mupq} and \eqref{eqn:cboundmu}, we can use individual convex lattice paths, Theorem \ref{thm:echcvx}, and \eqref{eqn:echsup} to compute precise lower bounds to values of $c_\beta$ at specific values of $z$. That is, to prove
\[
c_{\be}(z)\geq \lambda,
\]
it is enough to find a single lattice path $\Lam$ for which $N_k(1,z)/\ell_{\Om_\beta}(\Lam)=\lambda$: see Remark \ref{rmk:echlbs}. This is the method used in \cite{cghmp}.

However, in order to make use of \cite{M1}, we prove Theorem \ref{thm:main} using quasi-perfect Diophantine classes rather than ECH capacities. Analogously to \cite[Lem.~92]{ICERM}, we may translate between these perspectives:
\begin{lemma}\label{lem:Eck} If $\bE=(d,e,p,q,t)$ is a quasi-perfect Diophantine class, then
\[
\mu_{\bE,\beta}\left(\frac{p}{q}\right)\leq\frac{c_k(E(1,p/q))}{c_k(P(1,\beta))},
\]
where $k=(d+1)(e+1)-1=\frac{(p+1)(q+1)}{2}-1$.
\end{lemma}

Lemma \ref{lem:Eck} allows us to translate between the obstructions from ECH capacities, which are algorithmic and thus good tools for analyzing $c_\beta$ visually by Theorem \ref{thm:echcvx} (see Figures \ref{fig:E_1} and \ref{fig:sqrt3desc}), and quasi-perfect Diophantine classes, which carry more information. Note that if $c_\beta(p/q)=\mu_{\bE,\beta}(p/q)$ then the conclusion of Lemma \ref{lem:Eck} is an equality.

\begin{proof} It suffices to provide a lattice path $\Lam_\bE$ with $\mu_{\bE,\beta}(z)\leq N_k(1,z)/\ell_{\Omega_\beta}(\Lam_\bE)$: this is simply the rectangle with corners the origin, $(0,e), (d,e)$, and $(d,0)$. We check the conclusions of the lemma.

Firstly,
\[
(d+1)(e+1)=\frac{(p+1)(q+1)}{2} \iff 2de+2(d+e)+1=pq+p+q+1,
\]
which follows from \eqref{eqn:Ds}.

Secondly, the edges of $\Lam_{\bE}$ are $(d,0)$ and $(0,-e)$. For both, we may use $p_{\Om_\beta,\nu}=(\beta,1)$. Thus
\[
\ell_{\Omega_\beta}(\Lam_\bE)=\det\begin{pmatrix}d&\beta\\0&1\end{pmatrix}+\det\begin{pmatrix}0&\beta\\-e&1\end{pmatrix}=d+e\beta,
\]
which is the denominator of $\mu_{\bE,\beta}$.

Finally, it remains to show that $N_k(1,p/q)=W(p,q)\cdot\bw(p/q)=p$. Identify the nonnegative integer linear combinations of $1$ and $z=p/q$ with lattice points in $\Z^2_{\geq0}$. We will show that if $(x_0,y_0)$ is either $(p,0)$ or $(0,q)$, there are exactly $k=(p+1)(q+1)/2-1$ lattice points in the first quadrant with $x+py/q<x_0+py_0/q$, and thus $N_k(1,p/q)=x_0+py_0/q=p$.

Let $T$ be the triangle below the line $x+py/q<p$ and above the axes. If $I$ denotes the number of interior points of $T$ and $B$ its number of boundary points, the number of lattice points in the first quadrant below the line $x+py/q<p$ is $I+B-2$. By Pick's Theorem applied to $T$,
\[
I+\frac{B}{2}-1=\frac{pq}{2} \iff I+B-2=\frac{pq}{2}-1+\frac{B}{2}=\frac{pq}{2}-1+\frac{p+q+1}{2}=k,
\]
as desired.
\end{proof}
Note that it is not too difficult to extend the conclusion of Lemma \ref{lem:Eck} to an interval containing $p/q$ as in \cite[Lem.~92]{ICERM}, but we do not need this here.  We conclude this subsection with a figure illustrating the constraint
a single obstruction at a single $z$-value imposes on the embedding capacity function.

\begin{center}
    \begin{figure}[H]
\begin{tikzpicture}
\fill[blue!20] (0.33,0.12) -- (0.33,4.7) -- (10.5,4.7) -- (10.5,3.65) -- (4.4,3.65) -- cycle;
\begin{axis}[
	axis lines = middle,
	xtick = {0,1,...,6},
	ytick = {0,2,...,8},
	tick label style = {font=\small},
	xlabel = $z$,
	ylabel = $\lambda$,
	xlabel style = {below right},
	ylabel style = {above left},
	xmin=-0.2,
	xmax=6.2,
	ymin=-0.2,
	ymax=8.2,
	grid=major,
	width=4.75in,
	height=2.5in]
\addplot [blue, thin,
	domain = 0:6.25,
	samples = 100
](x,2.5*x);
\addplot [blue, thin,
	domain = 0:6.25,
	samples = 100
](x,6.25);
\addplot [blue, only marks,  thick, mark=*] coordinates{(2.5,6.25) };
\end{axis}
\end{tikzpicture}
\caption{The figure depicts the effect of an obstruction providing a lower bound for 
$c_X$ at the indicated blue point. The ellipsoid embedding function $c_X$ must lie in the blue shaded region by 
Proposition~\ref{prop:cXprops} (ii) and (iii).
}\label{fig:obstruction-effects} 
\end{figure}
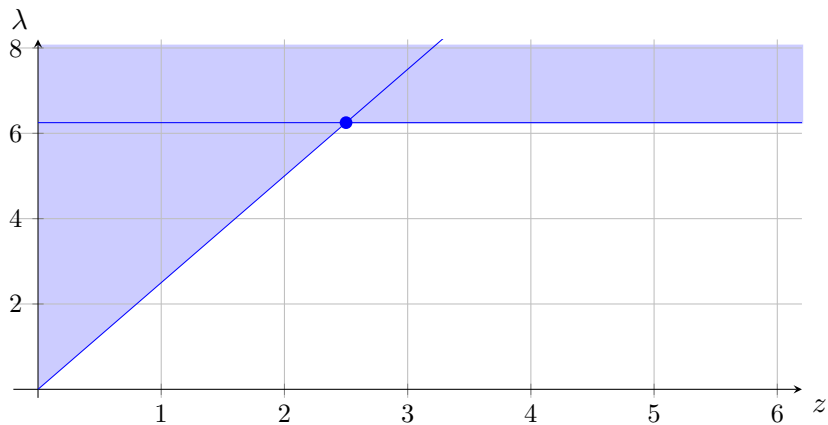
\end{center}

\subsection{Almost toric fibrations}\label{ssec:ATFdefs}

Symplectic embeddings provide a useful counterpoint to the obstructions described 
in sections~\ref{ssec:ECdefs}-\ref{ssec:ECHdefs}. We will use combinatorial techniques 
developed in the theory of {\bf almost toric fibrations (ATFs)}
to establish the existence of embeddings.  Introduced by Symington \cite{symington}
and developed further in \cite{leung-symington, evans},
an ATF is a completely integrable system on a compact symplectic four-manifold with
elliptic and focus-focus singularities.  This framework provides a map
from the manifold $M$ to $\R^2$ whose image is called the {\bf base diagram}.  
There are combinatorial operations called {\bf nodal trades}, {\bf nodal slides}, and {\bf mutation} 
on the base diagram that correspond
to symplectomorphisms of the corresponding manifolds.  These allow us to
discover embeddings of ellipsoids into the manifold $M$ by identifying appropriate
triangles inside the variously manipulated base diagrams.  We may then
use Proposition~\ref{prop:compact-domain} to deduce that the convex toric
domain $X$ also has the same ellipsoid embeddings.

In order to study the polydisk $P(1,\beta)$ with ATFs,
we first find a compact manifold  that has an
ATF with base diagram the $1\times b$ rectangle.
The manifold $M_\beta=\C P^1_1\times \C P^1_\beta$, a product
of two copies of projective space, the first with
size $1$ and the second with
size $\beta$, is equipped with a toric $T^2$ action
by rotation in each factor.  This has moment map
image the Delzant polygon which is the $1\times \beta$ rectangle.
This is our starting point for manipulations using ATF
tools.  These tools will change the map
$M_\beta\to \R^2$ and its image, but not the manifold $M_\beta$ itself.

The first step is to apply a {\bf nodal trade} at each of 
the three vertices $X$, $Y$ and $V$ which are not 
at the origin in $\R^2$.
Geometrically in $M_\beta$, this means excising the neighborhood
of the fixed point corresponding to a vertex and then
gluing in a local model of a focus-focus singularity.
At the level of the base diagram, this corresponds 
to adding a ray with a marked point emanating from 
the {\bf anchor} vertex $P$.  Above the marked point on the ray, there is a
pinched torus.  The pinch point is the new focus-focus singularity
for the updated map $M_\beta\to\R^2$.
If we let $\overrightarrow{E}$ and 
$\overrightarrow{F}$ denote the primitive vectors (in $\Z^2$)
pointing along the edges emanating from $P$, then the
smoothness of $M_\beta$ guarantees that
$\overrightarrow{E}$ and $\overrightarrow{F}$ form a
$\Z$ basis of $\Z^2$.
With this notation, then, the nodal ray that we introduce 
points in the direction $\overrightarrow{E} + \overrightarrow{F}$. 
A useful fact, which follows from a straightforward linear algebraic 
calculation, is that both pairs
$(\overrightarrow{E}, \overrightarrow{E}+\overrightarrow{F})$ and 
$(\overrightarrow{E}+\overrightarrow{F},\overrightarrow{F})$ are
$\Z$ bases of $\Z^2$.

The second operation we can apply to a base diagram is called a {\bf nodal slide}.
The local model for a focus-focus singularity has one degree of freedom, 
corresponding to moving the pinched torus further or closer to the level set
above the vertex.  In the base diagram, this corresponds to moving the
marked point along the nodal ray.

The third operation is {\bf mutation} along a nodal ray of the base diagram.
This changes the shape of the diagram.  At the level of the function $M_\beta\to \R^2$, if the  marked point's location does not move, this corresponds to taking the same function, but choosing a different branch cut to visualize the image of the function.

Combinatorially, the base diagram is divided in two by the line generated
by the nodal ray.  The mutation operation leaves one piece unchanged (which for us will always be the piece containing the origin) and acts
on the other piece by an affine linear transformation that
\begin{itemize}
    \item fixes the anchor vertex;
    \item fixes the nodal ray; and
    \item aligns the two edges emanating from the anchor vertex.
\end{itemize}
There is a unique transformation in $ASL_2(\Z)$ that achieves this, as a consequence
of the linear algebraic fact about the edge rays and nodal rays noted above.
The other changes to the base diagram are the creation of a new (anchor) vertex 
and nodal ray (the negative of the previous).  This is illustrated in 
Figure~\ref{fig:mutation} below.

\begin{center}
  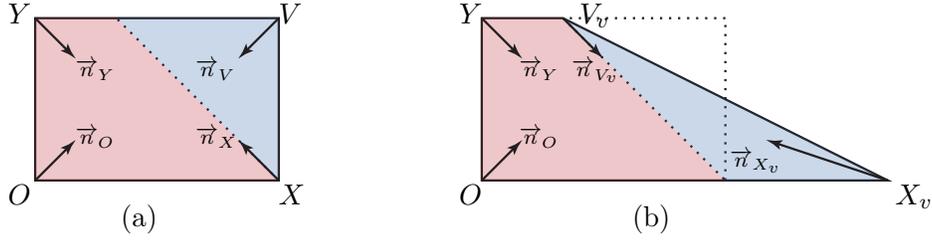
\begin{figure}[ht]
\begin{overpic}[%grid,
scale=0.85,unit=1mm]{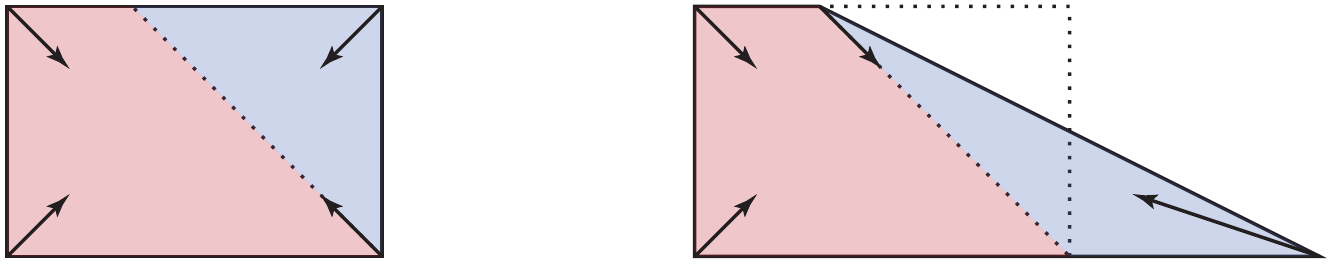}
   \put(12,-5){(a)}
   \put(80,-5){(b)}
   \put(-3,-2){$O$}
   \put(57,-2){$O$}
   \put(-3,22){$Y$}
   \put(57,22){$Y$}
   \put(33,-2){$X$}
   \put(33,22){$V$}
   \put(73,22){$V_v$}
   \put(115,-2){$X_v$}
   \put(6,6){$\scriptstyle\overrightarrow{n}_O$}
   \put(6,15){$\scriptstyle\overrightarrow{n}_Y$}
   \put(22,15){$\scriptstyle\overrightarrow{n}_V$}
   \put(22,6){$\scriptstyle\overrightarrow{n}_X$}
   \put(65,6){$\scriptstyle\overrightarrow{n}_O$}
   \put(65,15){${\scriptstyle\overrightarrow{n}_Y}$}
   \put(72,15){${\scriptstyle\overrightarrow{n}_{V_v}}$}
   \put(93,3){${\scriptstyle\overrightarrow{n}_{X_v}}$}
%   \put(85,18){${\scriptstyle\overrightarrow{n}_{V_v} = -\overrightarrow{n}_X}$}
%   \put(103,8){${\scriptstyle\overrightarrow{n}_{X_v}  = M\cdot\overrightarrow{n}_V}$}
\end{overpic}
\vskip 0.1in
\caption{We apply a mutation about the vertex $V$ to the figure in (a) to obtain the figure in (b).   The mutation fixes the red region and applies an affine linear transformation encoded in a matrix $M$ to the blue region.  The effect on the vertices and is indicated.  The nodal rays in (b) that differ from (a) are given by
 ${\protect\overrightarrow{n}_{V_v} = -\protect\overrightarrow{n}_X}$ and ${\protect\overrightarrow{n}_{X_v}  = M\cdot \protect\overrightarrow{n}_V}$.}
\label{fig:mutation}
\end{figure}
\end{center}

Procedurally, we apply a sequence of mutations with the goal
of finding wider and wider triangles inside the mutated base diagram.  
The impact that one mutation has on triangles that
fit inside the base diagram is illustrated in Figure~\ref{fig:mutation-impact}.

\begin{center}
  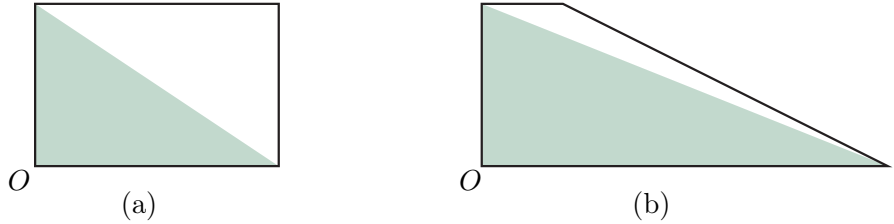
\begin{figure}[ht]
\begin{overpic}[%grid,
scale=0.85,unit=1mm]{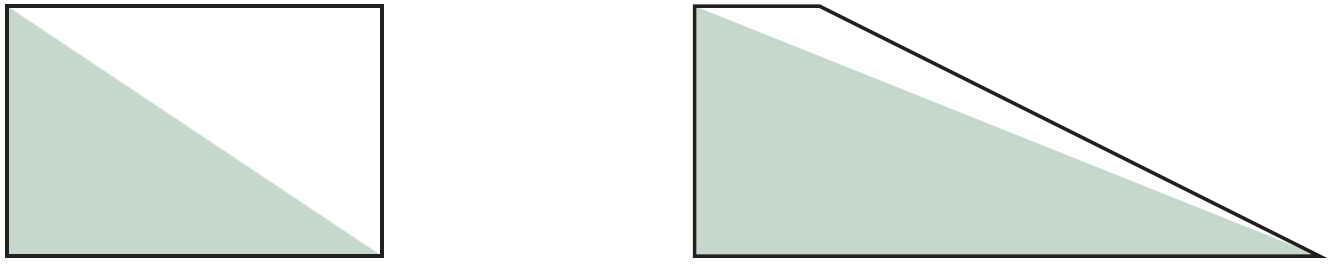}
   \put(12,-5){(a)}
   \put(80,-5){(b)}
   \put(-3,-2){$O$}
   \put(57,-2){$O$}
\end{overpic}
\vskip 0.1in
\caption{The figures in (a) and (b) are related by a mutation, as in 
Figure~\ref{fig:mutation}. The fact that the green triangles centered at $O$
have different proportions indicates that we have embeddings of ellipsoids with
different eccentricities into the corresponding polydisk.
}
\label{fig:mutation-impact}
\end{figure}
\end{center}

\begin{remark}\label{rmk:labeling}
When discussing ATF base diagrams and their mutations, we will use the following conventions.
    \begin{description}
        \item[Vertices] We set $O=(0,0)$, use $X$ and $Y$ to denote the vertices on the $x$- and $y$-axes, respectively, and use $V$ to denote the vertex strictly in the positive quadrant.
        \item[Nodal rays] The nodal ray of vertex $A$ is labeled $\vec{n}_A$.
        \item[Side directions] The primitive integral 
        vector parallel to the side $AB$ is denoted $\overrightarrow{AB}$.
        \item[Affine lengths] The affine length of the side $AB$ is denoted $|AB|$.
        \item[Mutations] The new vertex at its position (relative to the axes) after a mutation at vertex $A$ has a subscript lowercase $a$. For example, the vertex on the $y$-axis after mutation at $A$ is denoted $Y_a$.
        \item[Sequences of mutations] We denote a sequence of mutations by a word in the lowercase letters $x, y, v$, read from left to right. E.g. $v^2yx$ means ``mutate at $V$ twice, then mutate at $Y$, then mutate at $X$.''
    \end{description} 
\end{remark}
We note that after a mutation, the nodal rays are transformed in one of three ways: not at all; by taking the negative; or 
by applying the mutation matrix $M$.
Because our base diagrams are polygons, we will make use of the key identity
 \begin{equation}\label{eqn:4addzero}
        |OY_a|\overrightarrow{OY_a}+|Y_aV_a|\overrightarrow{Y_aV_a}-|X_aV_a|\overrightarrow{X_aV_a}-|OX_a|\overrightarrow{OX_a}=\begin{pmatrix}
            0\\0
        \end{pmatrix},
    \end{equation}
    derived from the fact that the four sides must close up. 
    
The following result makes precise the relationship between triangles in the base diagram and symplectic embeddings of ellipsoids.

\begin{prop}[{\cite[Prop.~2.35]{cghmp}}]\label{prop:atfswork}
   Suppose that a symplectic manifold $X$ is equipped with an almost toric fibration with base diagram 
   $\Delta_X$ that consists of a closed region in $\R^2_{\geq 0}$ that is bounded by the axes and a 
   convex (piecewise-linear) curve from $(a,0)$ to $(0,b)$, for $a,b \in \R^+$. Suppose in addition that there 
   is no nodal ray emanating from $(0,0)$. Then there exists a symplectic embedding of the ellipsoid $(1 -  \varepsilon)\cdot E(a,b)$ into $X$ 
   for any $0 < \varepsilon < 1$. 
\end{prop}

While the obstructions in sections~\ref{ssec:ECdefs}-\ref{ssec:ECHdefs} give lower bounds on $c_X$, as indicated in Figure~\ref{fig:obstruction-effects}, 
a single embedding forces certain
upper bounds on the embedding capacity function.  As we will see, the combination of the two 
can strongly restrict $c_X$.

%We use these propositions to prove our main theorem. The idea of the proof is contained in 
%the following figures.  First, we see that a single obstruction forces certain lower bounds on the embedding capacity 
%function, using Propositions~\ref{prop:oc}~(i) and \ref{prop:oc}~(ii).  Similarly, a single embedding forces certain
%upper bounds on the embedding capacity function, using Proposition~\ref{prop:ic1} and Conjecture~\ref{conj:ic2}.
%Figures \ref{fig:obstruction-effect}-\ref{fig:combination}. 
%First, we consider an obstruction, abstracting the lower bounds of Propositions \ref{prop:oc}~(i) and \ref{prop:oc}~(ii).

\begin{center}
\begin{figure}[H]
\begin{tikzpicture}
\fill[red!20] (0.33,0.12) -- (0.33,3.66) -- (6.05,3.66) -- (7.75,4.7)  -- (10.5,4.7) -- (10.5,0.12) -- cycle;
\begin{axis}[
	axis lines = middle,
	xtick = {0,1,...,6},
	ytick = {0,2,...,8},
	tick label style = {font=\small},
	xlabel = $z$,
	ylabel = $\lambda$,
	xlabel style = {below right},
	ylabel style = {above left},
	xmin=-0.2,
	xmax=6.2,
	ymin=-0.2,
	ymax=8.2,
	grid=major,
	width=4.75in,
	height=2.5in]
\addplot [red, thin,
	domain = 0:6.25,
	samples = 100
](x,6.25*x/3.5);
\addplot [red, thin,
	domain = 0:6.25,
	samples = 100
](x,6.25);
\addplot [red, only marks,  thick, mark=*] coordinates{(3.5,6.25) };
\end{axis}
\end{tikzpicture}
\caption{By contrast to Figure~\ref{fig:obstruction-effects}, an embedding provides an upper bound for $c_X$ at the indicated 
red point. The function $c_X$ must lie in the red shaded region by Proposition~\ref{prop:cXprops} 
(ii) and (iii).
}\label{fig:embedding-effects} 
\end{figure}
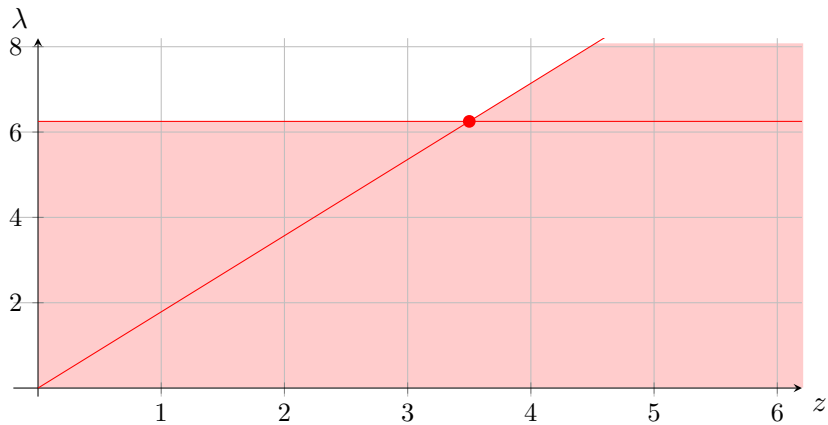
\end{center}

\subsection{Combining obstructions and embeddings}\label{ss:combination}

Combining the effects in Figures~\ref{fig:obstruction-effects} and \ref{fig:embedding-effects}, we see how to 
prove that the combination of lower bounds provided by obstructions
(quasi-perfect Diophantine classes or ratios of ECH capacities) with upper bounds provided by an embedding allows us 
to establish the existence of an infinite staircase.  
A combination of obstructions and embeddings allows us to nail down the ellipsoid embedding function for some ranges 
of $z$-values (indicated by violet segments in Figure \ref{fig:combination}), and provides bounds on $c_X$ for 
other ranges of $z$-values (indicated by violet regions in Figure \ref{fig:combination}).
In this way, one can establish the existence of an infinite staircase without computing the entire function.
Or if the embeddings and obstructions are lined up just so, one might just compute the entire function. Note, this is usually only a effective strategy before the accumulation point.

\begin{center}
\begin{figure}[H]
\begin{tikzpicture} %[transform canvas={scale=0.75}]
\fill[violet!20] (5.25,3.525) -- (7.1,4.8) -- (10.5,4.8) -- (10.5,3.525) -- cycle;
\fill[violet!20] (2.8,2.68) -- (3.08,2.98) -- (4,2.98) -- (3.6,2.68) -- cycle;
\begin{axis}[
	axis lines = middle,
	xtick = {0,1,...,6},
	ytick = {0,2,...,8},
	tick label style = {font=\small},
	xlabel = $z$,
	ylabel = $\lambda$,
	xlabel style = {below right},
	ylabel style = {above left},
	xmin=-0.2,
	xmax=6.2,
	ymin=-0.2,
	ymax=8.2,
	grid=major,
	width=4.75in,
	height=2.5in]
\addplot [gray, thin,
	domain = 0:2.53125,
	samples = 100
](x,2*x);
\addplot [gray, thin,
	domain = 0:2,
	samples = 100
](x,2.25*x);
\addplot [gray, thin,
	domain = 0:1.3333333,
	samples = 100
](x,3*x);
\addplot [gray, thin,
	domain = 0:1,
	samples = 100
](x,4*x);
\addplot [violet, ultra thick,
	domain = 1:4/3,
	samples = 100
](x,4);
\addplot [violet, ultra thick,
	domain = 4/3:1.5,
	samples = 100
](x,3*x);
\addplot [violet, ultra thick,
	domain = 2.25:2.53125,
	samples = 100
](x,5.0625);
\addplot [violet, ultra thick,
	domain = 2.53125:3,
	samples = 100
](x,2*x);
\addplot[dotted, violet, thick,
	domain = 1.5:2
](x,4.5);
\addplot[dotted, violet, thick,
	domain = 2:2.25
](x,2.25*x);
\addplot[dotted, violet, thick,
	domain = 1.6875:2.25
](x,5.0625);
\addplot[dotted, violet, thick,
	domain = 1.5:1.6875
](x,3*x);
\addplot[dotted, violet, thick,
	domain = 3:6.25
](x,6);
\addplot[dotted, violet, thick,
	domain = 3:6.25
](x,2*x);
\addplot [blue, only marks, very thick, mark=*] coordinates{(1,4) (1.5,4.5) (2.25,5.0625) (3,6) };
\addplot [red, only marks, very thick, mark=*] coordinates{(4/3,4) (5.0625/2,5.0625)};
\end{axis}
\end{tikzpicture}
\caption{This figure indicates several obstructions  at the blue dots and embeddings at the red dots.
Combining the bounds forced by these as shown in  Figures~\ref{fig:obstruction-effects} and \ref{fig:embedding-effects}, we deduce
that the ellipsoid embedding function must equal the violet segments and must lie in the violet shaded regions. 
In particular, it must be constant along the horizontal segment between the blue and red points at the same 
$\lambda$-value; and equal the line when a red point and blue point lie on a line through the origin.
}\label{fig:combination} 
\end{figure}
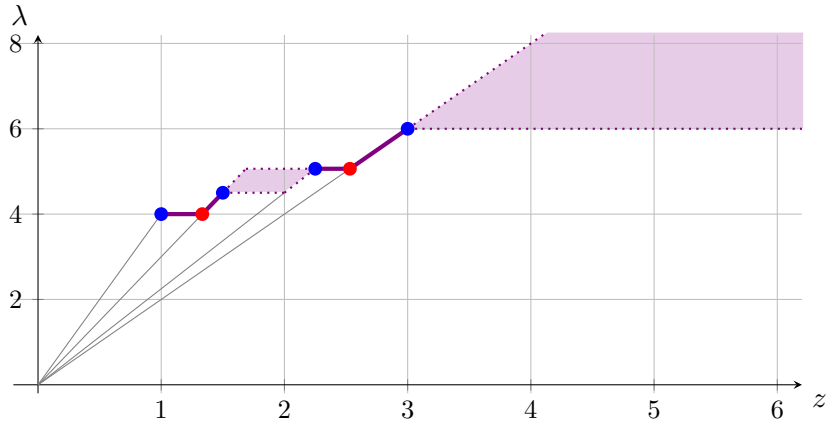
\end{center}

\section{Proof of the main theorem}\label{sec:proofs}

In this section we prove that the polydisk $P(1,\beta)$ has an infinite staircase accumulating to $\acc(\beta)$, where
\begin{equation}\label{eqn:baccb}
\beta=\frac{6+5\sqrt{30}}{12} \text{ and }
\acc(\beta)=\frac{54+11\sqrt{30}}{14}.
\end{equation}
The fact that $\acc(\beta)$ satisfies \eqref{eqn:accpteqn} with $\acc(\beta)=z$ can be verified by hand. Furthermore, set
\[
\bE=(17,6,41,5,22).
\]
The utility of $\bE$ is that it is a quasi-perfect Diophantine class whose obstruction $\mu_{\bE,\beta}$ equals the function $c_\beta$ for $z\in(\acc(\beta),41/5]$. We do not prove this latter claim, but note that on $(\acc(\beta),41/5]$ we do know (as shown in Figure \ref{fig:obstruction-effects} by setting the blue point equal to $(41/5,\mu_{\bE,\beta}(41/5))$) that
\[
c_\beta(z)\geq z\frac{\mu_{\bE,\beta}\left(\frac{41}{5}\right)}{41/5}=\frac{5z}{17+6\beta}.
\]
This is a special case of the analogous \cite[Prop.~42]{ICERM}. The numerics of $\bE$ will be crucial for studying $c_\beta$.

We next define the obstructions which we will use to prove that $c_\beta$ has an infinite staircase, following the procedure outlined in \S\ref{ss:combination}.

\begin{definition}\label{def:outerclasses}
We define the \textbf{outer class} 
\[
\bE_k:=t\bE_{k-1}-\bE_{k-2}
=(d_k,e_k,p_k,q_k,t_k)\]
where $\bE_0=(3,1,7,1,4)$ and $\bE_1=(64,23,155,19,82)$. The recursion constant is $t=22$ for all $k$. 
\end{definition}
\begin{definition}\label{def:hatclasses}
    We define the \textbf{inner class}
    \[
    \hat\bE_k:=t_{k-1}\bE_k-\bE=(\hat d_k,\hat e_k,\hat p_k,\hat q_k,\hat t_k).
    \] Note, $\hat\bE_1=(239,86,579,71,250).$
\end{definition}

\begin{rmk} In the sense of \cite{MMW}, the inner class $\hat\bE_k$ is the $x$-mutation of the triple $(\bE_{k-1},\bE_k,\bE)$. We discovered the $\hat\bE_k$ classes after Mike Usher pointed out the relationship between the $\hat A$ classes in \cite{usher} and $x$-mutation, see \S\ref{ssec:Uconj}.
\end{rmk}
The outer corners of the $\bE_k$ and $\hat\bE_k$ classes alternate in the sense that
\[
\cdots<\frac{p_{k-1}}{q_{k-1}}<\frac{\hat p_k}{\hat q_k}<\frac{p_k}{q_k}<\cdots,
\]
while the values their obstructions take at these $z$-values also alternate. See Lemma \ref{lem:alternate}, which is illustrated by Figure \ref{fig:main}.

\begin{figure}[h]
    \centering
    \includegraphics[width=\textwidth]{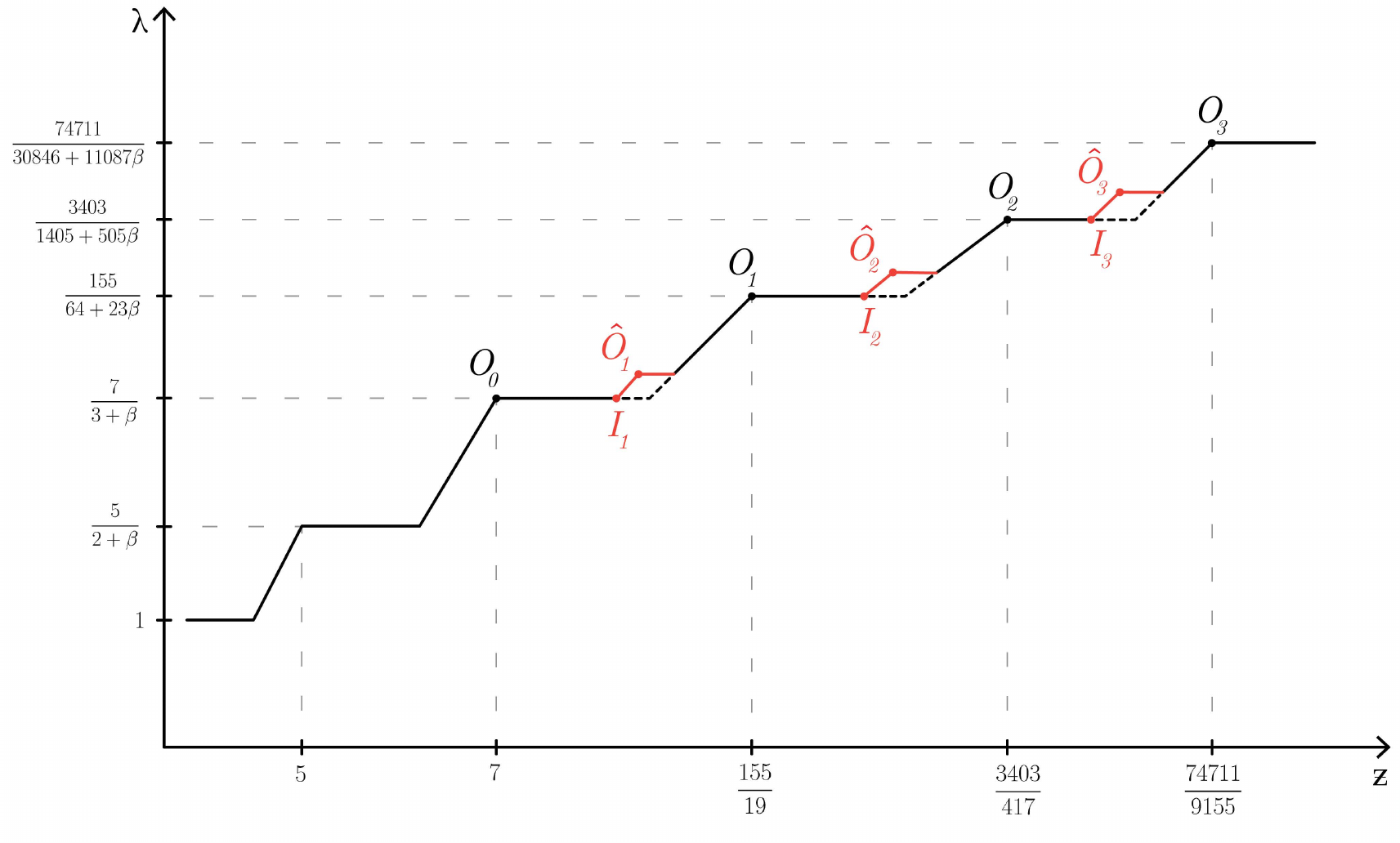}
    \caption{This figure indicates the arrangement of the first several outer and inner corners of $c_\beta$. The black outer corners, labeled $O_k$, arise from the outer $\bE_k$ classes and their coordinates are given in Proposition \ref{prop:oc}~(i). The red outer corners, labeled $\hat O_k$, arise from the inner $\hat\bE_k$ classes and their coordinates are given in Proposition \ref{prop:oc}~(ii). Because $c_\beta$ does not equal the obstruction $\mu_{\bE_k,\beta}$ from the outer $\bE_k$ classes near the intersection of $\mu_{\bE_k,\beta}$ and $\mu_{\bE_{k+1},\beta}$, these obstructions are indicated by dashed black lines where the obstructions $\mu_{\hat\bE_{k+1},\beta}$ are larger. %The coordinates of the red inner corners $I_k$ are computed in both Proposition \ref{prop:ic1} and Lemma \ref{lem:ics} (i).
    }
    \label{fig:main}
\end{figure}

In \S\ref{ssec:outer} we will prove the following proposition computing the value of $c_\beta$ at the outer corners of its infinite staircase:

\begin{prop}\label{prop:oc} We establish the following lower bounds on $c_\beta$.
\begin{enumerate}[label=(\roman*)]
    \item The outer classes $\bE_k$ determine the lower bounds
    \[
    c_\beta\left(\frac{p_k}{q_k}\right)\geq\frac{p_k
}{d_k+e_k\beta}.
    \]
    \item The inner classes $\hat\bE_k$ determine the lower bounds
\[
c_\beta\left(\frac{\hat p_k}{\hat q_k}\right)\geq\frac{\hat p_k}{\hat d_k+\hat e_k\beta}.
\]
\end{enumerate}
\end{prop}

% \label{prop:oc}~(i) We have the lower bound
%     \[
%     c_\beta\left(\frac{p_k}{q_k}\right)\geq\frac{p_k
% }{d_k+e_k\beta}.
%     \]
% \end{prop}

% \begin{prop}\label{prop:oc}~(ii) We have the lower bound
% \[
% c_\beta\left(\frac{\hat p_k}{\hat q_k}\right)\geq\frac{\hat p_k}{\hat d_k+\hat e_k\beta}.
% \]
% \end{prop}

We will also prove that the claimed outer corners $z=p_k/q_k$ have four-periodic continued fractions, which, upon proving Theorem \ref{thm:main} (i), proves Theorem \ref{thm:main} (ii).

Our final definition in this section provides notation for the intersections between the obstructions from the $\bE_k$ and $\hat\bE_k$.

\begin{definition} We set the following notation.
    \begin{itemize}
        \item We denote the points discussed in Proposition \ref{prop:oc} by
        \[
        O_k=\left(\frac{p_k}{q_k},\frac{p_k}{d_k+e_k\beta}\right), \mbox{ and } \hat O_k=\left(\frac{\hat p_k}{\hat q_k},\frac{\hat p_k}{\hat d_k+\hat e_k\beta}\right).
        \]
        \item We extend the lower bounds at $O_k$ and $\hat O_k$ by horizontal lines and lines through the origin, using Proposition \ref{prop:cXprops} (ii, iii), as illustrated in Figure \ref{fig:obstruction-effects}.
        \begin{itemize}
            \item Denote by $I_{k+1}=(z^{in}_{k+1},\lambda^{in}_{k+1})$ the intersection between the horizontal line through $O_k$ and the line through the origin and $\hat O_{k+1}$.
            \item Denote by $\hat I_{k+1}=(\hat z^{in}_{k+1},\hat \lambda^{in}_{k+1})$ the intersection between the horizontal line through $\hat O_{k+1}$ and the line through the origin and $O_{k+1}$.
        \end{itemize}
    \end{itemize}
\end{definition}

In \S\ref{ssec:inner} we will use ATFs to construct embeddings computing the value of $c_\beta$ at the points $I_k$, proving that they are inner corners. Specifically, we will show:

\begin{prop}\label{prop:ic1} At the intersections of the obstructions from $\bE_k$ and $\hat\bE_{k+1}$, we have the following upper bound:
    \[
    c_\beta(z^{in}_{k+1})\leq \lambda^{in}_{k+1}.
    \]
\end{prop}
%See Lemma \ref{lem:ics} (i) for the proof that the embedding proving Proposition \ref{prop:ic1} corresponds to the point in the $z\lambda$ plane at the intersection of the obstructions from $\bE_k$ and $\hat\bE_{k+1}$.

Next we state our conjecture which would, if proven, fully compute $c_\beta$ on $[1,\acc(\beta)]$. See Remark \ref{rmk:conj} for a discussion of the complications which arise in its potential proof.
\begin{conjecture}\label{conj:ic2} At the intersections of the obstructions from $\hat\bE_{k+1}$ and $\bE_{k+1}$, we have the following upper bound:
    \[
    c_\beta(\hat z^{in}_{k+1})\leq \hat \lambda^{in}_{k+1}.
    \]
\end{conjecture}
    
%See Lemma \ref{lem:ics} (ii) for the proof that the coordinates of Conjecture \ref{conj:ic2} are those of the point in the $z\lambda$ plane at the intersection of the obstructions from $\hat\bE_{k+1}$ and $\bE_{k+1}$.

In the following lemma we compute the coordinates of $I_{k+1}$ and $\hat I_{k+1}$.
%find the intersection points of the lower boundaries of the blue regions depicted in Figure \ref{fig:obstruction-effects} for $\bE_k$, $\hat\bE_{k+1}$, and $\bE_{k+1}$.% part (i) of the following lemma. (Part indicates the relevance of Conjecture \ref{conj:ic2} to computing $c_\beta$ on the violet shaded regions in Figure \ref{fig:combination}.)

\begin{lemma}\label{lem:ics}
    \begin{enumerate}[label=(\roman*)]
        \item We have
        \[
        (z_{k+1}^{in},\lambda_{k+1}^{in})=\left(\frac{p_k(\hat d_{k+1}+\hat e_{k+1}\beta)}{\hat q_{k+1}(d_k+e_k\beta)},\frac{p_k}{d_k+e_k\beta}\right).
        \]
        \item We have
        \[
        (\hat z_{k+1}^{in},\hat \lambda_{k+1}^{in})=\left(\frac{\hat p_{k+1}(d_{k+1}+e_{k+1}\beta)}{q_{k+1}(\hat d_{k+1}+\hat e_{k+1}\beta)},\frac{\hat p_{k+1}}{\hat d_{k+1}+\hat e_{k+1}\beta}\right).
        \]
    \end{enumerate}
\end{lemma}
\begin{proof}
The values of $\lambda^{in}_{k+1}$ and $\hat \lambda^{in}_{k+1}$ are immediate because they are the $\lambda$-values of the obstructions from $\bE_k$ and $\hat\bE_{k+1}$, respectively.

To compute $z^{in}_{k+1}$, we solve
\[
\frac{p_k}{d_k+e_k\beta}=\frac{\frac{\hat p_{k+1}}{\hat d_{k+1}+\hat e_{k+1}\beta}}{\frac{\hat p_{k+1}}{\hat q_{k+1}}}z^{in}_{k+1}
\]
for $z^{in}_{k+1}$, while to compute $\hat z^{in}_{k+1}$, we solve
\[
\frac{\hat p_{k+1}}{\hat d_{k+1}+\hat e_{k+1}\beta}=\frac{\frac{p_{k+1}}{d_{k+1}+e_{k+1}\beta}}{\frac{p_{k+1}}{q_{k+1}}}\hat z^{in}_{k+1}
\]
for $\hat z^{in}_{k+1}$.
\end{proof}

\begin{proof} (of Theorem \ref{thm:main} (i)) The lower bounds in Propositions \ref{prop:oc}~(i) and \ref{prop:oc}~(ii) combined with the upper bound in Proposition \ref{prop:ic1} prove by Lemma \ref{lem:ics} (i) that $c_\beta$ has infinitely many nonsmooth points at the inner corners between the obstructions from $\bE_k$ and $\hat\bE_{k+1}$, as indicated in Figure \ref{fig:combination}. (These inner corners are labeled $I_{k+1}$ in Figure \ref{fig:main}.) Note that to conclude that $c_\beta(\frac{p_k}{q_k})=c_\beta(z_{k+1}^{in})$ we use the fact that $c_\beta$ is increasing, which requires Lemma~\ref{lem:alternate} to know that 
\[
\frac{p_k}{q_k}\leq z^{in}_{k+1}\leq\frac{\hat p_{k+1}}{\hat q_{k+1}}.
\]
\end{proof}

\begin{rmk}
    Note that if we could show Conjecture \ref{conj:ic2}, then by Lemma \ref{lem:ics} and similar reasoning to the proof of Theorem \ref{thm:main} (i) we would be able to compute the entire function $c_\beta$ between the center $7$ of $\bE_0$ and $\acc(\beta)=\frac{54+11\sqrt{30}}{14}$. (It is very little extra work to compute $c_\beta$ on $[1,7]$, since it requires identifying only two outer and two inner corners.)
\end{rmk}

\subsection{Outer corners}\label{ssec:outer}

To prove Propositions \ref{prop:oc}, it suffices by \eqref{eqn:mupq} to show that the recursively defined families $\bE_k$ and $\hat\bE_k$ satisfy the Diophantine equations \eqref{eqn:Ds}.

For our proof, we use the ideas developed in \cite[Section ~2.2]{symm} to think of a quasi-perfect class as a integral point $(p,q,t)$ on a quadratic surface $X$ where $t=\sqrt{p^2+q^2-6pq+8}.$ In particular, as noted in Lemma~\ref{lem:depqt}, a tuple $(d,e;p,q,t)$ will satisfy the Diophantine equations if given a integral tuple $(p,q,t) \in X$, we define\footnote{Note, as defined in this way $(d,e)$ might not be integers for all integral choices of $(p,q,t)$. Thus, not all points on $X$ correspond to quasi-perfect classes.} $d,e$ by 
\[ 4d=p+q+t, \quad \text{and} \quad 4e=p+q-t.\]

We then use the result \cite[Lem.~3.1.2]{symm} which allows us to see that we can produce new tuples $(p,q,t) \in X$ via recursion assuming certain compatibility conditions hold. Let \[A:=\begin{pmatrix} -1 & 3 & 0 \\ 3 & -1 & 0 \\ 0 & 0 & 1 \end{pmatrix}, \quad \bx:=\begin{pmatrix} p \\ q \\ t \end{pmatrix}.\] 
Then the surface $X=\{\bx^T A \bx=8\}$, as $\bx^T A \bx=6pq-p^2-q^2+t^2.$

The lemma then states:
\begin{lemma}\cite[Lemma~3.1.2]{symm}
Suppose that $\bx_0$ and $\bx_1$ are integral vectors that saisfy the following conditions for some integer $\nu>0:$
    \begin{align}
        \bx_i^TA\bx_i&=8, \quad i=0,1, \label{eq:xtAx} \\
        \bx_1^TA\bx_0&=4\nu .\label{eq:nuCompat}
    \end{align}
Then, the vectors $\bx_2:=\nu \bx_1-\bx_0,\bx_1$ also satisfy these conditions for the given $\nu.$ 
\end{lemma}
We can then restate \cite[Cor.~3.1.1]{symm} for our purposes as 
\begin{cor} \label{cor:dioph} Any two integral triples $\bx_i=(p_i,q_i,t_i),$ $i=0,1$ that satisfy \eqref{eq:xtAx} and \eqref{eq:nuCompat} for a given $\nu$ can be extended to a sequence 
\[
\bx_i:=\nu\bx_{i-1}-\bx_{i-2},\quad i \geq 0,
\]
and each successive adjacent pair satisfies these conditions. Further, the corresponding quantities
\[ d_i=\frac{1}{4}(p+q+t), \quad e_i=\frac{1}{4}(p+q-t)\]
also satisfy this recursion and hence are integers, provided that they are integers for $i=0,1.$
\end{cor}

We now proceed in proving Prop~\ref{prop:oc} giving the bounds for the outer corners at $z_k=p_k/q_k$ and $\hat{z}_k=\hat{p}_k/\hat{q}_k.$

\begin{proof} (of Proposition \ref{prop:oc})
To prove (i) and (ii), we must check that the classes $\bE_k=t\bE_{k-1}-\bE_{k-2}$ and $\hat{\bE}_k=t_{k-1}\bE_{k}-\bE$ are Diophantine classes. By Cor~\ref{cor:dioph} and Lemma~\ref{lem:depqt}, it is enough to verify:
 \begin{itemlist}
     \item[-] $\bE_k$ and $\bE$ satisfy \eqref{eq:xtAx}.
     \item[-] $\bE_0=(3,1,7,1,4)$ and $\bE_1=(64,23,155,19,82)$ satisfy \eqref{eq:nuCompat} for $\nu=t=22$.
     \item[-] $\hat{\bE}_k$ and $\bE=(17,6,41,5,22)$ satisfy \eqref{eq:nuCompat} for $\nu=t_{k-1}.$
 \end{itemlist}
   
   By Cor~\ref{cor:dioph}, \eqref{eq:xtAx} will hold for $\bE_k$ if it holds for $\bE_0$ and $\bE_1.$ Thus, we must check this for $\bE_0,$ $\bE_1,$ and $\bE.$
We have
\[\bE_0: ~~ 6(7)-7^2-1^2+4^2=8,\]
\[\bE_1: ~~ 6(155)(19)-155^2-19^2+82^2=8,\]
\[\bE: ~~ 6(41)(5)-41^2-5^2+22^2=8.\]

Now, we check \eqref{eq:nuCompat} for $\bE_0,\bE_1$ with $\nu=22$: 
\[ 1(3 \cdot 155-19)+7(3 \cdot 19-155)+82 \cdot 4=4 \cdot 22.\] 

To check \eqref{eq:nuCompat} for the pair $\bE_k$ and $\bE$ with $\nu=t_{k-1},$ this involves verifying \[ 5(3p_k-q_k)+41(3q_k-p_k)+22t_k=4t_{k-1}.\] As this is a linear equation, we can verify it holds by induction by checking for $k=1,2$. This is an easy computation. 

Thus, $\bE_k$ and $\hat{\bE}_k$ are quasi-perfect Diophantine classes, and (i) and (ii) follow by \eqref{eqn:cboundmu} and \eqref{eqn:mupq}. 

\end{proof}

\begin{rmk}\label{rmk:echlbs} 
As in the proof of Lemma \ref{lem:Eck}, if $\bE=(d,e,p,q,t)$ and $\Lam_\bE$ represents the convex lattice path with corners the origin, $(0,e), (d,e)$, and $(d,0)$, then with
\[
k=\calL(\Lam_\bE)=\frac{(p+1)(q+1)}{2}-1=(d+1)(e+1)-1,
\]
we have
\[
c_\beta\left(\frac{p}{q}\right)\geq\frac{N_k(1,p/q)}{\ell_{\Omega_\beta}(\Lam_\bE)}=\frac{p}{d+e\beta}.
\]
Thus to prove Propositions \ref{prop:oc}~(i) and \ref{prop:oc}~(ii) it would also suffice to simply identify the lattice paths $\Lam_{\bE_k}$ and $\Lam_{\hat\bE_k}$.
\end{rmk}

Next we prove that the centers of the quasi-perfect Diophantine classes $\bE_k$ and $\hat\bE_k$ are arranged as depicted in Figure \ref{fig:main}.

\begin{lemma}\label{lem:alternate}
\begin{enumerate}[label=(\roman*)]
    \item The centers of the classes $\bE_k$ and $\hat \bE_k$ alternate:
    \[
    \cdots<\frac{p_k}{q_k}<\frac{\hat p_{k+1}}{\hat q_{k+1}}<\frac{p_{k+1}}{q_{k+1}}<\cdots
    \]
    \item The obstructions from the classes $\bE_k$ and $\hat\bE_k$ alternate:
    \[
    \cdots<\frac{p_k}{d_k+e_k\beta}<\frac{\hat p_{k+1}}{\hat d_{k+1}+\hat e_{k+1}\beta}<\frac{p_{k+1}}{d_{k+1}+e_{k+1}\beta}<\cdots
    \]
\end{enumerate}
\end{lemma}
\begin{proof}
Our goal is to show
\begin{equation}\label{eqn:alternategoal}
\frac{p_k}{q_k}<\frac{\hat p_{k+1}}{\hat q_{k+1}}<\frac{p_{k+1}}{q_{k+1}}.
\end{equation}

The first inequality in \eqref{eqn:alternategoal} is equivalent to
\begin{align*}
    p_k\hat q_{k+1}&<q_k\hat p_{k+1}
    \\p_k(t_kq_{k+1}-5)&<q_k(t_kp_{k+1}-41)
    \\t_kp_kq_{k+1}-5p_k&<t_kp_{k+1}q_k-41q_k,
\end{align*}
which follows if we can show that
\begin{equation}\label{eqn:altfinal}
    \frac{p_k}{q_k}<\frac{p_{k+1}}{q_{k+1}}\quad \mbox{and}  \quad \frac{p_k}{q_k}<\frac{41}{5}.
\end{equation}

Similarly, the second inequality in \eqref{eqn:alternategoal} is equivalent to
\begin{align*}
    \hat p_{k+1}q_{k+1}&<p_{k+1}\hat q_{k+1}
    \\(t_kp_{k+1}-41)q_{k+1}&<p_{k+1}(t_kq_{k+1}-5)
    \\t_kp_{k+1}q_{k+1}-41q_{k+1}&<t_kp_{k+1}q_{k+1}-5p_{k+1},
\end{align*}
which follows from the second inequality in \eqref{eqn:altfinal}.

The first inequality in \eqref{eqn:altfinal} is
\[
p_k(22q_k-q_{k-1})<q_k(22p_k-p_{k-1}) \iff p_{k-1}q_k<p_kq_{k-1},
\]
thus follows by induction and the base case $k=1$:
\[
\frac{p_0}{q_0}=7,\quad \frac{p_1}{q_1}=\frac{155}{19}\approx 8.158.
\]

The second inequality in \eqref{eqn:altfinal} is equivalent to a linear inequality in $p_k, q_k$, which holds because they both satisfy the same recursion and it holds for $k=0$: 
$$
5p_0<41q_0 \iff 5\cdot 7<41\cdot 1).
$$

To prove (ii), notice that
\[
\hat d_{k+1}+\hat e_{k+1}\beta=t_k(d_{k+1}+e_{k+1}\beta)+17+6\beta,
\]
thus by the same logic as in the proof of (i), all we need to show is
\begin{equation}\label{eqn:iialtfinal}
    \frac{p_k}{d_k+e_k\beta}<\frac{p_{k+1}}{d_{k+1}+e_{k+1}\beta},\quad \frac{p_k}{d_k+e_k\beta}<\frac{41}{17+6\beta}.
\end{equation}

The first inequality in \eqref{eqn:iialtfinal} is
\[
p_k(22d_k-d_{k-1}+22e_k\beta-e_{k-1}\beta)<(22p_k-p_{k-1})(d_k+e_k\beta) \iff \frac{p_{k-1}}{d_{k-1}+e_{k-1}\beta}<\frac{p_k}{d_k+e_k\beta},
\]
which follows by induction and the base case $k=1$:
\[
\frac{7}{3+\beta}<\frac{155}{64+23\beta} \iff 6\beta<17,
\]
which holds because $6\beta\approx16.693$.

The second inequality in \eqref{eqn:iialtfinal} is equivalent to a linear inequality in terms satisfying the same recursion, thus we simply need to check it for $k=0$:
\[
\frac{7}{3+\beta}\approx1.211<1.217\approx\frac{41}{17+6\beta}.
\]
\end{proof}

\begin{proof} (of Theorem \ref{thm:main} (ii), assuming (i))

As above, let $\{\frac{p_k}{q_k}\}$ be the sequence of rational numbers described by the recursion with seeds $p_0=7, q_0=1$ and $p_1=155, q_1=19$ and relation

\[
p_{k}=22p_{k-1}-p_{k-2}, \ q_{k}=22q_{k-1}-q_{k-2}
\]

\noindent for $k\geq2$. We prove that this sequence coincides with the sequence of continued fractions of the form
\[
\left[8,6,4,2,\frac{u_{k-2}}{v_{k-2}}\right]=:\frac{u_k}{v_k}
\]

\noindent for all $k\geq2$. Here, we assume that the seeds of both recursions are equal, so $u_j=p_j$ and $v_j=q_j$ for $j=0,1$.

To prove this equality, we use the following standard result of number theory, which is explained in Chapter 2.1 of \cite{ha}.

\begin{lemma}\label{lem: contfrac}
Let the continued fraction for a real number $\alpha$ be $[a_0,a_1,a_2,...]$. If $\{\frac{r_n}{s_n}\}$ denotes the sequence of convergents of $\alpha$ obtained by truncating this continued fraction expansion, then for any real number $z$,

\[
[a_0,a_1,a_2,...,a_n,z]=\frac{z r_n+r_{n-1}}{z s_n+s_{n-1}}.
\]

\noindent Furthermore, $r_{n+1}s_n-r_ns_{n+1}=(-1)^n$ for all $n\geq0$.
\end{lemma}

The number $\alpha=\frac{54+11\sqrt{30}}{14}=\acc(\beta)$ has the $4$-periodic continued fraction $$\alpha=[\{8,6,4,2\}^{\infty}].$$ The numerators and denominators of the second and third convergents of this continued fraction are $r_2=204,s_2=25,r_3=457,$ and $s_3=56$. Lemma \ref{lem: contfrac} combined with our recurrence relation yields

\[
\frac{u_k}{v_k}=\left[8,6,4,2,\frac{u_{k-2}}{v_{k-2}}\right]=\frac{r_3u_{k-2}+r_2v_{k-2}}{s_3u_{k-2}+s_2v_{k-2}}
\]

\noindent for all $k\geq2$. This can also be written using matrix notation:

\begin{equation}\label{eqn:matrix}
\begin{pmatrix}
u_k \\
v_k
\end{pmatrix}
=
\begin{pmatrix}
r_3 & r_2 \\
s_3 & s_2
\end{pmatrix}
\begin{pmatrix}
u_{k-2} \\
v_{k-2}.
\end{pmatrix}
\end{equation}

We assume $x_j=22x_{j-1}-x_{j-2}$ for $j<k$ and $x_j=u_j, v_j$. By \eqref{eqn:matrix}, we have
\begin{align*}
    u_k&=22u_{k-1}-u_{k-2}
    \\r_3u_{k-2}+r_2v_{k-2}&=22(r_3u_{k-3}+r_2v_{k-3})-(r_3u_{k-4}+r_2v_{k-4}),
\end{align*}
which follows from the inductive hypothesis. Similarly, we may obtain $v_k=22v_{k-1}-v_{k-2}$ from $v_k=s_3u_{k-2}+s_2v_{k-2}$.

\noindent Thus, the sequence of rational numbers $\{\frac{u_k}{v_k}\}$ is determined by the same seeds and the same recurrence relation as the sequence $\{\frac{p_k}{q_k}\}$, as claimed.
\end{proof}

We have also shown:
\begin{cor}\label{cor:pqacc} The limit of the outer corners is
\[
\lim_{k \to \infty}\frac{p_k}{q_k}=\acc(\beta).
\]
\end{cor}
\begin{proof}
    The continued fractions of the ratios $p_k/q_k$ converge to the continued fraction of $\alpha=[\{8,6,4,2\}^\infty]=\acc(\beta)$.
\end{proof}

Corollary \ref{cor:pqacc} may also be proved by solving the recursion in Definition \ref{def:outerclasses}, see \cite[Prop.~49]{ICERM}, however we do not do this here.

\subsection{Inner corners}\label{ssec:inner}

We describe a family of mutations whose existence proves Proposition \ref{prop:ic1} and explains the reasoning behind Conjecture \ref{conj:ic2}. Throughout we will freely use the conventions discussed in Remark~\ref{rmk:labeling}.

The following definition is a version of \cite[Def.~2.1.1]{MMW}. It describes algebraic relations between various classes, which later will be helpful in showing various identities hold that arise in the ATF proofs. 

 \begin{definition}\label{def:Tt}    Two quasi-perfect classes $\bE : = (d,e,p,q,t), \bE': =(d',e',p',q',t')$ 
 are said to be
 {\bf adjacent} if after renaming so that $p/q < p'/q'$ (if necessary), 
 the following relation  holds:
   %\begin{align}\label{eq:adjac}
 \[
 (p+q)(p'+q') - tt' = 8 pq'.
 \]
%\end{align}
Further,  they are called
 {\bf $t''$-compatible} if
  %\begin{align}\label{eq:tcompat}
  \[
tt'-4t''= p p' - 3(pq'+qp') + qq', \qquad \mbox{ i.e. }\ \  \bx^T A \bx' = 4t''.
\]
%\end{align} 
%We make corresponding definitions for a pair
% $\bx, \bx' \in X_\Z$ to be $t''$-compatible or  adjacent.
%Further we write $ \bx<\bx'$ (or $\bE< \bE'$) to denote that\footnote
%{
%Note that this condition has nothing to do with the relative size  of the ratios $p/q, p'/q'$.}
% $p<p',  q<q', t<t'$; and
% say that $\bx\in X_\Z$ is {\bf positive} if %$p,q, t>0$.
 \end{definition}

The following lemma is from \cite[Lem.~2.1.2]{MMW} about $t$-compatibility and adjacency. It proves how compatibility and adjacency hold throughout a recursive sequence. Note that this proof did not use the $(d,m)$ coordinates used in \cite{MMW} and just uses the $(p,q,t)$ coordinates, and thus, the lemma holds for our classes here.

  \begin{lemma}\label{lem:recur0} \begin{itemlist}\item[{\rm (i)}]
 Suppose that the points $\bx_0:=(p_0,q_0,t_0), \bx_1:=(p_1,q_1,t_1)$ are $t$-compatible for some $t\ge 3$ and  have coordinate by coordinate $ \bx_0<\bx_1$.
Then  $\bx_2: = t\bx_1 -\bx_0 \geq 0$.  Also, $\bx_1<\bx_2$ and the pair $\bx_1,\bx_2$  is $t$-compatible.
Further, if $\bx_0, \bx_1$ are adjacent, so are $\bx_1, \bx_2$.
Thus,  if  $\bE_0, \bE_1$ satisfy $p_0<p_1,q_0<q_1, t_0<t_1$ and are adjacent and  $t$-compatible, then so are the components of all successive pairs in the sequence obtained from $\bE_0, \bE_1$ by $t$-recursion.

\item[{\rm (ii)}]   If $\bE, \bE'$ are adjacent, then they are $t''$-compatible exactly if
 $$
| p'q-pq'| =t''.
 $$
 \end{itemlist}
\end{lemma}

% \begin{lemma} \label{lem:compatadj}
% The following classes are compatible:
% \begin{itemlist}
%     \item[{\rm(i)}] $\bE_k$ and $\bE_{k+1}$ are 22-compatible.
%     \item[{\rm(ii)}] $\bE_k$ and $\bE$ are $t_{k-1}$-compatible.
%     \item[{\rm(iii)}] $\hat{\bE}_{k+1}$ and $\bE_{k+1}$ are $t_k$-compatible.
%     \item[{\rm(iv)}] $\bE_k$ and $\hat{\bE}_{k+1}$ are $t_{k+1}$-compatible. 
% \end{itemlist}
% The following classes are adjacent: 
% \begin{itemlist}
%     \item[{\rm(i)}] $\bE_k$ and $\bE_{k+1}$ are adjacent.
%     \item[{\rm(ii)}] $\bE_k$ and $\bE$ are adjacent.
%     \item[{\rm(iii)}] $\hat{\bE}_{k+1}$ and $\bE_{k+1}$ are adjacent.
%     \item[{\rm(iv)}] $\bE_k$ and $\hat{\bE}_{k+1}$ are adjacent. 
% \end{itemlist}
% \end{lemma}
% \begin{proof}
% The proof follows by induction, using Lemma \ref{lem:recur0} (i).

% For the cases without $\hat{\bE}_k$ mentioned, we can check these for the cases when $k=0,k=1$. Then, by Lemma~\ref{lem:recur0}, we conclude that these conditions continue to hold for large $k.$

% For the cases with $\hat{\bE}_k$ mentioned the results follow by \cite[Prop.~2.1.9]{MMW}. Again, note in the proof of this proposition, the $(d;m)$ coordinates used in \cite{MMW} are not needed, and only the properties of $(p,q,t)$ were needed which are the same coordinates we are using here. 
% \end{proof}

Recall that $\bE=(d,e,p,q,t)=(17,6,41,5,22).$ 
The following lemma comes from \cite[Lem.~4.6]{M1}.
\begin{lemma}\label{lem:identities}
For the classes $\bE_\la,\bE_\mu,\bE_\rho$ where we have $(\bE_\la,\bE_\mu,\bE_\rho):=(\bE_k,\bE_{k+1},\bE)$ or $(\bE_\la,\bE_\mu,\bE_\rho):=(\bE_k,\hat{\bE}_{k+1},\bE_{k+1})$, the following identities hold:
\begin{itemize}  \item[{\rm (i)}] $p_\la+q_\la=q_\mu t_\rho-q_\rho t_\mu$ and $7p_\la-q_\la=p_\mu t_\rho-t_\mu p_\rho$
\item[{\rm (ii)}] $p_\rho+q_\rho=p_\mu t_\la-p_\la t_\mu$ and $p_\rho-7q_\rho=q_\la t_\mu-q_\mu t_\la$
\item[{\rm (iii)}] $p_\mu+q_\mu=q_\rho t_\la+p_\la t_\rho$, $7p_\mu-q_\mu=6 p_\la t_\rho+p_\rho t_\la-q_\la t_\rho$, and \newline $7q_\mu-p_\mu=6q_\rho t_\la+q_\la t_\rho-p_\rho t_\la$
%\item[{\rm (vii)}] $p_\la p_\rho-6 p_\la q_\rho+q_\la q_\rho=p_\la(p_\rho-6q_\rho)+q_\la q_\rho=t_\mu$
\item[{\rm (iv)}] $p_\la(p_\rho-6q_\rho)+q_\la q_\rho=t_\mu$
\item[{\rm (v)}] $q_\la t_\la+q_\rho t_\rho+q_\mu t_\mu=q_\mu t_\la t_\rho$
\item[{\rm (vi)}] $t_\la\begin{pmatrix} 1+p_\mu q_\mu-6q_\mu^2 \\ q_\mu^2 \end{pmatrix}=q_{x\mu}\begin{pmatrix} p_\mu-6q_\mu \\ q_\mu \end{pmatrix}+q_\mu \begin{pmatrix} p_\rho-6q_\rho \\ q_\rho \end{pmatrix}$ 
\item[{\rm (vii)}]\label{item:XV} $-t_\rho \begin{pmatrix} -q_\mu^2 \\ q_\mu p_\mu-1 \end{pmatrix}=q_{y\mu}\begin{pmatrix} q_\mu \\ -p_\mu \end{pmatrix}+q_\mu \begin{pmatrix} q_\la \\ -p_\la \end{pmatrix}$
\end{itemize}
\end{lemma}
\begin{proof}
These identities are a reformulation of the recursion 
compatibility and adjacency equations proven in \cite[Lem.~4.6]{M1} using the facts that:
\begin{itemize}
    \item $\bE_\la$ and $\bE_\mu$ are $t_\rho$-compatible and adjacent
    \item $\bE_\rho$ and $\bE_\mu$ are $t_\la$-compatible and adjacent,
\end{itemize}
which can be proved for the triple $(\bE_k,\bE_{k+1},\bE)$ by induction using Lemma \ref{lem:recur0} and for the triple $(\bE_k,\hat\bE_{k+1},\bE_{k+1})$ using \cite[Prop.~2.1.9]{MMW}. Note that in the proof of this proposition, the $(d;m)$ coordinates used in \cite{MMW} are not needed, and only the properties of $(p,q,t)$ were needed which are the same coordinates we are using here.
%just the assumptions of Lemma~\ref{lem:compatadj}.  
\end{proof}

The first sequence of mutations we consider is $v^2yxy^k$. Note, we found this sequence by adapting an analogous case found by Magill in \cite[Prop.~3.9]{M1}.\footnote{In \cite{M1}, the mutation sequences are instead written from right to left.} We show in Proposition \ref{prop:fullacc} that this sequence gives embeddings $(1-\eps)\cdot E(1,z_k)\sembeds P(\vol(\beta),\vol(\beta)\beta)$\footnote{Recall that $\vol(\beta)=\vol_\beta(\acc(\beta))$.} for a sequence $z_k$ such that $\lim_{k \to \infty} z_k=\acc(\beta)$. Each of the points $(z_k,\vol(\beta))$ lie strictly above the embedding function.

Then, for each $k$, we will perform several additional mutations that provide embeddings $(1-\eps)\cdot E(1,z)\sembeds P(\lambda,\lambda \beta)$ where $(z,\lambda)$ does lie on the graph of the embedding function: specifically, at the inner corners between the obstructions from $\bE_k$ and $\hat\bE_{k+1}$, proving Proposition \ref{prop:ic1}.

The effects of the successive mutations in the sequence $v^2y^2$ are illustrated in Figure \ref{fig:vvyx}.% and \ref{fig:xmut}.

\begin{figure}
    \subfigure[$v$]{
         \centering
         \includegraphics[height=2.5cm]{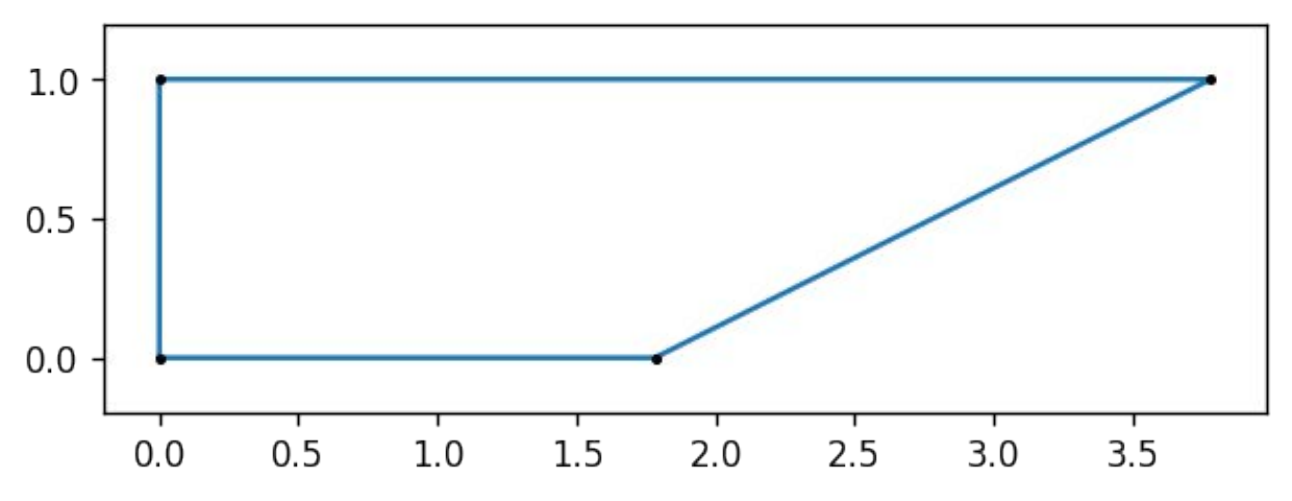}
     }\hfil
    \subfigure[$v^2$]{
         \centering
         \includegraphics[height=2.45cm]{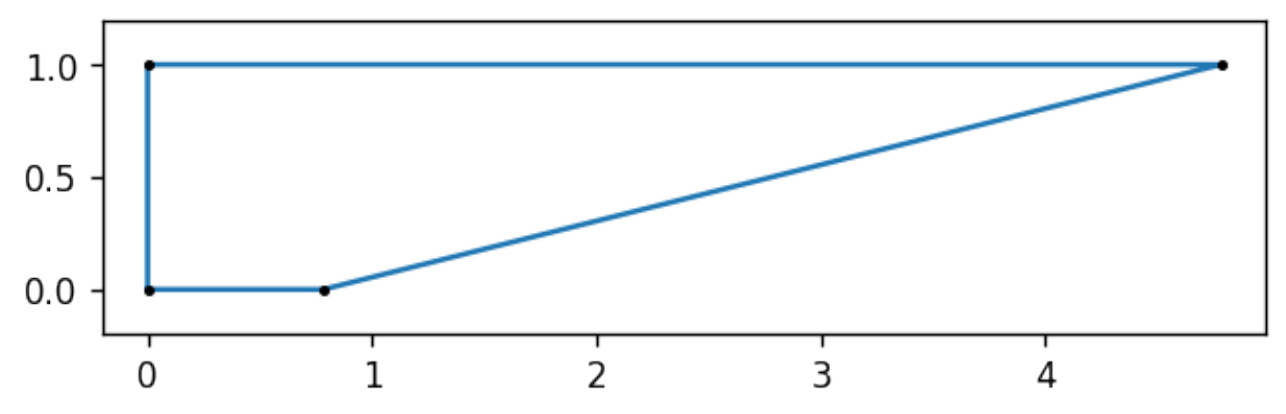}
         }
     \hfil
     \subfigure[$v^2y$]{
         \centering
         \includegraphics[height=1.8cm]{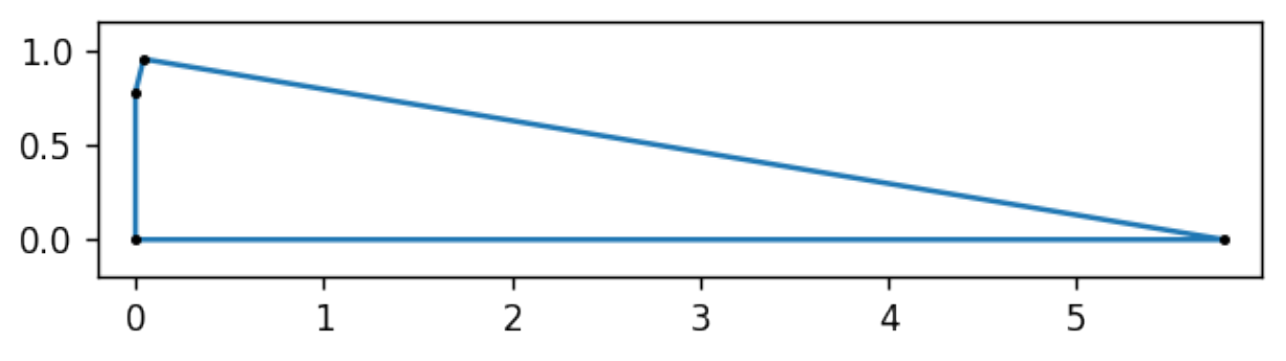}
     }\hfil
     \subfigure[$v^2y^2$]{
         \centering
         \includegraphics[height=1.8cm]{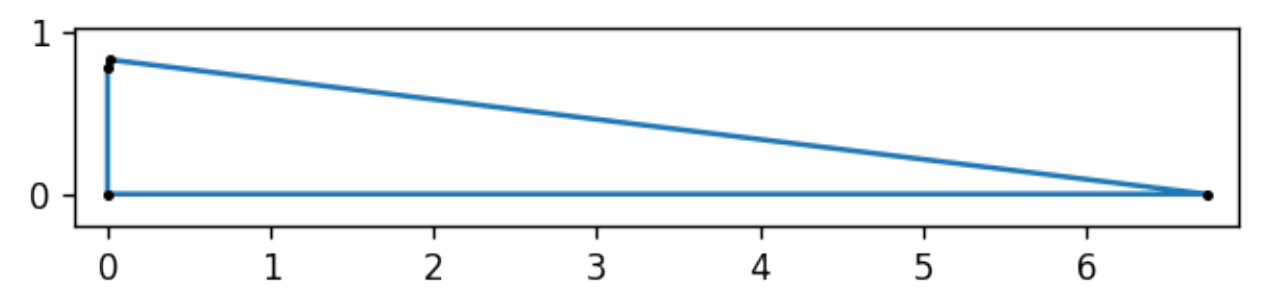}
     }\hfil
     \caption{An illustrative example of the mutation sequence $v^2y^2$, where each figure represents the polygon $\Omega_\beta$ after one step of mutation. Figure $(c)$ and $(d)$ have their axes reflected: the correct figures are the ones displayed with $z$ and $\lambda$ switched. Already it is clear that more mutations by $y$ would cut the edge $XV$ shorter and shorter. We do not include the mutation by $x$ here, even though the actual sequence considered is $v^2yxy$, because its effect would be very difficult to see at this scale.}
     \label{fig:vvyx}
\end{figure}

We will frequently use the following simplification of $\vol(\beta)$.
\begin{lemma}\label{lem:volbaccb} We have the relation
    \[
    \frac{1}{\vol(\beta)}=-1+\frac{\sqrt{30}}{3}=\frac{4\beta-7}{5}.
    \]
\end{lemma}
\begin{proof}
    Following the method of proof of \cite[Lem.~2.2.7]{symm} and replacing $3-b$ with the affine perimeter in our case, which is $2+2\beta$, for any value of $\beta$ we have
    \[
    \vol(\beta)=\frac{1+\acc(\beta)}{2+2\beta}.
    \]
    With $\beta=(6+5\sqrt{30})/12$ and $\acc(\beta)=\frac{54+11\sqrt{30}}{14}$, we simplify
    \[
    \frac{1}{\vol(\beta)}=\frac{2+2\frac{6+5\sqrt{30}}{12}}{1+\frac{54+11\sqrt{30}}{14}}=-1+\frac{\sqrt{30}}{3}.
    \]
    This proves the first equality. The second is a simple computation.
\end{proof}
\begin{rmk}\label{rmk:d'e'} 

% Given a quasi-perfect Diophantine class $\bE$, we define
% \[
% d'=2q-d,\quad e'=2q-e.
% \]
    % Notice that 
The conclusion of Lemma \ref{lem:volbaccb} is similar to \cite[Lem.~5.1~(iii)]{M1} where we find that in the case of the Hirzebruch surface, if to the right of the accumulation the function $c_\beta(z)$ is given by a class $\bE=(d,m,p,q,t)$, then
\[ \vol(b)=\frac{q}{(m-q)b-(d-3q)},\]
where $\vol(b)=\vol_b(\acc_H(b))$, noting that the volume obstruction $\vol_b(z)$ has the formula $\sqrt{z/(1-b^2)}$ when the target is $H_b$. In our case, with $\bE=(17,6,41,5,22)$, $\beta=(6+5\sqrt{30})/12$, and our definition of $\vol$, we have 
%$d'=10-17$ and $e'=10-6$, thus
\begin{equation}\label{eqn:qe'd'}
\vol(\beta)=\frac{5}{4\beta-7}=\frac{q}{(2q-e)\beta+(2q-d)}.
\end{equation}

\end{rmk}

We compute the result of the first four mutations, illustrated in Figure \ref{fig:vvyx}.
\begin{lemma}\label{lem:vvyx} After performing the sequence of mutations $v^2yx$ to the diagram $\Omega_\beta$, the nodal rays are
\[
\vec{n}_Y=\begin{pmatrix}1\\-7\end{pmatrix},\quad \vec{n}_V=\begin{pmatrix}-3\\-1\end{pmatrix},\quad \vec{n}_X=\begin{pmatrix}11\\5\end{pmatrix},
\]
the direction vectors are
\[
\overrightarrow{OY}=\begin{pmatrix}0\\1\end{pmatrix},\quad \overrightarrow{OX}=\begin{pmatrix}1\\0\end{pmatrix},\quad \overrightarrow{YV}=\begin{pmatrix}1\\-6\end{pmatrix},\quad \overrightarrow{XV}=\begin{pmatrix}56\\25\end{pmatrix},
\]
and the affine lengths are
\[
|OY|=3+\beta,\quad |OX|=\frac{1}{\vol(\beta)},\quad |YV|=\frac{7+4\beta}{19},\quad |XV|=\frac{3-\beta}{95}.
\]
\end{lemma}

\begin{proof}
The diagram $\Omega_\beta$ has nodal rays
\[
\vec{n}_Y=\begin{pmatrix}1\\-1\end{pmatrix},\quad \vec{n}_V=\begin{pmatrix}-1\\-1\end{pmatrix},\quad \vec{n}_X=\begin{pmatrix}-1\\1\end{pmatrix},
\]
direction vectors
\[
\overrightarrow{OY}=\overrightarrow{XV}=\begin{pmatrix}0\\1\end{pmatrix},\quad \overrightarrow{OX}=\overrightarrow{YV}=\begin{pmatrix}1\\0\end{pmatrix},
\]
and affine lengths
\[
|OY|=|XV|=1,\quad |OX|=|YV|=\beta.
\]

\noindent\textit{Step 1: first mutation at $V$.} The nodal ray $\vec{n}_V$ hits the side $OX$ at $(\beta-1,0)$, giving us the affine lengths
\[
|OY_v|=|X_vV_v|=1,\quad |OX_v|=\beta-1,\quad |Y_vV_v|=\beta+1.
\]

The matrix $M$ for mutation at $V$ must satisfy
\[
M\vec{n}_V=\vec{n}_V,\quad M\overrightarrow{VX}=\overrightarrow{YV} \iff M=\begin{pmatrix}2&-1\\1&0\end{pmatrix}.
\]
Thus the result of a $V$-mutation has nodal rays
\[
\vec{n}_{Y_v}=\vec{n}_Y=\begin{pmatrix}1\\-1\end{pmatrix},\quad \vec{n}_{V_v}=M\vec{n}_X=\begin{pmatrix}-3\\-1\end{pmatrix},\quad \vec{n}_{X_v}=-\vec{n}_V=\begin{pmatrix}1\\1\end{pmatrix}.
\]

The unchanged direction vectors are
\[
\overrightarrow{OY_v}=\begin{pmatrix}0\\1\end{pmatrix},\quad \overrightarrow{OX_v}=\overrightarrow{Y_vV_v}=\begin{pmatrix}1\\0\end{pmatrix},
\]
and
\[
    \overrightarrow{X_vV_v}=M\overrightarrow{OX}=\begin{pmatrix}2\\1\end{pmatrix}.
\]

\noindent\textit{Step 2: second mutation at $V$.} We now replace each result $A_v$ of Step 1 with $A$ so that we do not have to stack subscripts. The nodal ray $\vec{n}_V=(-3,-1)$ hits the side $OX$ at $(b+1,1)+(-3,-1)=(b-2,0)$, giving us the affine lengths
\[
|OY_v|=|X_vV_v|=1,\quad |OX_v|=\beta-2,\quad |Y_vV_v|=\beta+2.
\]

The mutation matrix $M$ must satisfy
\[
M\vec{n}_V=\vec{n}_V,\quad M\overrightarrow{VX}=\overrightarrow{YV} \iff M=\begin{pmatrix}4&-9\\1&-2\end{pmatrix}.
\]
Thus the nodal rays are
\[
\vec{n}_{Y_v}=\vec{n}_Y=\begin{pmatrix}1\\-1\end{pmatrix},\quad \vec{n}_{V_v}=M\vec{n}_X=\begin{pmatrix}-5\\-1\end{pmatrix},\quad \vec{n}_{X_v}=-\vec{n}_V=\begin{pmatrix}3\\1\end{pmatrix}.
\]

The unchanged direction vectors are
\[
\overrightarrow{OY_v}=\begin{pmatrix}0\\1\end{pmatrix},\quad \overrightarrow{OX_v}=\overrightarrow{Y_vV_v}=\begin{pmatrix}1\\0\end{pmatrix},
\]
and
\[
\overrightarrow{X_vV_v}=M\overrightarrow{OX}=\begin{pmatrix}4\\1\end{pmatrix}.
\]

\noindent\textit{Step 3: mutation at $Y$.} Again, we replace $A_v$ with $A$. The nodal ray $\vec{n}_Y$ hits the side $XV$, because its $x$-intercept is at $(1,0)$ and $\beta-2<1$. The mutation matrix $M$ must satisfy
\[
M\vec{n}_Y=\vec{n}_Y,\quad M\overrightarrow{YV}=\overrightarrow{OY} \iff M=\begin{pmatrix}0&-1\\1&2\end{pmatrix}.
\]
Thus the nodal rays are
\[
\vec{n}_{Y_y}=M\vec{n}_V=\begin{pmatrix}1\\-7\end{pmatrix},\quad \vec{n}_{V_y}=-\vec{n}_Y=\begin{pmatrix}-1\\1\end{pmatrix},\quad \vec{n}_{X_y}=\vec{n}_X=\begin{pmatrix}3\\1\end{pmatrix}.
\]

We know the affine lengths
\[
|OY_y|=|OY|+|YV|=\beta+3,\quad |OX_y|=|OX|=\beta-2,
\]
and the unchanged direction vectors
\[
\overrightarrow{OY_y}=\begin{pmatrix}0\\1\end{pmatrix},\quad \overrightarrow{OX_y}=\overrightarrow{OX}=\begin{pmatrix}1\\0\end{pmatrix},\quad \overrightarrow{X_yV_y}=\overrightarrow{XV_y}=\begin{pmatrix}4\\1\end{pmatrix}.
\]
Furthermore,
\[
\overrightarrow{Y_yV_y}=M\overrightarrow{VX}=M\begin{pmatrix}-4\\-1\end{pmatrix}=\begin{pmatrix}1\\-6\end{pmatrix}.
\]

Finally, we solve \eqref{eqn:4addzero} with $a=y$ to obtain
\[
|Y_vV_v|=\frac{\beta+2}{5},\quad |X_yV_y|=\frac{3-\beta}{5};
\]
note that $|Y_yV_y|+|X_yV_y|=|XV|=1$.

% Next we add the sides of the mutated triangle, where $Y+s\vec{n}_Y$ is on the side $XV$, to obtain
% \[
% 0=|YV|\overrightarrow{YV}-|Y_yV_v|\overrightarrow{XV}-s\vec{n}_Y \iff |Y_vV_v|=\frac{b+2}{5}.
% \]
% Moreover,
% \[
% |Y_yV_y|+|X_yV_y|=|XV|=1 \iff |X_yV_y|=\frac{3-b}{5}.
% \]

\noindent\textit{Step 4: mutation at $X$.} We replace $A_y$ with $A$. The nodal ray $\vec{n}_X$ hits the side $YV$ because it has positive slope. The mutation matrix $M$ must satisfy
\[
M\vec{n}_X=\vec{n}_X,\quad M\overrightarrow{XV}=\overrightarrow{OX} \iff M=\begin{pmatrix}-2&9\\-1&4\end{pmatrix},
\]
thus the nodal rays are
\[
\vec{n}_{Y_x}=\vec{n}_Y=\begin{pmatrix}1\\-7\end{pmatrix},\quad \vec{n}_{V_x}=-\vec{n}_X=\begin{pmatrix}-3\\-1\end{pmatrix},\quad \vec{n}_{X_x}=M\vec{n}_V=\begin{pmatrix}11\\5\end{pmatrix}.
\]

We know the affine lengths
\[
|OY_x|=|OY|=\beta+3,\quad |OX_x|=|OX|+|XV|=\beta-2+\frac{3-\beta}{5}=\frac{4\beta-7}{5},
\]
and the unchanged direction vectors
\[
\overrightarrow{OY_x}=\overrightarrow{OY}=\begin{pmatrix}0\\1\end{pmatrix},\quad \overrightarrow{Y_xV_x}=\overrightarrow{YV_x}=\begin{pmatrix}1\\-6\end{pmatrix},\quad \overrightarrow{OX_x}=\begin{pmatrix}1\\0\end{pmatrix}.
\]
Furthermore,
\[
\overrightarrow{X_xV_x}=M\overrightarrow{VY}=M\begin{pmatrix}-1\\6\end{pmatrix}=\begin{pmatrix}56\\25\end{pmatrix}.
\]

Finally, we solve \eqref{eqn:4addzero} with $a=x$ to obtain
\[
|Y_xV_x|=\frac{7+4\beta}{19},\quad |X_xV_x|=\frac{3-\beta}{95}.
\]

% Again we add the sides of the mutated triangle, where $X+s\vec{n}_X$ is on the side $YV$:
% \[
% 0=|XV|\overrightarrow{XV}-|X_xV_x|\overrightarrow{YV}-s\vec{n}_X \iff |X_xV_x|=\frac{3-b}{95}.
% \]
% Thus
% \[
% |Y_xV_x|+|X_xV_x|=|YV|=\frac{b+2}{5} \iff |Y_xV_x|=\frac{7+4b}{19}.
% \]
\end{proof}

The next lemma allows us to compute the effect of $k$ additional mutations at the corner $Y$.
\begin{lemma} \label{lem:correctside}
While performing the sequence $v^2yxy^k$, for each of the final $y$ mutations, the nodal ray $\vec{n}_Y$ always intersects the side $\overrightarrow{XV}$.
\end{lemma}
\begin{proof} 
%We prove that the mutations of the $P(1,\frac{6+5\sqrt{30}}{12})$ polydisk along the vector situated at the $y$-axis behave as expected to achieve a full filling of $P(1,\frac{6+5\sqrt{30}}{12})$ by the ellipsoid $E(\frac{-3+\sqrt{30}}{3},\frac{8+\sqrt{30}}{2})$ in the limit. After two mutations of the rectangle with side lengths $1$ and $\frac{6+5\sqrt{30}}{12}$ along the second vertex, one mutation along the first vertex, and one along the third vertex (where we label vertices starting from the one on the $y$-axis and proceeding clockwise),
We first compute the exact corners of the effect of the sequence $v^2yx$ applied to $\Omega_\beta$. From Lemma \ref{lem:vvyx} we obtain the vertices $X=(1/\vol(\beta),0)=(-1+\sqrt{30}/3,0)$ and $Y=(0,\beta+3)=(0,(42+5\sqrt{30})/12)$. %$(0,\frac{42+5\sqrt{30}}{12})$.
% For the vertex $V$:
% \[
% V=\begin{pmatrix}0\\3+b\end{pmatrix}+|YV|\overrightarrow{YV}=\begin{pmatrix}0\\3+b\end{pmatrix}+\frac{7+4b}{19}\begin{pmatrix}1\\-6\end{pmatrix}=\begin{pmatrix}\frac{7+4b}{19}\\\frac{5(3-b)}{19}\end{pmatrix}=\begin{pmatrix}\frac{25+5\sqrt{30}}{57}\\\frac{25(6-\sqrt{30})}{228}\end{pmatrix}.
% \]

Let $h_k$ denote the height of the quadrilateral along the $y$-axis after the mutation sequence $v^2yxy^k$, with $h_0=(42+5\sqrt{30})/12$ as above. Likewise, let $(x_k,y_k)$ denote the vector $\vec{n}_Y$ at the vertex $(0,h_k)$ along which we are mutating, with $(x_0,y_0)=(1,-7)$ by Lemma \ref{lem:vvyx}, and let $t_k$ denote the $x$-coordinate of the intersection point of the line through the point $(0,h_k)$ in the direction of the vector $(x_k,y_k)$ with the line through the point $(-1+\sqrt{30}/3,0)$ with slope $\frac{25}{56}$. In terms of $h_k, x_k,$ and $y_k$, $t_k$ is given by the formula

\begin{equation}\label{eq1}
t_k=\frac{x_k(168h_k+25(\sqrt{30}-3))}{3(25x_k-56y_k)}.
\end{equation}

By definition of mutation, $h_k>h_{k-1}$ for all $k$. Assume by induction that $\vec{n}_Y$ intersects $\overrightarrow{XV}$ for the first $k-1$ mutations by $y$. Letting $V_j$ denote the vertex $V$ after the mutation sequence $v^2yxy^{j-1}$, our inductive hypothesis implies that $V_j$ has both $x$ and $y$ coordinates less than $V_{j-1}$ if $j<k$, so if $V_j=(x_v,y_v)$,
\begin{equation}\label{eq2}
    \frac{y_k}{x_k}<\frac{h_k-y_v}{x_v}<%\frac{h_0-\frac{7+4b}{19}}{\frac{5(3-b)}{19}}=
    -6,
\end{equation}
which is the slope of the initial side $\overrightarrow{YV}$.

Assume by way of contradiction that
\[
t_k\leq-1+\frac{\sqrt{30}}{3}.
\]
% \noindent As before, note that $x_k>0$ and $y_k<0$ for all $k\in\mathbb{Z}_{\geq0}$ and that the inequality
% \begin{equation}\label{eq2}
% \frac{y_k}{x_k}<-6
% \end{equation}
%\noindent holds for all nonnegative integers $n$, since the slope of the initial upper right side of the quadrilateral is $-6$. Using these facts, we prove the following lemma.
% \begin{lemma}
% For any $n\in\mathbb{Z}_{\geq0}$, if $t_n\leq\frac{-3+\sqrt{30}}{3}$, then $h_n<2(\sqrt{30}-3)$.
% \end{lemma}
%\begin{proof}
Using the formula in \eqref{eq1} for $t_k$, we have

\[
\frac{x_k(168h_k+25(\sqrt{30}-3))}{3(25x_k-56y_k)}\leq\frac{-3+\sqrt{30}}{3} \iff h_k\leq\frac{(3-\sqrt{30})y_k}{3x_k}.
\]

% \noindent Rearranging, this inequality is equivalent to

% \[
% h_n\leq\frac{-(\sqrt{30}-3)y_n}{3x_n}.
% \]

\noindent Then, by the inequality \eqref{eq2}, we obtain

\[
h_k\leq\frac{(3-\sqrt{30})y_k}{3x_k}<\frac{(-6)(3-\sqrt{30})}{3}=2(\sqrt{30}-3),
\]

%\noindent as claimed.
%\end{proof}

However, $h_k\geq h_0=\frac{42+5\sqrt{30}}{12}$ for all $k\in\mathbb{Z}_{\geq0}$ and $h_0=\frac{42+5\sqrt{30}}{12}>2(\sqrt{30}-3)$. Thus we have a contradiction.% for all nonnegative integers $k$. The contrapositive of the lemma above then indicates that $t_n>\frac{-3+\sqrt{30}}{3}$ for all $n\in\mathbb{Z}_{\geq0}$, the desired result.
\end{proof}

We now compute the nodal rays and directions of the sides after the mutation sequence $v^2yxy^k$.
\begin{lemma}\label{lem:rays}
After performing the sequence $v^2yxy^k$ the nodal rays are given by:
\[ \vec{n}_Y=\begin{pmatrix}q_k \\ -p_k \end{pmatrix}, \quad  \vec{n}_V=\begin{pmatrix}-q_{k-1} \\ p_{k-1} \end{pmatrix}, \quad \vec{n}_X=\begin{pmatrix} 11 \\ 5 \end{pmatrix}
\]
and the direction vectors are given by: 
\[ \overrightarrow{YV}=\begin{pmatrix} q_k^2 \\ -p_k q_k+1\end{pmatrix}, \quad \overrightarrow{XV}=\begin{pmatrix} 56 \\ 25 \end{pmatrix}\] 
\end{lemma}
\begin{proof}
    We first must check the base case when $k=0$, which was computed in Lemma~\ref{lem:vvyx}. This is seen as by the defining recursion $x_k=22x_{k-1}-x_{k-2}$ we have $p_{-1}=-1$ and $q_{-1}=3$. 
    For the inductive step, we explain how this lemma is equivalent to \cite[Lem.~6.6]{M1}, so follows by the proof there. 

    In \cite[Lem.~6.6]{M1}, the lemma is similarly looking at $y$ mutations to a quadrilateral where the nodal ray intersects the $|XV|\overrightarrow{XV}$ side of the polygon. The lemma assumes that the quadrilateral is defined via a triple $\Tt$ notated as $Q(\Tt)$. Looking at the definition of $Q(\Tt)$ in \cite[Def.~3.8]{M1}, we see that if we set $\bE_\la=\bE_k,$ $\bE_{\mu}=\bE_{k+1}$, and $\bE_\rho=\bE,$ the definition for the nodal rays and direction vectors are the same as the identities we must prove. Further, after assuming which side the nodal ray will hit, checking formulas for nodal rays and direction vectors after a mutation do not depend on the side lengths. Thus, this lemma is equivalent to  \cite[Lem.~6.6]{M1}. The proof uses the identities we already established in Lemma~\ref{lem:identities}.  
\end{proof}

Next we compute the affine lengths of the sides after the mutation sequence $v^2yxy^k$.
\begin{lemma}\label{lem:als}
After performing the sequence $v^2yxy^k$ the affine lengths are given by: 
\begin{align*}
    |OY| &= \frac{d_k+e_k\beta}{q_k}, \quad 
    |OX|= -1 + \frac{\sqrt{30}}{3} =\frac{4\beta-7}{5}=\frac{1}{\vol(\beta)} \\
    |YV| &= \frac{4\beta + 7}{q_k q_{k+1}}, \quad 
    |XV| = \frac{d_k-e_k\beta}{5q_{k+1}}
\end{align*}
\end{lemma}

\begin{proof} 
We will show this by induction on $k$. The base case $k=0$ is proved in Lemma \ref{lem:vvyx}. %recall that when $k = 0$, we have
% \begin{align*}
%     |OY| &= 3 + b = \frac{1}{q_0} (e_0 b + d_0) \\
%     |OX| &= \frac{1}{5}(4 b - 7) \\
%     |YV| &= \frac{1}{19} (4 b + 7) = \frac{1}{q_k q_{k+1}} (4 b + 7) \\
%     |VX| &= \frac{1}{95} (-b + 3) = \frac{1}{5q_{k+1}} (-e_k b + d_k)
% \end{align*}
Suppose that the conclusions hold for $k$. Then, for $k + 1$, by Lemma \ref{lem:correctside}, in performing the consecutive $y$-mutations, the nodal ray will hit the $|XV|$ side of the polygon. Therefore, the $|OX|$ side remains constant.

For $|OY|$, it suffices to show that $|OY_y| = |OY| + |YV|$. By the induction hypothesis, we have
\begin{align*}
      |OY| + |YV| 
    = &\frac{d_k+e_k\beta}{q_{k}} + \frac{4\beta + 7}{q_k q_{k+1}} \\
    = &\frac{(q_{k+1} e_k + 4)\beta + (q_{k+1} d_k + 7)}{q_k q_{k+1}}.
\end{align*}
It remains to show that $q_{k+1} e_k + 4 = q_{k} e_{k+1}$ and $q_{k+1} d_k + 7 = q_{k} d_{k+1}$. Both hold by induction. The base case is easily checked. For the inductive step, we show the details of former and the latter follows similarly. Assuming the equality holds for $k$, by Definition \ref{def:outerclasses}, 
\begin{align*}
    q_{k+1} e_k + 4 &= (22q_{k} - q_{k-1}) e_k + 4 \\
                    &= 22q_k e_k - (q_{k-1} e_k - 4) \\
                    &= 22q_k e_k - q_k e_{k-1} \\
                    &= q_k (22e_k - e_{k-1}) \\
                    &= q_k e_{k+1}
\end{align*}
This completes the proof for $|OY_y|.$

For $|XV_y|$,
we must show
\[ |XV_y|=\frac{d_{k+1}-e_{k+1}\beta}{5q_{k+2}}.\] 
By adding the sides of the quadrilateral that is fixed during mutation, we have the following equality \[ \begin{pmatrix} 0 \\ -|OY| \end{pmatrix} + \begin{pmatrix} |OX| \\ 0 \end{pmatrix} + |XV_y| \begin{pmatrix} 56 \\ 25 \end{pmatrix} =s\begin{pmatrix} q_k \\ -p_k \end{pmatrix}\]
where $s$ is the length of the nodal ray $\vec{n}_Y$ where it intersects the side $XV.$ This equality gives us two equations
\begin{align*}
    |OX|+56|XV_y|&=s q_k \\
    -|OY|+25|XV_y|&=-sp_k.
\end{align*}
We can solve the first one for $s$ and substitute it into the second one to get the equation 
\[ -|OY|+25|XV_y|=-\frac{p_k}{q_k}(|OX|+56|XV_y|).\]
Solving for $|XV_y|$, we get
\[ |XV_y|=\frac{q_k|OY|-p_k|OX|}{25 q_k+56p_k}.\]
First, we consider the denominator. We must show that
\[ 25q_k+56p_k=q_{k+2}.\] 
This follows from \eqref{eqn:matrix}. We then consider the numerator where we substitute in the formulas for $|OY|$ and $|OX|:$
\begin{align*}
 q_k|OY|-p_k|OX|&= (e_k\beta+d_k)-\frac{p_k(4\beta-7)}{5} \\
                &=\frac{1}{5}(b(5e_k-4p_k)+(5d_k+7p_k)) .
\end{align*} 
It remains to show that
\begin{align*}
   - 4p_k+5e_k=-e_{k+1} \\
    5d_k+7p_k=d_{k+1}.
\end{align*}

Again, we prove the first identity by induction and the second follows similarly. The base case is easily checked. Suppose $4p_k - 5e_k = e_{k+1}$ for all $k$. Then, for $ k + 1$, by Definition \ref{def:outerclasses}, 
\begin{align*}
    4p_{k+1} &= 4(22p_k - p_{k-1}) \\
             &= 22(e_{k+1} + 5e_k) - (e_k + 5e_{k-1}) \\
             &= (22e_{k+1} - e_k) + 5(22e_{k} - e_{k-1}) \\
             &= e_{k+2} + 5e_{k+1}.
\end{align*}

It remains to show that the formula for $|YV|$ holds. We can verify this by checking the first equality in
\[\frac{4\beta+7}{q_{k+1}q_{k+2}}=\frac{d_k-e_k\beta}{5q_{k+1}}-\frac{d_{k+1}-e_{k+1}\beta}{5q_{k+2}}=|VX|-|V_yX|=|Y_yV_y|.\] 
This is equivalent to showing 
\[ \beta(-e_kq_{k+2}+e_{k+1}q_{k+1})+(q_{k+2}d_k-d_{k+1}q_{k+1})=5(4\beta+7).
\] 
Therefore, we must show that
\begin{align*}
    -e_kq_{k+2}+e_{k+1}q_{k+1}&=20 \\
    q_{k+2}d_k-d_{k+1}q_{k+1}&=35. 
\end{align*}

We prove the first identity here and the second follows by a similar manner. Suppose that for all $ k$, we have $-e_k q_{k+2} + e_{k+1} q_{k+1} = 20$. Then, for $ k + 1$, 
\begin{align*}
      -e_{k+1} q_{k+3} + e_{k+2} q_{k+2} &=  -e_{k+1}(22q_{k+2} - q_{k+1}) + e_{k+2} q_{k+2} \\
    &=  -22e_{k+1} q_{k+2} + e_{k+1} q_{k+1} + e_{k+2} q_{k+2} \\
    &=  -22e_{k+1} q_{k+2} + (20 + e_{k} q_{k+2}) + e_{k+2} q_{k+2} \\
    &=  20 - (22e_{k+1} - e_{k})q_{k+2} + e_{k+2} q_{k+2} \\
    &=  20 - e_{k+2} q_{k+2} +e_{k+2} q_{k+2} \\
    &=  20.
\end{align*}
\end{proof}

\begin{center}
\begin{figure}[h]
    \includegraphics[width=4in]{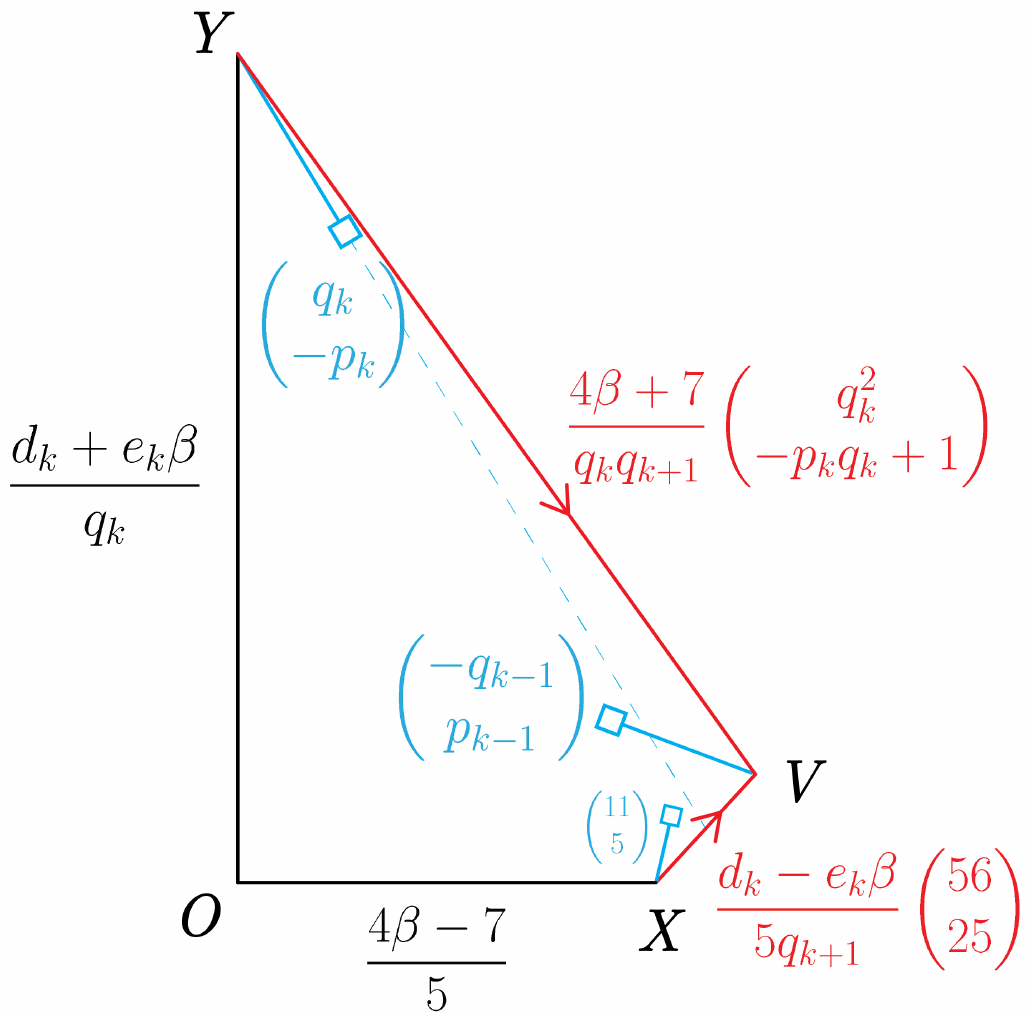}
    \caption{This figure illustrates Lemmas~\ref{lem:rays} and \ref{lem:als}. The
    nodal rays are drawn in light blue, with a square indicating their marked
    point. The fact that 
    $\vec{n}_Y$ intersects $\protect\overrightarrow{XV}$, 
    Lemma~\ref{lem:correctside}, is 
    indicated by the dashed blue line. The affine lengths of $|OX|$ and $|OY|$ are in 
    black, while the vectors $|YV|\protect\overrightarrow{YV}$ 
    and $|XV|\protect\overrightarrow{XV}$ 
    are labeled in red, with their directions indicated by arrowheads.}
    \label{fig:22131k}
\end{figure}
\end{center}

Together, these lemmas prove:
\begin{prop}\label{prop:fullacc}
    There is a full filling at the accumulation point. That is,
    \[
    c_\beta(\acc(\beta))=\vol(\beta).
    \]
\end{prop}
\begin{proof}
    By Lemma \ref{lem:als} and Lemma \ref{lem:volbaccb}, the sequence $v^2yxy^k$ of mutations of the rectangle $\Omega_\beta$ is a convex quadrilateral containing the
    \[
    \frac{d_k+e_k\beta}{q_k}\times\left(-1+\frac{\sqrt{30}}{3}\right)=\frac{d_k+e_k\beta}{q_k}\times\frac{1}{\vol(\beta)}
    \]
    right triangle abutting the axes. By multiplying by $\vol(\beta)$, we invoke Proposition \ref{prop:atfswork} to obtain an embedding
    \[
    (1-\eps)\cdot E\left(1,\frac{\vol(\beta)(d_k+e_k\beta)}{q_k}\right)\sembeds P(\vol(\beta),\vol(\beta)\beta)
    \]
    for all $\eps>0$. It therefore remains to show
    \[
    \lim_{k\to\infty}\frac{\vol(\beta)(d_k+e_k\beta)}{q_k}=\acc(\beta).
    \]
    First we note that
    \begin{equation}\label{eqn:lowerblocked}
    \vol(\beta)=\frac{5\acc(\beta)}{17+6\beta},
    \end{equation}
    which can be checked using the formulas for $\beta$ and $\acc(\beta)$ in \eqref{eqn:baccb}. Thus our goal becomes
    \[%\begin{equation}\label{eqn:volaccgoal}
        \lim_{k\to\infty}\frac{d_k+e_k\beta}{q_k}=\frac{17+6\beta}{5}.
    \]%\end{equation}
    We find a closed form for the recursion $x_k=22x_{k-1}-x_{k-2}$ with $x_k=d_k,e_k,q_k$. Set
    \begin{align*}
    r&=11+2\sqrt{30}
    \\d&=\frac{3}{2}+\frac{31}{120}\sqrt{30},\quad \ov{d}=\frac{3}{2}-\frac{31}{120}\sqrt{30}
    \\e&=\frac{1}{2}+\frac{1}{10}\sqrt{30},\quad \ov{e}=\frac{1}{2}-\frac{1}{10}\sqrt{30}
    \\q&=\frac{1}{2}+\frac{1}{15}\sqrt{30},\quad \ov{q}=\frac{1}{2}-\frac{1}{15}\sqrt{30}.
    \end{align*}
    Then
    \begin{align*}
    d_k&=dr^k+\ov{d}r^{-k}
    \\e_k&=er^k+\ov{e}r^{-k}
    \\q_k&=qr^k+\ov{q}r^{-k}.
    \end{align*}
        Then we have
        \begin{align*}
            \lim_{k\to\infty}\frac{d_k+e_k\beta}{q_k}&=\lim_{k\to\infty}\frac{dr^k+\ov{d}r^{-k}+\left(er^k+\ov{e}r^{-k}\right)b}{qr^k+\ov{q}r^{-k}}
            \\&=\lim_{k\to\infty}\frac{d+\ov{d}r^{-2k}+\left(e+\ov{e}r^{-2k}\right)b}{q+\ov{q}r^{-2k}}\\&=\frac{d+e\beta}{q}
            \\&=\frac{17+6\beta}{5}.
        \end{align*}
\end{proof}
\begin{rmk}
    Notice that the slope of $\overrightarrow{YV}$ has limit
    \[
    \lim_{k\to\infty}\frac{p_kq_k-1}{q_k^2}=\lim_{k\to\infty}\frac{p_k}{q_k}=\acc(\beta)
    \]
    by Corollary \ref{cor:pqacc}. Coupled with the fact that
    \[
    \overrightarrow{XV}=\frac{d_k-e_k\beta_k}{5q_{k+1}}\begin{pmatrix}
        56\\25
    \end{pmatrix},
    \]
    to prove Proposition \ref{prop:fullacc} it would suffice to show that $d_k/e_k\to \beta$, so that the short side $\overrightarrow{XV}$ approaches zero and thus the ratio $|OY|/|OX|$ approaches the slope of $\overrightarrow{YV}$. However, this would also require solving the recursion.
\end{rmk}

To get the points on the capacity function, we now consider the sequence $v^2yxy^kxy^2$. This allows us to prove Proposition \ref{prop:ic1} and support Conjecture \ref{conj:ic2}, because
\begin{itemize}
    \item the sequence of mutations $v^2yxy^kxy$ provides an embedding realizing the inner corner between the obstructions from $\bE_k$ and $\hat\bE_{k+1}$, while
    \item the sequence of mutations $v^2yxy^kxy^2$ conjecturally provides an embedding realizing the inner corner between the obstructions from $\hat\bE_{k+1}$ and $\bE_{k+1}$.
\end{itemize}

We will use the notation
\[ d':=2q-d, e':=2q-e.\]
In the following lemma, we are going to use no subscripts to denote the vertices from $v^2yxy^k,$ and then add a subscript of $x$ to get the vertices from $v^2yxy^kx$.
\begin{lemma}\label{lem:k3}
Beginning with the data from Lemma~\ref{lem:als} from performing the sequence $v^2yxy^k$, one mutation by $x$ gives the nodal rays
\[
\vec{n}_{Y_x}=\begin{pmatrix}q_k\\-p_k\end{pmatrix},\quad \vec{n}_{V_x}=\begin{pmatrix}-11\\-5\end{pmatrix},\quad \vec{n}_{X_x}=\begin{pmatrix}121p_{k-1}+54q_{k-1}\\56p_{k-1}+25q_{k-1}\end{pmatrix},
\]
the direction vectors
\[
\overrightarrow{Y_xV_x}=\begin{pmatrix}q_k^2\\-p_kq_k+1\end{pmatrix},\quad \overrightarrow{X_xV_x}=\begin{pmatrix}-54q_k^2-121p_kq_k+121\\-25q_k^2-56p_kq_k+56\end{pmatrix},
\]
and the affine lengths
\[
\begin{split}
|OY_x|=\frac{d_k+e_k\beta}{q_k},\quad |OX_x|=\frac{d'_{k+1}+e'_{k+1}\beta}{q_{k+1}},
\\|Y_xV_x|=\frac{-d'_{k+1}+e'_{k+1}\beta}{q_k\hat{q}_{k+1}},\quad |X_xV_x|=\frac{d_k-e_k\beta}{q_{k+1}\hat{q}_{k+1}}.
\end{split}
\]

\end{lemma}

\begin{proof}
We first give the proof for the direction vectors and nodal rays. 

By Lemma \ref{lem:rays}, after performing the sequence $v^2yxy^k$ we have $\vec{n}_X = (11, 5)$, $\vec{n}_V = (-q_{k-1}, p_{k-1})$, and $\overrightarrow{XV} = (56, 25)$. Note that for the next mutation at $X$, because $\vec{n}_X$ has positive slope, it will always hit the edge $\overrightarrow{YV}$. Thus, the mutation matrix should satisfy
\[  M\vec{n}_X=\vec{n}_X, \quad
   M\overrightarrow{XV}
   =\overrightarrow{OX} \iff M =  \begin{pmatrix} -54 & 121 \\ -25 & 56 \end{pmatrix}. \]
The polygon after mutation at $X$ should thus have $\vec{n}_{Y_x}=\vec{n}_Y$,
\[
\vec{n}_{V_x}=-\vec{n}_X=\begin{pmatrix}-11\\-5\end{pmatrix},\quad \vec{n}_{X_x}=M\vec{n}_V=\begin{pmatrix}121p_{k-1}+54q_{k-1}\\56p_{k-1}+25q_{k-1}\end{pmatrix},
\]
and
\[
\overrightarrow{X_xV_x}=-M\overrightarrow{YV}=\begin{pmatrix}-54q_k^2-121p_kq_k+121\\-25q_k^2-56p_kq_k+56\end{pmatrix},
\]
while $\overrightarrow{Y_xV_x}=\overrightarrow{YV}$ because $\vec{n}_X$ hits $\overrightarrow{YV}$.
% \begin{align*}
%     \vec{n}_{V_x} &= (-11, -5) \\
%     \vec{n}_{X_x} &= M \overline{n}_V 
%                     = (54q_{k-1} + 121p_{k-1}, 25q_{k-1} + 56p_{k-1}) \\
%     \overrightarrow{V_xX_x} &= M \overrightarrow{YV}
%                          = ( -54q_k^2 - 121p_kq_k + 121, -25q_k^2 - 56p_kq_k + 56)
% \end{align*}

We now give the proofs for the affine lengths.
Note that $|OY_x|=|OY|$ because $Y_x=Y$ Next, to compute $|OX_x|$, we check that given the formulas for $|OX|$ and $|XV|$ from Lemma~\ref{lem:als}, we have
\[
|OX_x|=|OX|+|XV|=\frac{4\beta-7}{5}+\frac{d_k-e_k\beta}{5q_{k+1}}=\frac{d'_{k+1}+e'_{k+1}\beta}{q_{k+1}}.
\]
% Letting $(d_\rho,e_\rho,q_\rho)=(17,6,5),$ these follow from:
% \begin{align*} 
%  e_\rho'q_{k+1}-e_k&=q_\rho e'_{k+1} \\
%  d_\rho'q_{k+1}+d_k&=q_\rho d'_{k+1}.
%  \end{align*} 
This follows from the identities
\begin{align*} 
 4q_{k+1}-e_k&=5e'_{k+1} \\
 -7q_{k+1}+d_k&=5d'_{k+1},
 \end{align*} 
which hold by induction because they are linear identities and the $(d_k,e_k,q_k)$ (and thus $d_k'$ and $e_k'$) satisfy the same linear recursion.

We now look at $|Y_xV_x|$. Following the proof of \cite[Lem.~6.1~(ii)]{M1} (which solves for $|Y_xV_x|$ using the fact that the sides of the quadrilateral which is fixed under $x$-mutation, with sides $|OX|\overrightarrow{OX}, -|Y_xV_x|\overrightarrow{YV}, |OY_x|\overrightarrow{OY}$, and a side parallel to $\vec{n}_X$, must add to zero), the stated formula for $|Y_xV_x|$ holds if our analogue of \cite[(6.0.2)]{M1} gives us the claimed value for $|Y_xV_x|$, that is
\begin{equation}\label{eqn:YxVx}
    |Y_xV_x|=\frac{11|OY|+5|OX|}{-11+q_k(41p_k+5q_k-30p_k)}=\frac{-d'_{k+1}+e'_{k+1}\beta}{q_k\hat{q}_{k+1}}.
\end{equation}

For the denominator of \eqref{eqn:YxVx}, note that
\[
-11+q_k(41p_k+5q_k-30p_k)\overset{(*)}{=}t_{k+1}q_k-11\overset{(**)}{=}t_kq_{k+1}-5=\hat q_{k+1},
%11(-1+p_kq_k)+5q_k^2=\hat{q}_{k+1}.
\]
where $(*)$ uses Lemma \ref{lem:identities} (iv) and $(**)$ uses the second conclusion of Lemma \ref{lem:identities} (ii), both applied to $(\bE_k,\bE_{k+1},\bE)$.
%This uses Lemma~\ref{lem:identities}~(iii)~and~(v). 

For the numerator of \eqref{eqn:YxVx}, we must show $11|OY|+5|OX|=(-d'_{k+1}+e'_{k+1}\beta)/q_k.$
We have
\[
11|OY|+5|OX|=\frac{11(d_k+e_k\beta)}{q_k}+(4\beta-7),%=\frac{1}{q_k}((11e_k+4q_k)b+(11d_k-7q_k))
\]
so we must check that
\begin{align*}
    11e_k+4q_k&=e'_{k+1}\\
    11d_k-7q_k&=-d'_{k+1}.
\end{align*}
As in the proof of the $|OX_x|$ formula, these hold by induction, using the fact that all terms satisfy the same recursion.
%Similarly to \cite[Lem.~4.6]{M1}, these hold by Lemma~\ref{lem:identities}~(v)~and~(vii). 

Finally, %letting $(d'_\rho,e'_\rho)=(-7,4),$
to verify the formula for $|X_xV_x|$, we check that
\[
|X_xV_v|=|YV|-|Y_xV_x|=\frac{4\beta+7}{q_kq_{k+1}}-\frac{-d'_{k+1}+e'_{k+1}\beta}{q_k\hat{q}_{k+1}}=\frac{d_k-e_k\beta}{q_{k+1}\hat{q}_{k+1}}.
\]
This is equivalent to
\begin{align*}
    4\hat{q}_{k+1}-e'_{k+1}q_{k+1}&=-e_kq_k \\
    7\hat{q}_{k+1}+d'_{k+1}q_{k+1}&=d_kq_k.
\end{align*}
The first formula can be verified as follows: %similarly to equation (6.0.3) of \cite[Lem.~6.1]{M1}.
\begin{align*}
    4\hat{q}_{k+1}-e'_{k+1}q_{k+1}&=-e_kq_k
    \\4(22t_kq_{k+1}-5)-(2q_{k+1}-e_{k+1})q_{k+1}&=-e_kq_k
    \\16t_kq_{k+1}-80-7q_{k+1}+p_{k+1}q_{k+1}-t_{k+1}q_{k+1}&=-p_kq_k-q_k^2+t_kq_k \text{ by Lemma \ref{lem:depqt}}\\-6t_kq_{k+1}+30-7q_{k+1}+p_{k+1}q_{k+1}&=-p_kq_k-q_k^2 \text{ by Lemma \ref{lem:identities} (v)}
    \\5t_kq_{k+1}+30-22q_kq_{k+1}&=-p_kq_k-q_k^2 \text{ by Lemma \ref{lem:identities} (iii)}
    \\t_kq_{k+1}+6&=t_{k+1}q_k \text{ by Lemma \ref{lem:identities} (i)},
\end{align*}
which holds by Lemma \ref{lem:identities} (ii). All applications of Lemma \ref{lem:identities} use the triple $(\bE_k,\bE_{k+1},\bE)$. The verification of the second formula uses the exact same sequence of identities.
\end{proof}

The next lemma is the key step which allows us to prove Proposition \ref{prop:ic1} and thus Theorem \ref{thm:main}. Similar to Lemma \ref{lem:k3}, we use no subscripts to denote the vertices after the sequence $v^2yxy^kx$, and a subscript $y$ to denote the vertices after the final $y$ mutation. 
\begin{lemma}\label{lem:31}
After performing the sequence of mutations $v^2yxy^kxy$, the affine lengths of the axis sides are
\begin{align*}
    |OY_y| &= \frac{\hat{d}_{k+1} + \hat{e}_{k+1}\beta}{\hat{q}_{k+1}}\\
    |OX_y| &= \frac{d_k + e_k\beta}{p_k} 
\end{align*}
\end{lemma}

\begin{proof} 
First we show that for the last $y$-mutation, $\vec{n}_Y$ extends to hit the side $\overrightarrow{OX}$ (rather than $\overrightarrow{XV}$ as for earlier $y$-mutations in Lemmas \ref{lem:vvyx} and \ref{lem:correctside}). As a consequence, we prove the formula for $|OX_y|$. By Lemma \ref{lem:k3}, we need to show that
\begin{enumerate}[label=(\roman*)]
    \item the $x$-coordinate of the $x$-intercept of the line of slope $-p_k/q_k$ (the slope of $\vec{n}_Y$) through the point $(0,|OY|)=(0,(d_k+e_k\beta)/q_k)$ is less than $|OX|=(d'_{k+1}+e'_{k+1}\beta)/q_{k+1}$, and that
    \item this $x$-coordinate equals the claimed value $(d_k+e_k\beta)/p_k$ of $|OX_y|$.
\end{enumerate}

Note that this $x$-coordinate is at the solution to
\[
-\frac{d_k+e_k\beta}{q_k}=-\frac{p_k}{q_k}x \iff x=\frac{d_k+e_k\beta}{p_k},
\]
so proving (i) suffices to prove (ii). To prove (ii), we need to show
\begin{equation}\label{eqn:hitsOX}
\frac{d_k+e_k\beta}{p_k}<\frac{d'_{k+1}+e'_{k+1}\beta}{q_{k+1}}.
\end{equation}
% which follows from
% \begin{align*}
% q_{k+1}d_k&<p_kd'_{k+1}
% \\q_{k+1}e_k&<p_ke'_{k+1}.
% \end{align*}
 Recall that in Lemma \ref{prop:fullacc}, we solved the recursion defining the $\bE_k$ and found
    \begin{align*}
    r&=11+2\sqrt{30}
    \\d&=\frac{3}{2}+\frac{31}{120}\sqrt{30},\quad \ov{d}=\frac{3}{2}-\frac{31}{120}\sqrt{30}
    \\e&=\frac{1}{2}+\frac{1}{10}\sqrt{30},\quad \ov{e}=\frac{1}{2}-\frac{1}{10}\sqrt{30}
    \\q&=\frac{1}{2}+\frac{1}{15}\sqrt{30},\quad \ov{q}=\frac{1}{2}-\frac{1}{15}\sqrt{30}.
    \end{align*}
We can further compute
\[
p=\frac{7}{2}+\frac{13}{20}\sqrt{30},\quad \ov{p}=\frac{7}{2}-\frac{13}{20}\sqrt{30}.
\]
Expanding \eqref{eqn:hitsOX} using $x_k=xr^k+\ov{x}r^{-k}$ for $x=d,e,p,q$, we want to show
\[
c_{2k+1}r^{2k+1}+c_1r+c_{-1}r^{-1}+c_{-2k-1}r^{-2k-1}>0,
\]
where
\begin{align*}
    c_{2k+1}&=(2q-d)p+(2q-e)p\beta-dq-eq\beta=0,
    \\c_1&=(2q-d)\ov{p}+(2q-e)\ov{p}\beta-\ov{d}q-\ov{e}q\beta=-\frac{17}{12}+\frac{31}{120}\sqrt{30}\approx-0.0017,
    \\c_{-1}&=(2\ov{q}-\ov{d})p+(2\ov{q}-\ov{e})p\beta-d\ov{q}-e\ov{q}\beta=\frac{27}{8}+\frac{37}{60}\sqrt{30}\approx6.7526, \ \mbox{and}
    \\c_{-2k-1}&=(2\ov{q}-\ov{d})\ov{p}+(2\ov{q}-\ov{e})\ov{p}\beta-\ov{d}\ov{q}-\ov{e}\ov{q}\beta=-\frac{215}{24}+\frac{39}{24}\sqrt{30}\approx-0.0578.
\end{align*}
Because $c_{2k+1}=0,\ r>1,\ k\geq0$, and $c_{-2k-1}<0$, for all $k\geq1$,
\[
c_{2k+1}r^{2k+1}+c_1r+c_{-1}r^{-1}+c_{-2k-1}r^{-2k-1}>c_1r+c_{-1}r^{-1}+c_{-2k-1}r^{-2(k-1)-1},
\]
which means
%\[
\begin{align*}
    c_{2k+1}r^{2k+1}+c_1r+c_{-1}r^{-1}+c_{-2k-1}r^{-2k-1} & >c_1r+(c_{-1}+c_{-2k-1})r^{-1}\\
    & =-196+215\sqrt{\frac{5}{6}}\approx0.2673>0.
\end{align*}
%\]
% \begin{align*}
%     d_kq_{k+1}+e_kq_{k+1}b&<d_{k+1}'p_k+e_{k+1}'p_kb
%     \\(dq+eqb)r^{2k+1}+(d\ov{q}+e\ov{q}b)r^{-1}&+(\ov{d}q+\ov{e}qb)r+(\ov{d}\ov{q}+\ov{e}\ov{q}b)r^{-2k-1})
%     \\&<(dq+eqb)r^{2k+1}+(d\ov{q}+e\ov{q}b)r^{-1}+(\ov{d}q+\ov{e}qb)r+(\ov{d}\ov{q}+\ov{e}\ov{q}b)r^{-2k-1})
% \end{align*}

Finally we prove the statement about $|OY_y|$. By Lemma~\ref{lem:k3}, we must verify that
\[
|OY_y|=|OY|+|YV|=
\frac{d_k+e_k\beta}{q_k}+\frac{-d'_{k+1}+e'_{k+1}\beta}{q_k\hat{q}_{k+1}}=\frac{\hat{d}_{k+1}+\hat{e}_{k+1}\beta}{\hat{q}_{k+1}}.
\]
% Note, this equality is analogous to the proof of the formula for $|OY|$ in Lemma~\ref{lem:als}. The left hand side equals:
% \[ \frac{1}{\hat{q}_{k+1}q_{k}}(b(\hat{q}_{k+1}e_k+e'_{k+1})+\hat{q}_{k+1}d_k-d'_{k+1}).\]
Therefore, the formula for $|OY|$ will hold if
\begin{align*} 
 e'_{k+1}&=q_k\hat{e}_{k+1}-\hat{q}_{k+1}e_k
 \\
d'_{k+1}&= \hat{q}_{k+1}d_k-q_k\hat{d}_{k+1}.
\end{align*} 
Using Remark~\ref{rmk:d'e'}, and Lemma \ref{lem:depqt} to replace all $d, e$ terms with $p, q, t$, the first identity becomes
\begin{align*}
    7q_{k+1}-p_{k+1}+t_{k+1}&=q_k\hat p_{k+1}+q_k\hat q_{k+1}-q_k\hat t_{k+1}-\hat q_{k+1}p_k-\hat q_{k+1}q_k+\hat q_{k+1}t_k
    \\7q_{k+1}-p_{k+1}&=-q_k\hat t_{k+1}+\hat q_{k+1}t_k \text{ by Lemma \ref{lem:identities} (vii)},
\end{align*}
which holds by Lemma~\ref{lem:identities}~(ii). (Both uses of Lemma \ref{lem:identities} (vii) are applied to the triple $(\bE_k,\hat\bE_{k+1},\bE_{k+1})$.) The proof of the second identity is almost identical.
\end{proof}

\begin{rmk} Note, that in Lemma~\ref{lem:31} the nodal ray of the additional $y$ mutation hits the side $|OX|$ of the quadrilateral rather than the side $|XV|$ like in Lemma~\ref{lem:als}. This implies the formulas in Lemma~\ref{lem:31} are no longer in parallel with the formulas derived by Magill in \cite{M1}. Instead, as found in \cite{M2}, when the nodal ray emanating from $Y$ changes the side it intersects with this corresponds to moving from an embedding strictly above the function to one that lies on the function. 
\end{rmk}

We prove Proposition \ref{prop:ic1}, proving that there is a full filling at the inner corner between the obstructions from $\bE_k$ and $\hat\bE_{k+1}$.

\begin{proof} (of Proposition \ref{prop:ic1})
By Lemma \ref{lem:31}, there is an embedding
\[
(1-\eps)\cdot E\left(\frac{d_k+e_k\beta}{p_k},\frac{\hat d_{k+1}+\hat e_{k+1}\beta}{\hat q_{k+1}}\right) \sembeds P(1,\beta)
\]
for all $\eps>0$. Thus there is an embedding
\[
E\left(1,\frac{p_k(\hat d_{k+1}+\hat e_{k+1}\beta)}{\hat q_{k+1}(d_k+e_k\beta)}\right) \sembeds \frac{1}{1-\eps}\cdot P\left(\frac{p_k}{d_k+e_k\beta},\frac{p_k\beta}{d_k+e_k\beta}\right),
\]
implying that
\[
c_\beta\left(\frac{p_k(\hat d_{k+1}+\hat e_{k+1}\beta)}{\hat q_{k+1}(d_k+e_k\beta)}\right)\leq\frac{p_k}{d_k+e_k\beta}
\]
once we take the infimum defining $c_\beta$.
\end{proof}

To close this section, we explain our reasoning behind Conjecture \ref{conj:ic2}.
\begin{rmk}\label{rmk:conj}
We initially believed that
\[
c_\beta|_{[1,\acc(\beta)]}=\sup_k\left\{\mu_{\bE_k,\beta}\middle|_{[1,\acc(\beta)]}\right\},
\]
and thought that Theorem \ref{thm:main} would be proved as a consequence of the lower bounds from Proposition \ref{prop:oc}~(i) and an embedding providing an upper bound at the intersection between the horizontal line through $O_{k-1}$ and the line through the origin and $O_k$ (see Figure \ref{fig:main} or \ref{fig:mainzoomed}).
%$\lambda=p_k/(d_k+e_k\beta)$ and the line through the origin and $(p_{k+1}/q_{k+1},p_{k+1}/(d_{k+1}+e_{k+1}\beta)$.
Specifically, we thought there might be an infinite sequence of mutations starting with $v^2yxy^kxy$ which approached this hypothetical inner corner. We held on to this expectation because the obstructions from the classes $\hat\bE_k$ are extremely difficult to visualize computationally.\footnote{We also naively failed to notice that the intersection between the obstructions from $\bE_0$ and $\bE_1$ is below the volume obstruction. This is not true for all $k$, however.} %: note the scale in Figure \ref{fig:E0hatE1}. 
The obstruction from $\hat\bE_2$ corresponds to the $4,769,607,569^\text{th}$ ECH capacity by Lemma \ref{lem:Eck}, while comparing the ratios between even the first $10,000,000$ ECH capacities is very computationally expensive!

\begin{figure}[H]
    \centering
    \includegraphics[width=.7\textwidth]{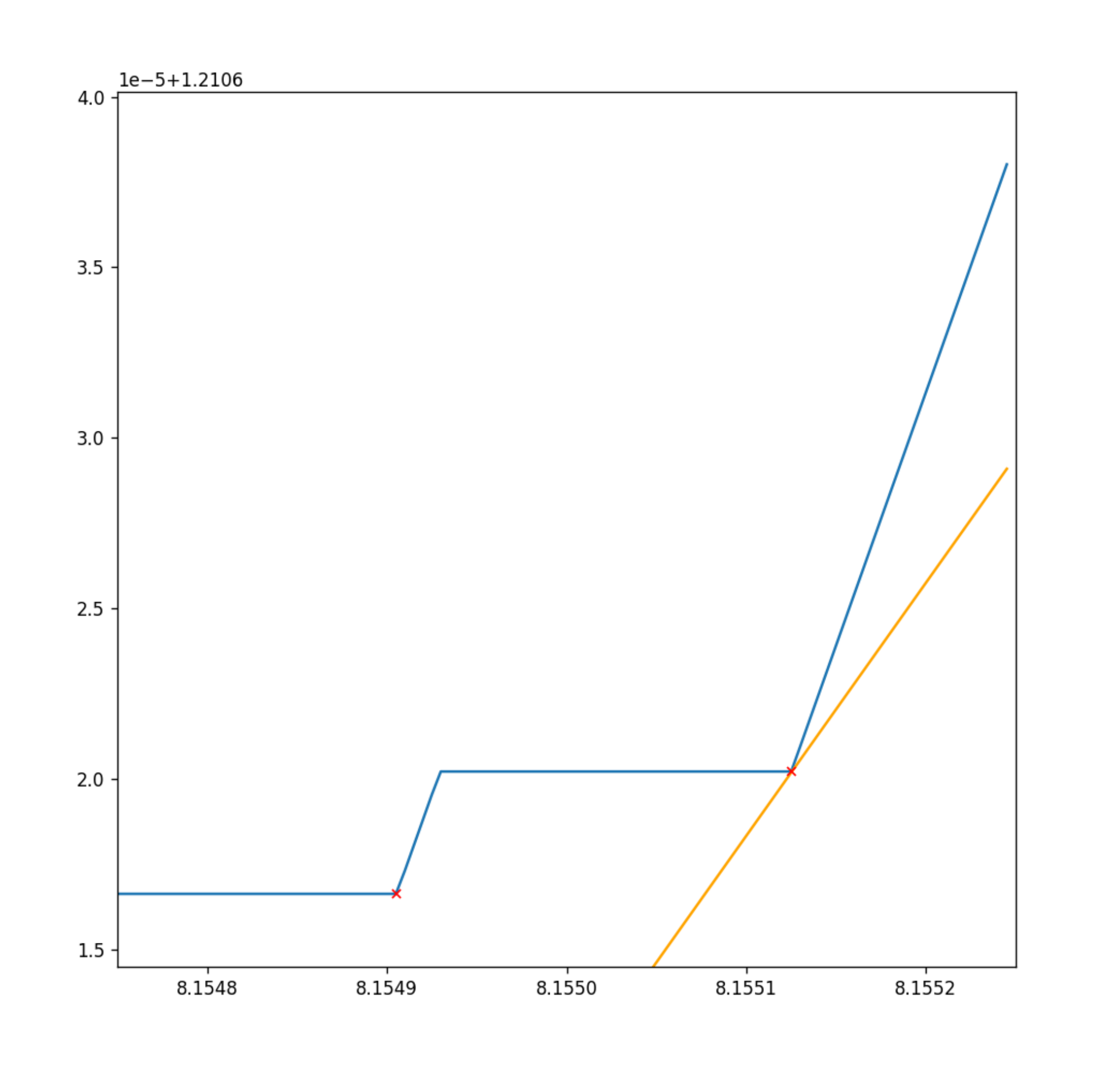}
    \caption{Here we have depicted the obstruction from the inner class $\hat\bE_1$. The orange curve is the volume obstruction $\vol_\beta(z)$, and the embedding function $c_\beta$ is in blue. The left hand red cross depicts the embedding from the mutation sequence $v^2yx^2y$ proving Proposition \ref{prop:ic1}, while the right hand red cross depicts the embedding from the mutation sequence $v^2yx^2y^2$ supporting Conjecture \ref{conj:ic2}.}
    \label{fig:E_1}
\end{figure}

By comparison with \cite{usher, usherletter}, we eventually discovered the inner classes $\hat\bE_k$, and found the same sequence considered by Magill in \cite{M2}, namely $v^2yxy^kxy$, proves Proposition \ref{prop:ic1}. In order to compute all of $c_\beta$ on $[1,\acc(\beta)]$, we would need to prove Conjecture \ref{conj:ic2} by identifying a sequence of mutations of $P(1,\beta)$ so that
\[
|OX|=\frac{\hat d_{k+1}+\hat e_{k+1}\beta}{\hat p_{k+1}}, \quad |OY|=\frac{d_{k+1}+e_{k+1}\beta}{q_{k+1}}.
\]
Our hypothesized sequence of mutations is $v^2yxy^kxy^2$, which would provide a full filling and thus upper bound for $c_\beta$ at the inner corner between the obstructions from $\hat\bE_{k+1}$ and $\bE_{k+1}$: in Figure \ref{fig:mainzoomed}, this would mean that $c_\beta$ equals the dashed black line.

\begin{figure}[h]
    \centering
    \includegraphics{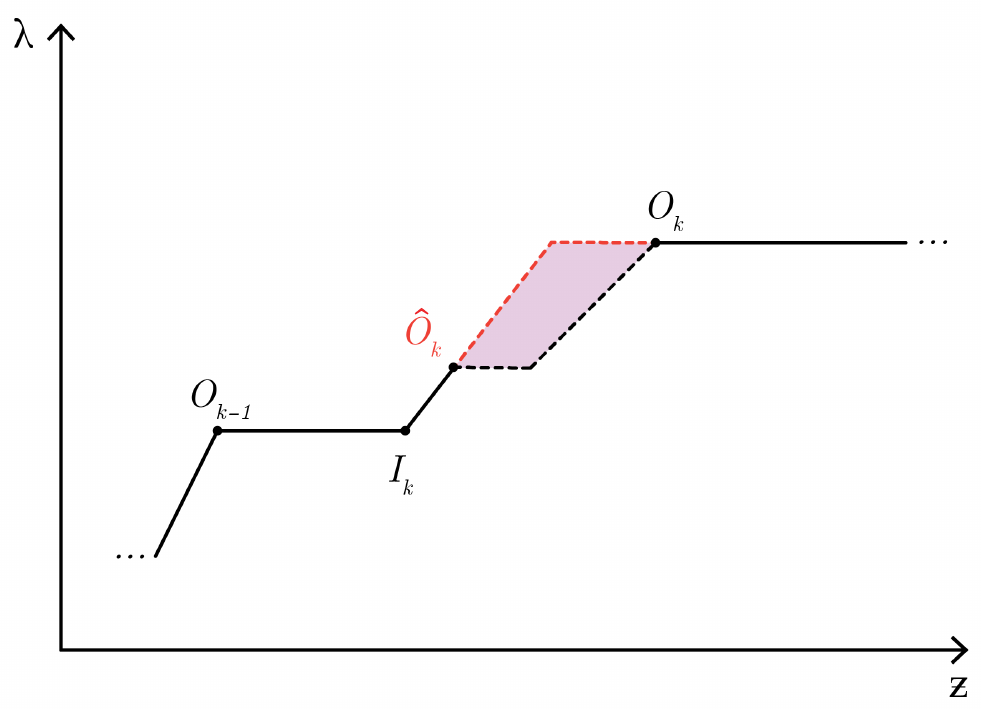}
    \caption{This figure provides more detail on the schematic presented in Figure \ref{fig:combination}. Using Propositions \ref{prop:oc}~(i), \ref{prop:oc}~(ii), and \ref{prop:ic1}, along with Lemma \ref{lem:ics} (i), we have computed $c_\beta$ along the solid black lines. However, we would need to prove Conjecture \ref{conj:ic2} to compute $c_\beta$ in between $\hat O_k$ and $O_k$. We know the function must lie in the violet quadrilateral (possibly on its boundary), and we conjecture the function is given by its lower boundary, the dashed black line.}
    \label{fig:mainzoomed}
\end{figure}

However, computing $|OX|$ and $|OY|$ for the sequence $v^2yxy^kxy^2$ -- even though it differs from the sequence $v^2yxy^kxy$ considered in Lemma \ref{lem:31} by only one $y$-mutation -- is considerably more time consuming because for the final two $y$ mutations, the nodal ray $\vec{n}_Y$ hits the bottom side of the quadrilateral $|OX|$ rather than the side $|XV|$ that the previous $y^k$ mutations hit. Thus, to determine the combinatorics for the quadrilaterals of the $y$ mutations hitting the $|OX|$ would involve many new computations. While this could be done, computing the whole function is not necessary to claim there is an infinite staircase. 

\end{rmk}

\section{Other properties of the embedding function}\label{sec:ideas}

In this section we collect several observations about the structure of the ellipsoid embedding function for polydisks which may be useful for future work.

\subsection{Towards Conjecture \ref{thm:generalb}} Proving Conjecture \ref{thm:generalb} would require analogues of Propositions \ref{prop:oc}~(i) and \ref{prop:ic1}. For the obstructions providing the outer corners, this means identifying new Diophantine classes and proving that they are Diophantine. In analogy to \cite{MMW}, we expect that the correct classes $\bE_{k,n}$ and $\hat\bE_{k,n}$ can be obtained from $\bE_{k,2}:=\bE_k$ and $\hat\bE_{k,2}:=\hat\bE_k$ as follows:
\begin{itemize}
    \item Modify $p_{k,2}/q_{k,2}:=p_k/q_k$ by adding $2n-4$ to each entry in its continued fraction; this is $p_{k,n}/q_{k,n}$.
    \item Use Lemma \ref{lem:depqt} to define $d_{k,n}$ and $e_{k,n}$.
\end{itemize}
To prove the analogue of Proposition \ref{prop:ic1} requires identifying new embeddings. In analogy to \cite{M1} we expect that this amounts to performing the mutations $v^n$ at the start of every sequence of mutations considered in \S\ref{ssec:inner}, rather than just $v^2$.

It would also be possible to prove Conjecture \ref{thm:generalb} using \cite[Thm.~4.4]{usher}. However, this would require proving that the Diophantine quasi-perfect classes $\bE_{k,n}$ are perfect,\footnote{For a class to be {\bf perfect}, it must be represented by a symplectically embedded sphere, rather than one that is only immersed. That the classes $\bE_{k,n}$ and $\hat\bE_{k,n}$ are quasi-perfect can be proved following the $n=2$ case discussed in the proof of Proposition \ref{prop:oc}.} which we do not do in this paper.

Finally we discuss the obstructions analogous to $\bE$ appearing after the accumulation point of the conjectural staircases $c_{\beta_n}$.

\begin{rmk}\label{prop:stepafteracc}\hfill

\begin{enumerate}[label=(\roman*)]
    \item We predict that the ECH capacity which gives the step after the accumulation point, generalizing the obstruction for $\bE$ in the case $n=2$, for the infinite staircases of Conjecture \ref{thm:generalb} will have index
\[
k_n=(2n+1)(2n^2+6n+5)=4n^3+14n^2+16n+5
\]
for $n\geq2$. At these steps, the $z-$coordinate of the associated outer corner is given by the fraction $\frac{p_n}{q_n}$, where
\[
p_n=4n^2+10n+5,
\]
and
\[
q_n=2n+1.
\]
The first few of these values are summarized in the table below:

\begin{center}
    \begin{tabular}{c|c|c|c }
    %\hline
    $n$ & $p_n$ & $q_n$ & $k_n$ \\ [0.5ex]
    \hline%\hline
    2 & 41 & 5 & 125 \\
    \hline
    3 & 71 & 7 & 287 \\
    \hline
    4 & 109 & 9 & 549 \\
    \hline
    5 & 155 & 11 & 935 \\
    \hline
    6 & 209 & 13 & 1469 \\
    %\hline
    \end{tabular}
\end{center}

One can check using the formulas for $p_n,q_n,$ and $k_n$ above that
\[
k_n=\frac{(p_n+1)(q_n+1)}{2}-1
\]
for every $n\geq2$, as predicted by the proof of Lemma \ref{lem:Eck}.

Also, note that
\[
p_n+q_n=2(2n^2+6n+3)=t_{n+1},
\]
which is the predicted coefficient of the recursion governing the outer corners of the next infinite staircase in this family.

\item Another way to identify the steps after the accumulation point is to compare to the case of $H_b$ and use Conjecture \ref{conj:Pfractal}. For $H_b$, these steps are the obstructions from the quasi-perfect Diophantine classes centered at\footnote{Here we are starting with $n=2$ as in \cite{usher}, as opposed to $n=0$ as in \cite{ICERM, MMW}.}
\[
[7,4], [9,6], \dots, [2n+3,2n].
\]
Thus for $P(1,\beta)$, we expect the centers to be at $[2n+4,2n+1]$. This agrees with the $p_n$ and $q_n$ computed in (i):
\[
[2n+4,2n+1]=2n+4+\frac{1}{2n+1}=\frac{(2n+4)(2n+1)+1}{2n+1}.
\]

We can then use
\[
t_n=\sqrt{p_n^2-6p_nq_n+q_n^2+8}
\]
and Lemma \ref{lem:depqt} to identify the corresponding quasi-perfect Diophantine classes
\[
\bE_n:=(2n^2+4n+1, 2n+2, 4n^2+10n+5, 2n+1, 2(2n^2+2n-1)).
\]
Notice $\bE_2$ is what we have been referring to as $\bE$.

\item The relevance of the $\bE_n$s to the staircases $c_{\beta_n}$ can also be seen in the fact that the obstruction $\mu_{\bE_n,\beta_n}(z)$ crosses through the volume obstruction $\vol_{\beta_n}(z)$ at $z=\acc(\beta_n)$. That is,
\[
    \mu_{\bE_n,\beta_n}(\acc(\beta_n))=\frac{q_n\acc(\beta_n)}{d_n+e_n\beta_n}=\sqrt{\frac{\acc(\beta_n)}{2\beta_n}}=\frac{1+\acc(\beta_n)}{2+2\beta_n}=\vol(\beta_n),
\]
implying that \eqref{eqn:qe'd'} holds with $q_n, d_n$, and $e_n$ replacing $q, d$, and $e$. We won't prove either of these claims here, but we do note that it is a straightforward if tedious computation using the formulas for $\beta_n$ and $\acc(\beta_n)$ in Conjecture \ref{thm:generalb} and for $q_n, d_n$, and $e_n$ in (ii) above.

% We won't prove this, but assuming \eqref{eqn:qe'd'}, we do have
% \begin{align*}
% \vol(\beta_n)=\frac{q_n}{e'_n\beta_n+d'_n}&=\frac{q_n\acc(\beta_n)}{d_n+e_n\beta_n}
% \\d_n+e_n\beta_n&=\acc(\beta_n)(e'_n\beta_n+d'_n)
% \\&=\acc(\beta_n)((2q_n-e_n)\beta_n+(2q_n-d_n))
% \end{align*}

% As in the proof of Lemma \ref{lem:volbaccb},
% \begin{align*}
%     \vol(\beta_n)&=\frac{1+\acc(\beta_n)}{2\beta_n+2}
%     \\&=
% \end{align*}

% This fact can be proved in several ways (e.g. following the proofs of Lemma \ref{lem:volbaccb} and Proposition \ref{prop:fullacc}), but instead we simply note that, inspired by \eqref{eqn:qe'd'},
% \[
% \frac{q_n}{e'_n\beta_n+d'_n}=\frac{2n+1}{2n\beta_n+2n^2-1},
% \]
% while
% \[

% \]
% generalizing \eqref{eqn:lowerblocked}.
\end{enumerate}
\end{rmk}

\subsection{Usher's Conjecture}\label{ssec:Uconj}

In \cite{usher}, Usher considers the family
\[
L_{n,0}:=\sqrt{n^2-1},
\]
proving that $c_{L_{n,0}}$ has an infinite staircase. He does this by proving that a sequence of classes $A_{k,n}$\footnote{Again note that we use $k$ to denote a step of the staircase where Usher uses $i$.} are perfect.

Usher's classes $A_{k,n}$ play the role of our classes $\bE_k$ (or $\bE_{k,n}$ more generally). However, he also identified other obstructions
\[
\hat A_{k,n}:=t_{k-1,n}A_{k,n}-\bE(n),
\]
where $\bE(n):=\bE=(n+1,1,2n+3,1,2n)$.\footnote{Note that the $\hat A_{k,n}$ are not defined in precisely this way in \cite[\S4.6]{usher}; this definition is inspired by ``$x$-mutation'' investigated in \cite{MMW}. His $k$-indexing of the $\hat A_{k,n}$ classes is also one less than what we define here.}

The classes $\bE(n)$ are similar to our $\bE=(17,6,41,5,22)$. Our new staircase $c_\beta$ accumulates to precisely the point where the obstruction from $\bE$ for $z<41/5$ crosses the volume curve: see \eqref{eqn:lowerblocked}. Meanwhile, Usher's staircases satisfy
\[
\mu_{\bE(n),L_{n,0}}(\acc(L_{n,0}))=\frac{\acc(L_{n,0})}{n+1+L_{n,0}}=\vol_{L_{n,0}}(\acc(L_{n,0})).
\]

Usher conjectured he could compute the whole function up to the accumulation point. His conjecture says
\begin{conjecture}[{\cite[Conj.~4.23]{usher}}]
    Between the center of $A_{0,n}$ and $\acc(L_{n,0})$, $c_{L_{n,0}}$ equals the supremum of the obstructions $\mu_{\bE,L_{n,0}}$, where $\bE$ is one of the $A_{k,n}$ or $\hat A_{k,n}$.
\end{conjecture}

Our proof of Theorem \ref{thm:main} (i) proves that $c_\beta$ is determined on $[p_k/q_k,\hat p_{k+1}/\hat q_{k+1}]$ by only the classes $\bE_k$ and $\hat\bE_k$; proving Conjecture \ref{conj:ic2} would solve our version of Usher's conjecture. 
We expect that Usher's conjecture for the $c_{L_{n,0}}$ staircases could be solved by proving:
\begin{itemize}
    \item the analogue of Proposition \ref{prop:ic1} using the mutation sequence $v^{n-1}y^{k+1}xy$ to compute the inner corner between $A_{k,n}$ and $\hat A_{k+1,n}$;
    \item the analogue of Conjecture \ref{conj:ic2} using the mutation sequence $v^{n-1}y^{k+1}xy^2$ to compute the inner corner between $\hat A_{k+1,n}$ and $A_{k+1,n}$.
\end{itemize}

\subsection{Descending staircases and fractal structure}

If Conjecture \ref{conj:Pfractal} is true, then the set of $\sqrt{3}<\beta\leq\sqrt{8}$ for which $c_\beta$ has an infinite staircase is homeomorphic to the Cantor set. Key to proving this analogy is understanding how $c_\beta$ with $\beta=(6+5\sqrt{30})/12$ can be obtained from $c_{\sqrt{8}}$ and a descending infinite staircase $c_\beta$ ``mirroring'' $c_{\sqrt{3}}$.

In fact, the descending staircase with
\[
\beta=\frac{24+7\sqrt{3}}{13}
\]
was one of the first conjectural infinite staircases we found via computer exploration, and it is precisely this mirror! It is shown in Figure \ref{fig:sqrt3desc}. The reason why we think of $\beta=\frac{24+7\sqrt{3}}{13}$ and $\sqrt{3}$ as paired is that both accumulate to the point where the obstruction from $(3,1,7,1,4)$ intersects the volume curve $\vol_\beta(z)$; the ascending staircase $c_{\sqrt{3}}$ from below, where the obstruction has positive slope, and the descending staircase $c_\beta$ from above, where the obstruction is horizontal.

\begin{figure}[h]
    \subfigure[]{
         \centering
         \includegraphics[width=.43\textwidth]{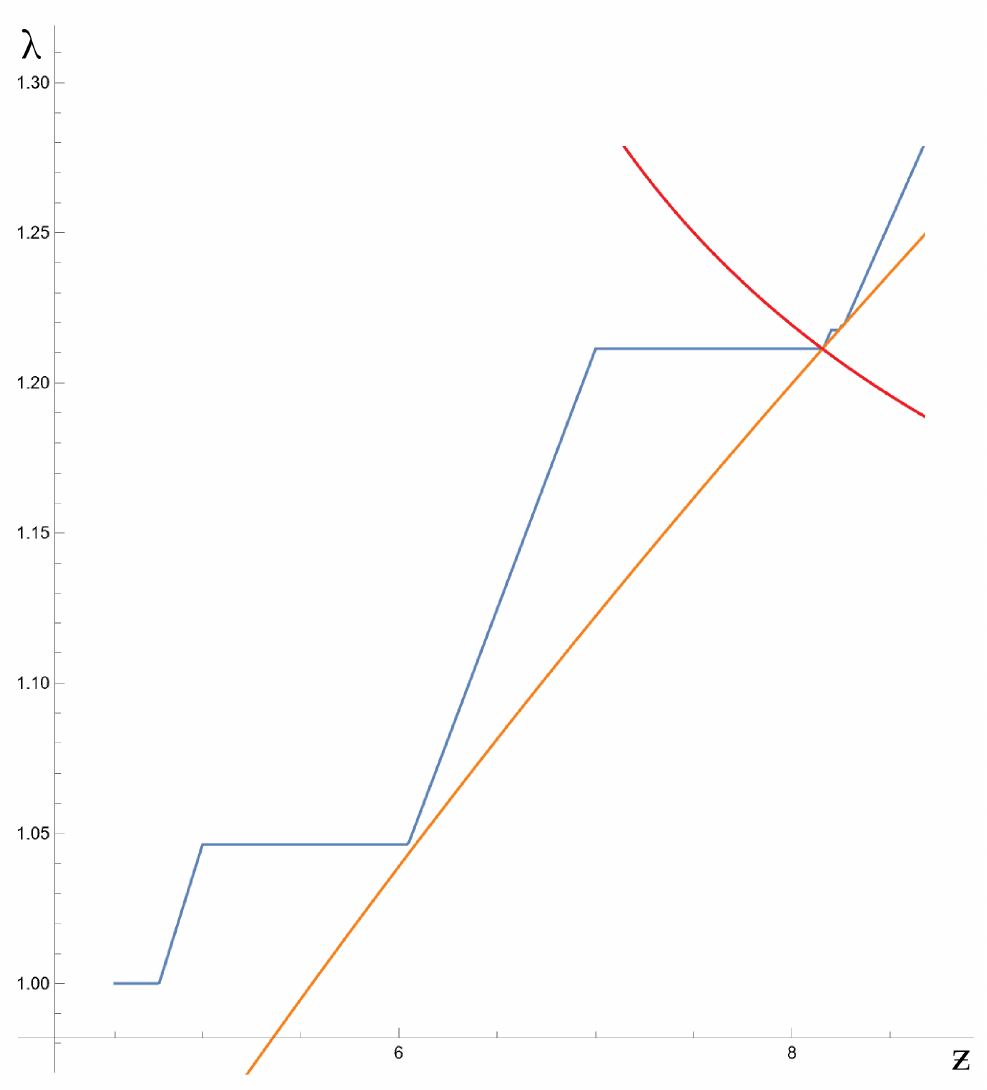}
     }\hfil
    \subfigure[]{
         \centering
         \includegraphics[width=.53\textwidth]{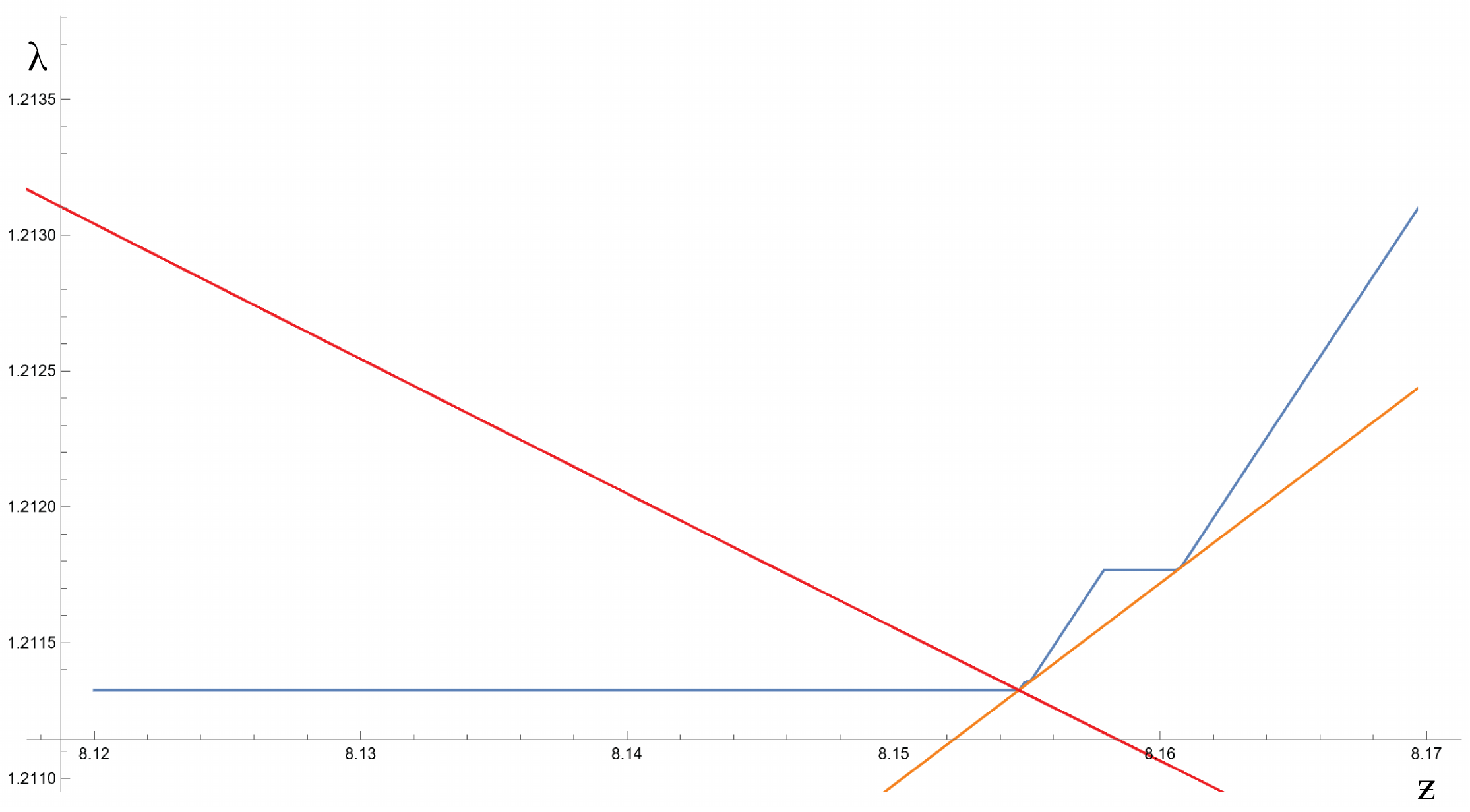}
         }
     \hfil
     \caption{This figure depicts the conjectural infinite staircase $c_\beta$ with $\beta=(24+7\sqrt{3})/13$. In both figures the orange curve is $\vol_\beta(z)$ and $c_\beta$ is in blue. The accumulation point curve $(\acc(\beta),\vol(\beta))$ is in red with $\beta$ varying. Thus the accumulation point of $c_\beta$ ought to occur at the intersection of these three curves. Figure (a) suggests that the accumulation point is precisely where the horizontal obstruction from the class $(3,1,7,1,4)$ intersects the volume curve. 
     In (b), we have zoomed in.}
    \label{fig:sqrt3desc}
\end{figure}

The first two steps of $c_\beta$ are $(5,1,9,1,6)$ and $\bE=(17,6,41,5,22)$. These three classes form what is called in \cite{MMW} a ``compatible triple.'' Thus we expect the same Cantor set structure to arise for $P(1,\beta)$ infinite staircases with $7<\acc(\beta)<9$ as does for the target $H_b$ with $6<\acc_H(b)<8$, see \cite[Thm.~1.1.1]{MMW}.

Moreover, the language of ``blocking classes'' developed in \cite{ICERM} and relying on the accumulation point formula from \cite{cghmp} provides more detail for understanding the results of \cite{integralpolydisks}. In \S\ref{ssec:toric} we defined the notion of a blocked $\beta$-value. Using the lower bound \eqref{eqn:cboundmu} on $c_\beta$ by the obstructions $\mu_{\bE,\beta}$, we say a quasi-perfect Diophantine class $\bE$ \textbf{blocks} $c_\beta$ from having an infinite staircase if $\mu_{\bE,\beta}(\acc(\beta))>\vol_\beta(\acc(\beta))$. We expect that $c_n, n\in\Z_{\geq2}$ are blocked by the perfect classes $\bE(n)$.  This illustrates the power of \cite[Thm.~1.13]{cghmp}: it reaffirms why the classes $\bE(n)$ (which appeared in different notation as the classes $E_n$ in \cite[(1-4)]{integralpolydisks}) are natural key players for the computations of Cristofaro-Gardiner, Frenkel, and Schlenk.

%without it, Cristofaro-Gardiner, Frankel, and Schlenk
%must compute the entire embedding capacity function to verify that there is not an infinite staircase.

\subsection{Brahmagupta moves}\label{ssec:B}

In \cite{symm}, \cite{MMW}, and \cite{usher}, the authors found a symmetry, referred to as a Brahmagupta move by \cite{usher}, that acts on quasi-perfect classes to construct infinitely many different targets that have infinite staircases given one target with an infinite staircase.

It is more natural to define these symmetries via their action on the $z$-variable than the $\beta$-variable in the case of the polydisk (or the $b$-variable in the case of the Hirzebruch surface). Given an infinite sequence of classes $\{\bE_k\}$ centered at $p_k/q_k$ that form the steps of one infinite staircase, the symmetry sends the class $\bE_k$ to $S(\bE_k)$ where $S(\bE_k)$ is centered at $(6p_k-q_k)/q_k.$ On the $z$-coordinate, the symmetry can be expressed as the function $S(z)=6-1/z.$ In the work of \cite{symm}, \cite{MMW}, and \cite{usher}, the authors proved that for the staircases being considered the classes $\{S(\bE_k)\}$ form the steps of a new\footnote{In the special cases where $\beta=1/2$ or $b=1/3$, when the infinite staircase accumulates to $3+2\sqrt{2}$, the symmetry $S$ actually fixes the staircase, but in all other cases studied a new staircase is found.} infinite staircase. Further, by iterating $S$, for each positive integer $i,$  $\{S^i(\bE_k)\}$ is an infinite sequence of classes, which in practice always corresponds to the steps of a new infinite staircase. However, it has not been proven in general that $S$ sends staircases to staircases, and in \cite{symm}, \cite{MMW}, and \cite{usher} the authors required specific estimates about the starting $\{\bE_k\}$ staircases to conclude that for each $i$, the $S^i(\bE_k)$ also form staircases.

%In \cite{usher}, Usher considers a bi-infinite family of polydisks generated by the $L_{n,0}$ by ``Brahmagupta moves.'' These are transformations of the set of quasi-perfect Diophantine classes which are continuous on the centers of the classes. In addition, performing the same Brahmagupta move to every class {\color{red} of Usher's infinite staircases for $P(1,L_{n,0})$} produces a new infinite staircase {\color{red}and thus a correspondence between their accumulation points. Therefore, because} the function $\acc$ is 1-1, we can also think of the Brahmagupta moves as transformations on the parameter $\beta$ of the polydisk.

%The Brahmagupta moves have an analogue (in the sense of the $b\leftrightarrow \beta$ correspondence of Conjecture \ref{conj:Pfractal}) for the $H_b$s. The {\color{cyan} moves on the classes of the $H_b$ staircases} are studied in \cite{symm}; see \cite[Rmk.~2.2.6]{symm} for a detailed discussion of the analogy. The formula is exactly the same in the cases of $P(1,\beta)$ and $H_b$ on the centers of the quasi-perfect Diophantine classes making up the infinite staircase: both send a class centered at $p/q$ to a class centered at $6-q/p$. Thus we express the moves on the $z$-coordinate using the function $S(z)=6-1/z$, while on the $\beta$-coordinate we consider compositions of
% \[
% \beta\mapsto\acc^{-1}\circ S\circ\acc(\beta).
% \]
% The move on the $b$-coordinate is exactly the same, but with $\acc_H$ replacing $\acc$.

For the polydisk, because the function $\acc$ is 1-1, we can also consider the effect of the symmetries on the parameter $\beta$ of the polydisk via
\[
\beta\mapsto\acc^{-1}\circ S\circ\acc(\beta).
\]
The transformation on the $b$-coordinate is exactly the same, but with $\acc_H$ replacing $\acc$ and its domain restricted to account for the fact that $\acc_H$ is 2-to-1 in general.

In \cite{MMW} it was proved that the images of the four-periodic infinite staircase accumulating to $[\{7,5,3,1\}^\infty]$ under the symmetries also have infinite staircases in their ellipsoid embedding functions. Thus we expect that the images of $P(1,(6+5\sqrt{30})/12)$ under the Brahmagupta moves likely also have infinite staircases.
\begin{conjecture}
    The functions $c_{\beta_i}$ have infinite staircases, where
    \[
    \beta_i:=\acc^{-1}\circ S^i\circ\acc\left(\frac{6+5\sqrt{30}}{12}\right).
    \]
\end{conjecture}

\section{Code for exploring ATFs}\label{sec:code}

In \S\ref{ssec:ATFdefs}, we gave a detailed introduction to ATFs. In this section, we will continue the topic to discuss the Python realization of mutations. 

To start, we use \texttt{Decimal} data type for the calculation. As seen in the mutation sequences providing a full filling at the accumulation point (Proposition \ref{prop:fullacc}), the affine lengths of edges can get extremely small after only a few steps of mutation. This goes beyond the limit of any type of traditional floating data type and could lead to errors and breakdowns of the program. With \texttt{Decimal}, however, one can set however many digits needed with exact precision. This helps greatly when looking for the inner corners near the accumulation point. Further, with enough digits, one can compute the continued fractions of the ratios $|OY|/|OX|$ and, after the periodic pattern is clear, reverse engineer the precise quadratic irrational accumulation points. Here is the code for our setup, with 29 digits:

\begin{lstlisting}[language=Python]
    import decimal
    from decimal import Decimal as D
    # number of digits calculated:
    decimal.getcontext().prec = 10000 
    # number of digits printed:
    N = 29 
\end{lstlisting}

We construct the \texttt{node} class to integrate the vertex, the nodal ray at the vertex, and the edge departing clockwise from that vertex. 
\begin{lstlisting}[language=Python]
    class node (object):
        def __init__ (self, vertex, nodal_ray, edge, 
                                        affine_length):
            self.vertex = [D(vertex[0]), D(vertex[1])]
            self.nodal_ray = [D(nodal_ray[0]), D(nodal_ray[1])]
            self.edge = [D(edge[0]), D(edge[1])]
            self.affine_length = D(affine_length)
\end{lstlisting}

The next definition, \texttt{init\_polydisk(b)}, initializes the polydisk $P(1,\beta)$
\begin{lstlisting}[language=Python]
    def init_polydisk (b):
        global n
        global nodes
        n = 4
        nodes = [None] * 4
        nodes[0] = node([0,0], [1,1], [0,1], 1.)
        nodes[1] = node([0,1], [1,-1], [1,0], b)
        nodes[2] = node([b,1], [-1,-1], [0,-1], 1)
        nodes[3] = node([b,0], [-1,1], [-1,0], b)
\end{lstlisting}

The following two functions, \texttt{dist} and \texttt{dot}, are hand-written helper functions to facilitate the usage of $\texttt{Decimal}$. We then compute the mutation matrix $M$.
\begin{lstlisting}[language=Python]
    def dist (x,y):
        # distance between x and y
        return ( (x[0]-y[0])**2 + (x[1]-y[1])**2 ).sqrt()
    def dot (mat, vec):
        # multiplication of 2*2mat and 2*1vec
        return [ mat[0][0]*vec[0]+mat[0][1]*vec[1], 
                 mat[1][0]*vec[0]+mat[1][1]*vec[1] ]
    
    def solve_matrix (v1, v2, w1, w2):
        # solve the matrix M such that M(v1)=v2, M(w1)=w2
        mat = [ [w1[1], -v1[1]],
                [-w1[0], v1[0]] ]
  
        res = [ dot(mat, [v2[0],w2[0]]), 
                dot(mat, [v2[1],w2[1]]) ]
  
        res[0][0] = res[0][0] / (v1[0]*w1[1]-v1[1]*w1[0])
        res[0][1] = res[0][1] / (v1[0]*w1[1]-v1[1]*w1[0])
        res[1][0] = res[1][0] / (v1[0]*w1[1]-v1[1]*w1[0])
        res[1][1] = res[1][1] / (v1[0]*w1[1]-v1[1]*w1[0])

        return res
\end{lstlisting}

In the program we label the vertices clockwise using the numbers $0-3$, starting from the origin as $0$. Then \texttt{intersect\_one(i,j)} solves for the intersection point between the lines of the $i$-th nodal ray and the $j$-th edge. The function will return the intersection point if it lies on the edge segment and $[-1,-1]$ otherwise.
\begin{lstlisting}[language=Python]
    def intersect_one (i,j):
        # solve the intersection between i-th nodal ray 
        # and j-th edge
        global n
        global nodes

        # copy as local variables
        n1 = nodes[i].vertex
        n2 = nodes[j].vertex
        n3 = nodes[(j+1)%n].vertex
        v1 = nodes[i].nodal_ray
        v2 = nodes[j].edge

        # solve for the intersection point
        vec = [ v1[1]*n1[0]-v1[0]*n1[1], 
                    v2[1]*n2[0]-v2[0]*n2[1] ]
        mat = [[ -v2[0], v1[0] ],
               [ -v2[1], v1[1] ]]

        itx = dot(mat, vec)
        itx[0] = itx[0] / (v1[0]*v2[1] - v1[1]*v2[0])
        itx[1] = itx[1] / (v1[0]*v2[1] - v1[1]*v2[0])

        # check if the intersection is on the edge
        if abs(n2[0] - n3[0]) == 0:
            lmbda = (itx[1]-n3[1]) / (n2[1]-n3[1])
        else:
            lmbda = (itx[0]-n3[0]) / (n2[0]-n3[0])
        if (lmbda<0 or lmbda>1):
            return [-1,-1]
    
        return itx
\end{lstlisting}

The next function, \texttt{intersect\_all(x)}, solves for the edge that the $x$-th nodal ray intersects with. This is achieved by finding the intersections of the $x$-th nodal ray with all other edges, throwing away invalid intersections, and keeping the one with the shortest distance to the $x$-th vertex.
\begin{lstlisting}[language=Python]
    def intersect_all (x):
        # solve the intersecting edge for the x-th nodal ray
        global n
        global nodes

        # the variables for the intersecting edge
        min_edge = x
        min_itx = []
        min_dis = math.inf
  
        for i in range(n):
            # i is adjacent to x
            if (i==x or i==(x-1)%n):
            continue

        # the intersection of x-th nodal ray 
        # and i-th edge is invalid
        itx = intersect_one(x,i)
        if (itx == [-1,-1]):
            continue
    
        # maintain the closest intersection
        dis = dist(nodes[x].vertex, itx)
        if (dis < min_dis):
            min_edge = i
            min_itx = itx
            min_dis = dis

        return (min_edge, min_itx)
\end{lstlisting}

With the above foundations, the function \texttt{mutate(x)} calculates the polygon after mutating the $x$-th nodal ray. It has two secondary helper functions \texttt{mutate\_counterclockwise} and \texttt{mutate\_clockwise}, depending on whether the intersecting edge is to the left or right of the the nodal ray. Here we demonstrate the code for the former as the two are extremely similar.
\begin{lstlisting}[language=Python]
    def mutate_counterclockwise (head, tail, itx):
        # mutate with nodal_ray < intersecting edge
        global n
        global nodes 

        mat = solve_matrix( nodes[head].nodal_ray, 
                            nodes[head].nodal_ray, 
                            nodes[head].edge, 
                            nodes[(head-1)%n].edge )

        # construct the new node
        new_length = nodes[tail].affine_length 
                    * dist(itx, 
                    nodes[(tail+1)%n].vertex) 
                    / dist(nodes[tail].vertex, 
                    nodes[(tail+1)%n].vertex)
        new = node(itx, [-nodes[head].nodal_ray[0], 
                        -nodes[head].nodal_ray[1]], 
                        nodes[tail].edge, new_length)
        nodes = np.insert(nodes, tail+1, new)

        # adjust the head and tail node
      nodes[tail].affine_length -= new_length
      nodes[head-1].affine_length += nodes[head].affine_length
      nodes = np.delete(nodes, head)

        # update remaining nodes
        for i in range(head, tail):
            pre = nodes[(i-1)%n]
            nodes[i].vertex[0] = pre.vertex[0] 
                                + pre.affine_length*pre.edge[0]
            nodes[i].vertex[1] = pre.vertex[1] 
                                + pre.affine_length*pre.edge[1]
            nodes[i].nodal_ray = dot(mat, nodes[i].nodal_ray)
            nodes[i].edge = dot(mat, nodes[i].edge)
    
    def mutate (x):
        # mutate once by x-th nodal_ray
        global n
        global nodes

        # y is the intersecting edge
        # itx is the intersection point
        (y, itx) = intersect_all(x)

        if (x<y):
            mutate_counterclockwise(x,y,itx)
            return y
        else:
            mutate_clockwise(y,x,itx)
            return y+1
\end{lstlisting}

Finally, we have two interface functions \texttt{plot\_nodes} and \texttt{print\_embd} that output direct information for use. The first, \texttt{plot\_nodes}, plots the polygon with respect to edge length ratio; \texttt{print\_embd} prints the staircase coordinate $(z, \lambda)$ such that $E(1,z) \sembeds P(\lambda,\lambda b)$. The embedding is constructed by fitting the right triangle $\Delta OXY$ into the polygon. Below is an example that gets the $v^2y^2$ example above.

\begin{lstlisting}[language=Python]
    b = (6 + 5 * D(30).sqrt()) / 12
    init_polydisk(b)
    mutate(2)
    mutate(2)
    mutate(1)
    mutate(1)
    plot_nodes()
    print_embd()
\end{lstlisting}

It should be easy to generalize the initialization functions for other types of polygons beyond rectangles, such as triangles and trapezoids. For a complete file of the code, see the \href{https://github.com/Ginko2501/Almost-Toric-Fibration/tree/main}{Github Repository}.


\begin{thebibliography}{CGHMP}

\bibitem[BHM]{ICERM} M. Bertozzi, T. S. Holm, E. Maw, D. McDuff, G. T. Mwakyoma, A. R. Pires, and M. Weiler, ``Infinite staircases for Hirzebruch surfaces.'' \textit{Research Directions in Symplectic and Contact Geometry and Topology}, ed. Bahar Acu et. al., Springer AWMS Vol. 27, pp. 47-157 (2021).

\bibitem[CV]{CV} R.\ Casals and R.\ Vianna, ``Full Ellipsoid Embeddings and Toric Mutations,'' \textit{Select. Math.}, 28(61), 2022.

\bibitem[CG1]{cg} D. Cristofaro-Gardiner, ``Symplectic embeddings from concave toric domains into convex ones,'' {\it J. Diff. Geom.}, 112(2):199-232, 2019.

\bibitem[CGFS]{integralpolydisks} D. Cristofaro-Gardiner, D. Frenkel, and F. Schlenk, ``Symplectic embeddings of four-dimensional ellipsoids into integral polydiscs,'' \textit{Alg. \& Geo. Topol.}, 17 (2):1189–1260, 2017.

\bibitem[CGHMP]{cghmp} D. Cristofaro-Gardiner, T. Holm, A. Mandini, and A. R. Pires,  ``On infinite staircases in toric symplectic four-manifolds,'' \textit{arXiv:2004.13062}.

\bibitem[CGHR]{cghr} D. Cristofaro-Gardiner, M. Hutchings, and V.G.B. Ramos, ``The asymptotics of ECH capacities,'' \textit{Invent. Math.} 199 (2015), 187-214.

\bibitem[E]{evans}
J. Evans, \textit{Lectures on Lagrangian Torus Fibrations.} London Mathematical Society Student Texts Vol. 105, Cambridge University Press, Cambridge, 2023.

\bibitem[FM]{frenkelmuller} D. Frenkel, D. M\"uller, ``Symplectic embeddings of $4$-dimensional ellipsoids into cubes,'' {\it J. Symp. Geom.} 13(4):765-847, 2015.

\bibitem[Ha]{ha} A. Hatcher, {\it Topology of Numbers.} American Mathematical Society, Providence, RI, 2022. \url{https://pi.math.cornell.edu/~hatcher/TN/TNpage.html}

\bibitem[H1]{Hut} M. Hutchings, ``Quantitative embedded contact homology,'' {\it J. Diff. Geom.} 88(2):231--266, 2011.

\bibitem[H2]{Hsurvey} M. Hutchings, ``Recent progress on symplectic embedding problems in four dimensions.'' \textit{Proc. Natl. Acad. Sci. USA}
108 (2011), no. 20, 8093–8099.

\bibitem[LS]{leung-symington}
N.C. Leung and M. Symington, ``Almost toric symplectic four-manifolds.'' {\it J. Symp. Geom.} {\bf 8} (2010), no. 2, 143–187.

%\bibitem[LiLi]{li-li} B-H  Li and T-J Li, ``Symplectic genus, minimal genus and diffeomorphisms,'' {\it Asian J.\ Math.}\  {\bf 6} (2002)123--144.

%\bibitem[LiLiu]{li-liu} T-J Li, A-K Liu, ``Uniqueness of symplectic canonical class, surface cone and symplectic cone of $4$-manifolds with $b^+=1$,'' {\it J.\ Differential Geom}.\  {\bf 58} (2001) 331--370.

\bibitem[M1]{M1} N. Magill, ``Unobstructed embeddings in Hirzebruch surfaces,'' to appear in \textit{J. Symp. Geom.} \textit{arXiv: 2204.12460}.

\bibitem[M2]{M2} N. Magill, ``Almost toric fibrations in 2-fold blow ups of $\C P^2$,'' in preparation.

\bibitem[MM]{symm} N. Magill and D. McDuff, ``Staircase symmetries in Hirzebruch surfaces,'' to appear in \textit{Alg \& Geo. Topol.}, \textit{arXiv:2010.08567}.

\bibitem[MMW]{MMW} N. Magill, D. McDuff, and M. Weiler, ``Staircase patterns in Hirzebruch surfaces,'' \textit{arXiv:2203.06453}.

\bibitem[MPW]{MPW} N. Magill, A. R. Pires, and M. Weiler, ``A complete classification of infinite staircases for Hirzebruch surfaces,'' arXiv:.

\bibitem[Mc1]{Mc0}  D. McDuff,   ``Symplectic embeddings of 4-dimensional ellipsoids,''
{\it J. Topol.} {\bf 2} (2009), no. 1, 1--22. 
{\it Corrigendum}: {\it J. Topol.} {\bf 8} (2015) no 4, 1119--1122.


\bibitem[Mc2]{m2} D. McDuff, ``The Hofer conjecture on embedding symplectic ellipsoids'', {\it J. Differential Geom.} 88(3):519--532, 2011.

\bibitem[MP]{mp} D. McDuff and L. Polterovich, ``Symplectic packings and algebraic geometry,'' \textit{Inventiones Mathematicae} 115, (1994) 405–29.

\bibitem[McSc]{ball} D. McDuff and F. Schlenk, ``The embedding capacity of 4-dimensional symplectic ellipsoids,'' {\it Ann. Math}  (2) 175 (2012), no. 3, 1191--1282.

\bibitem[McSal]{McSal}  D. McDuff and D. Salamon, {\it Introduction to Symplectic Topology}, 3rd edition, Oxford University Press, 2017.

\bibitem[S]{symington}
M.\ Symington, ``Four dimensions from two'' {\it Topology and geometry of manifolds (Athens, GA, 2001)}, Proc. Sympos. Pure Math. 71 153-–208, Amer. Math. Soc., Providence, RI, 2003.


\bibitem[U1]{usher} M. Usher, ``Infinite staircases in the symplectic embedding problem for four-dimensional ellipsoids into polydisks'', {\it Algebr. Geom. Topol.}, Volume 19, Number 4 (2019), 1935-2022.

\bibitem[U2]{usherletter} M. Usher, private communication, 2022.




\end{thebibliography}
\end{document}